\documentclass[aap]{imsart}

\RequirePackage{amsthm, amsmath, amsfonts, amssymb, bbm, color, enumerate, graphicx, mathtools, tikz, relsize}
\usepackage[T1]{fontenc}
\usepackage{eulervm}
\RequirePackage[numbers]{natbib}
\RequirePackage[colorlinks,citecolor=blue,urlcolor=blue]{hyperref}
\numberwithin{equation}{section}

\startlocaldefs
\newcommand{\deltap}{\delta{\scriptstyle{+}}}

\renewcommand{\qq}{\mathbf{q}}
\newcommand{\al}{\alpha}

\newcommand{\tQ}{\tilde{Q}}

\newcommand{\bbZ}{\mathbb{Z}}
\newcommand{\bbZp}{\bbZ_+}
	
\newcommand{\clm}{\mathcal{M}}
\newcommand{\clc}{\mathcal{C}}
\newcommand{\cld}{\mathcal{D}}

\newcommand{\clr}{\mathcal{R}}
\newcommand{\cls}{\mathcal{S}}
\newcommand{\eps}{\epsilon}
\newcommand{\lan}{\langle}
\newcommand{\ran}{\rangle}
\newcommand{\la}{\lambda}
\newcommand{\RR}{\mathbb{R}}
\newcommand{\RRp}{\mathbb{R}_+}
\newcommand{\NN}{\mathbb{N}}
\newcommand{\one}{\mathbf{1}}

\newcommand{\aop}{a_0(p)}
\newcommand{\ao}{a_0}

\newcommand{\mop}{m_0(p)}
\newcommand{\mo}{m_0}
\newcommand{\sigmap}{\sigma(p)} 
\newcommand{\cop}{C_0(p)} 
\newcommand{\co}{C_0} 
\newcommand{\hop}{H_0(p)} 
\newcommand{\ho}{H_0} 

\newcommand{\astar}{a^*}

\newtheorem{theorem}{Theorem}

\newtheorem{lemma}[theorem]{Lemma}
\newtheorem{proposition}[theorem]{Proposition}
\newtheorem{remark}{Remark}

\let\plainqed\qedsymbol
\newcommand{\claimqed}{$\lrcorner$}

\usepackage{fancyhdr}
 
\usepackage{natbib}

\usepackage{amssymb}
\usepackage{amsmath}
\usepackage{amsthm}

\usepackage{soul}

\newcommand{\Expect}[1]{\operatorname{\mathbb{E}}\left[#1\right]}

\newcommand{{\LPC}}{\textbf{LPC}}

\newcommand{\Z}{\mathcal{Z}}
\newcommand{\tZ}{\widetilde{\mathcal{Z}}}
\renewcommand{\tQ}{\widetilde{\mathcal{Q}}}

\allowdisplaybreaks

\renewcommand{\eps}{\varepsilon}

\endlocaldefs

\begin{document}

\begin{frontmatter}
\title{Heavy Traffic Scaling Limits for shortest remaining processing time queues with heavy tailed processing time distributions}
\runtitle{SRPT Queues with Heavy Tailed Processing Times}

\begin{aug}
\author[A]{\fnms{Sayan} \snm{Banerjee}\ead[label=e1]{sayan@email.unc.edu}},
\author[A]{\fnms{Amarjit} \snm{Budhiraja}\ead[label=e2,mark]{amarjit@unc.edu}}
\and
\author[B]{\fnms{Amber L.} \snm{Puha}\ead[label=e3,mark]{apuha@csusm.edu}}
\address[A]{Department of Statistics and Operations Research,
University of North Carolina at Chapel Hill, Hanes Hall, 318 E Cameron Ave \#3260, Chapel Hill, NC 27599
\printead{e1,e2}}

\address[B]{Department of Mathematics,
California State University San Marcos, 333 S.\ Twin Oaks Valley Road, San Marcos, CA 92096-0001
\printead{e3}}
\end{aug}

\begin{abstract}
We study a single server queue operating under the shortest remaining processing time (SRPT) scheduling policy; that is, the server preemptively serves the job with the shortest remaining processing time first.  Since one  needs to keep track of the remaining
processing times of all jobs in the system in order to describe the evolution, a natural state descriptor for an SRPT queue is  a measure valued process  in which the state of the system at a given
time is the finite nonnegative Borel measure on the nonnegative real line that puts a unit atom at the remaining processing time of each
job in system. In this work we are interested in studying the asymptotic behavior of the suitably scaled measure valued state descriptors for a sequence of SRPT queuing systems. Gromoll, Kruk, and Puha (2011) have studied this problem under diffusive scaling (time is scaled by $r^2$ and the mass of the measure normalized by $r$, where $r$ is a scaling parameter approaching infinity). In the setting where the processing time distributions have {\em bounded support}, under suitable conditions, they show that the measure valued state descriptors converge in distribution to the process that at any given time is a single atom located at the right edge of the support of the processing time distribution with the size of the atom fluctuating randomly in time.
In the setting where the processing time distributions have {\em unbounded support}, under suitable conditions, they show that the diffusion scaled measure valued state descriptors converge in distribution to the process that is identically zero.  In Puha (2015) for the setting where the processing time distributions have {\it unbounded support and light tails}, a nonstandard scaling of the queue length process is shown to give rise to a form of state space collapse that results in a nonzero limit.

In the current work we consider the case where processing time distributions have finite second moments and regularly varying tails. 
Results of Puha (2015) suggest that the right scaling  for the measure valued process is governed  by a parameter $c^r$
that is given as a certain inverse function related to the 
tails of the first moment of the processing time distribution. Using this parameter we consider a novel scaling for the measure valued process in which
the time is scaled by a factor of $r^2$, the mass is scaled by the factor $c^r/r$ and the space (representing the remaining processing times)
is scaled by the factor $1/c^r$. We show that the scaled measure valued process converges in distribution  (in the space of paths of measures). 
 In a sharp contrast to results for bounded support and light tailed service time distributions, this time there is no state space collapse and the limiting measures are not concentrated on a single atom. Nevertheless, the description of the limit is  simple and given explicitly in terms of a certain $\mathbb{R}_+$ valued random field  which is determined from a single Brownian motion. Along the way we  establish convergence of suitably scaled workload and queue length processes. We also show that as  the tail of the distribution of job processing times becomes lighter in an
appropriate fashion, the difference between the limiting queue length process and the limiting workload process converges to zero, thereby approaching the behavior of state space collapse.
\end{abstract}

\begin{keyword}[class=MSC2020]
\kwd[Primary ]{60K25, 60F17}
\kwd{Queueing theory}
\kwd{Functional limit theorems}
\kwd[; secondary ]{60G57, 60G60, 68M20}
\end{keyword}

\begin{keyword}
\kwd{Heavy traffic}
\kwd{heavy tails}
\kwd{queueing}
\kwd{shortest remaining processing time}
\kwd{regular variation}
\kwd{measure valued processes}
\kwd{functional central limit theorem}
\kwd{random field}
\kwd{state space collapse}
\kwd{synchronization phenomenon}
\kwd{intertwined SRPT queues}
\kwd{derivative of the Skorohod map}
\end{keyword}

\end{frontmatter}

\section{Introduction}
We study a single-server, single-class queue operating under the shortest remaining processing time (SRPT) service discipline.
Jobs arrive to the queue according to a renewal process.  Each such job has associated with it a processing time, which
is a random variable that represents the amount of time that the server must spend working on this job to complete its service.
The processing times are assumed to be independent and identically distributed.  In an SRPT queue, jobs are served
one at a time such that the job with the shortest remaining processing time is served first.  In particular, upon completing the
service of a given job, the server then takes into service the job in system with the shortest remaining processing time.  This
is done with preemption so that when a job arrives with a processing time that is smaller than the remaining processing time of
the job in service, the server places the job in service on hold and begins serving the job that just arrived.  Processing is done
in a nonidling fashion so that the server idles only when the system is empty.  While SRPT has a large memory requirement
for implementation since remaining processing times of all jobs in the queue
must be known, it  has desirable
optimality properties. In particular, it is the service discipline that minimizes queue length (see Schrage \cite{S68} and Smith \cite{S78}). Therefore, SRPT can serve as a performance
benchmark (e.g.\ Chen and Dong \cite{CD20}). The survey paper \cite{S93} by Schreiber provides nice discussion of early works concerning SRPT.

One challenge associated with a detailed  analysis of SRPT is that, due to the need to keep track of the remaining
processing times of all jobs in the system, the state descriptor for an SRPT queue is infinite dimensional, even for exponentially distributed processing times. 
In order to describe the state of the system, Down, Gromoll, and Puha \cite{DGP1109, DGP909} introduce a measure valued process  in which the state of the system at a given
time is the finite nonnegative Borel measure on the nonnegative real line that puts a unit atom at the remaining processing time of each
job in system.  Under natural modeling assumptions and asymptotic conditions, they prove a fluid limit
theorem (a functional law of large numbers)  for this measure valued state descriptor.  This yields a fluid analog for the response time of jobs in system at time zero as a function
of their remaining processing times at time zero. In the critically loaded case, the rate at which this fluid analog for the response time grows as time tends to infinity is seen to be dependent on the tail behavior of the processing time distribution. 
These results are consistent with  the growth rates obtained in \cite{LWZ11} for steady state mean response times as the traffic intensity increases to one.
In follow on work, Kruk \cite{K16} proves a fluid limit theorem for multiclass SRPT queues that includes convergence of the response times to the expression studied in \cite{DGP1109}, which justifies it as an approximation. Atar, Biswas, Kaspi and Ramanan \cite{ABKR18} develop more general fluid limits for SRPT and other priority queues with time varying arrivals and
service rates.

In this work, we consider a sequence of SRPT queues indexed by a scaling parameter $r$ approaching infinity.
We are interested in studying the asymptotic behavior of the measure valued state descriptors for this sequence of SRPT queuing systems under diffusion
and other suitable scalings. This 
captures the performance deviation of a critically loaded  SRPT queue
from the fluid limit by describing the fluctuations. 
Gromoll, Kruk, and Puha \cite{GKP11} provide a first step in this direction by establishing a diffusion limit theorem (a functional central limit theorem),
for the sequence of measure valued processes. In \cite{GKP11} for the case where the processing time distributions have {\em bounded support}, it is shown that, with standard diffusive scaling (time is scaled by $r^2$ and the mass of the measure normalized by $r$), under natural modeling
assumptions and mild asymptotic and standard heavy traffic conditions, the mass of the (scaled) measure valued state descriptors in the limit concentrates on a single atom located at the right
edge of the support of the processing time distribution with the size of the atom fluctuating randomly in time.  This is similar in spirit to results for static
priority queues where only the queue associated with the lowest priority class is nonempty in the diffusion limit (see \cite{BD01,W71}).  The result for the bounded support case suggests
that for processing time distributions with unbounded support, with standard diffusive scaling, one
should obtain the trivial limit of the zero process for the scaled measure valued process.  This is indeed true under suitable conditions as is also shown
 in \cite{GKP11}.  These results are rederived by Kruk \cite{K19} via an alternative argument that leverages diffusion
limits for earliest deadline first queues obtained in Kruk \cite{K07}.  Although the measure valued processes under the standard diffusion scaling converge to the zero process, the workload  under the diffusive scaling, which is given as the first moment of the state descriptor measure, does not converge to the zero process. Indeed, since SRPT is a nonidling service discipline, the diffusion limit for the workload process
(which is independent of the scheduling policy) corresponds to a semi-martingale reflected Brownian motion (SRBM)  \cite{IW70}.  Heuristically the above results say that, for processing time distributions with unbounded support, SRPT minimizes the queue length so efficiently that, in the diffusion limit, the queue length process is  of a smaller order than the workload process.

This raises the important problem of quantifying the precise difference in orders of the queue length and workload processes.  In  \cite{P15}, Puha studies the case where the the processing time distributions have light tails (rapidly varying with index $-\infty$, e.g. an exponential distribution) and identifies the key quantity that determines the correct scaling for the queue length process. This quantity, denoted as $c^r$
and defined in equation \eqref{def:cr} here, is given in terms of a certain inverse function related to the 
tails of the first moment of the processing time distribution. Using the scaling factor $c^r$, \cite{P15} establishes a state space collapse result that specifies conditions under which 
\begin{equation}\label{eq:lighttail}
	(c^r \hat Q^r, \hat W^r) \mbox{ converges in distribution to }  (W^{\infty}, W^{\infty}), \qquad\hbox{as }r\to\infty,
	\end{equation}
where $\hat Q^r$ and $\hat W^r$ are the queue length and workload processes, respectively, of the $r$-th system with standard diffusive scaling and $W^{\infty}$ is a certain SRBM on $\mathbb{R}_+$.  Although \cite{P15} does not consider the convergence of the measure valued state descriptor, the result in \eqref{eq:lighttail} suggests that with an appropriate scaling, this measure valued process converges in distribution to a process of  Dirac measures at one (with random weights); see Remark \ref{heavylightcomp} for additional comments on this point.

In this work, we study the setting where the processing times have finite second moments and regularly varying tails (see \eqref{eq:regvar}).
{\color{black}Such heavy tailed processing time distributions arise naturally in various application domains, e.g., file transfer models and cloud computing \cite{CSN09, L12},
which motivates us to consider the performance of SRPT in this setting in more detail.  For this, we} study the asymptotic behavior of the full measure valued state descriptor under an appropriate scaling. As in \cite{P15} the quantity $c^r$ is once more central to identifying the correct scaling. The scaled measure valued process, denoted as $\tZ^r(\cdot)$, is defined using three types of scaling: the time is scaled by a factor of $r^2$, the mass is scaled by the factor $c^r/r$ and the space (representing the remaining processing times) is scaled by the factor $1/c^r$; see \eqref{ztildef} for a precise definition. One of our main results (Theorem \ref{meascon}) gives convergence of $\tZ^r(\cdot)$ in distribution, in 
$\mathcal{D}([0,\infty): \clm_F)$ (the space of right continuous functions with left limits equipped with the usual Skorohod topology, where $\clm_F$ is the space of finite nonnegative measures on $\mathbb{R}_+$ with the  topology of weak convergence), to a limit measure valued process $\tZ(\cdot)$. In a sharp contrast to results for bounded support and light tailed service time distributions, this time there is no state space collapse and the limiting measures are not concentrated on a single atom. Nevertheless, the description of the limit is  simple and given explicitly in terms of a certain $\mathbb{R}_+$ valued random field $\{W_a(t), t \in [0, \infty), a \in [0, \infty]\}$ which is determined from a single Brownian motion; see \eqref{eq:GX} -- \eqref{eq:xinfinitydefn}. Roughly speaking, $W_a(\cdot)$ can be interpreted as the asymptotic (diffusion scaled) workload process associated with jobs in the system with remaining processing times at most $ac^r$.
In terms of $\{W_a(\cdot), a\in(0,\infty)\}$, the limiting measure valued process $\tZ(\cdot)$ is characterized as follows: for $t\in[0,\infty)$,
$\tZ(t)(\{0\})=0$, $\tZ(t)([0,\infty))=\int_{[0,\infty)}\frac{1}{x^2}W_x(t)dx$ and
$$
\tZ(t)[a,b] : = \int_a^b\frac{1}{x^2}W_x(t)dx + \frac{W_b(t)}{b} - \frac{W_a(t)}{a}, \quad 0 < a < b < \infty.
$$
Along the way we also establish convergence of suitably scaled workload and queue length processes by proving in Theorem \ref{workfn} that, as $r\to\infty$,
$$(c^r\hat Q^r(\cdot), \hat W^r(\cdot)) \mbox{ converges in distribution to }
\left(\int_0^{\infty} \frac{1}{x^2} W_x(\cdot) dx, W_{\infty}(\cdot)\right)$$
in $\mathcal{D}([0,\infty): \RR_+^2)$, where $\hat Q^r$  and $\hat W^r$ are the queue length process and workload process, respectively, of the $r$-th system with the standard diffusive scaling. 

Results of \cite{P15} and Theorems \ref{workfn} and \ref{meascon}  in the current paper suggest that the phenomenon of state space collapse is closely related to the tail behavior of the service time distributions. In Theorem \ref{collapse} we make this heuristic precise by establishing that if the tail of the distribution of job processing times becomes lighter in an
appropriate fashion, the difference between the limiting queue length process and the limiting workload process converges to zero, thereby approaching the behavior of state space collapse exhibited in \cite{P15}
for light tailed processing time distributions. In Theorem \ref{tailtZ}, we prove another type of
`asymptotic state space collapse' which roughly says that, asymptotically,  the cumulative (scaled) workload due to jobs with remaining processing time
more than $ac^r$ (for large $a$) can be obtained by multiplying the number of such jobs present in the system with the expected value of a (full) processing time conditioned to be more than $a$.

The results of this work give information on response times of jobs with a given remaining processing time
in SRPT queues under heavy traffic.
Understanding the behavior of these response times is of interest as they quantify the `unfair' treatment of 
jobs with large processing times under the SRPT discipline
 \cite{BCM98,SG98,S92,T92}. 
 For Poisson arrivals, steady state mean response times have been studied by Bansal and
Harchol-Balter \cite{BHB01} and Lin, Wierman and Zwart \cite{LWZ11}.  In \cite{BHB01}, the 
steady state mean response and slowdown times are studied, with a focus on heavy tailed processing time distributions,
as are characteristic of empirical workloads.  In particular, \cite{BHB01} shows that the degree of unfairness as compared with processor sharing,
a computer time sharing algorithm widely regarded as fair, is relatively small (see also Wierman and Harchol-Balter \cite{WHB03}
for a broader discussion of fairness). Related to this, results of \cite{DGP1109, DGP909} show that fluid analogs of response
times in SRPT queues are sublinear for very heavy tailed processing distribution, which is a performance improvement over processor sharing.
In the related work \cite{LWZ11}, expressions obtained in Perera \cite{P93} and
Schassberger \cite{S90} are used to establish growth rates for the steady state mean
response times as the traffic intensity increases to one (critical loading or heavy traffic).  The rates that they obtain depend
on the tail behavior of the processing time distribution.  For instance, they grow exponentially for exponential
processing times and polynomially for heavy tailed processing times.  In view of the above results on dependence of key performance metrics for SRPT queues on the tail properties of processing time distributions it is of significant interest to understand the precise relationships between these tail properties and scaling limits of SRPT queues in heavy traffic. The current work contributes toward this goal.

\subsection{Methodology}\label{ss:methods}

We now make some comments on the proof of one of our key results, namely Theorem \ref{workfn}. Central to our analysis are certain truncated workload processes $\{W^r_a(t)\}_{t\ge 0}$, $a \in [0,\infty]$, where
$W_a^r(t)$ gives the amount of work (normalized by $r$) associated with jobs with remaining processing time at most $ac^r$ at time $r^2t$ in the $r$-th system. We show in Theorem \ref{workless} that the joint distribution of $W_{a_1}^r, \ldots, W_{a_k}^r$ for finitely many threshold levels $0 \le a_1<\cdots<a_k\le \infty$
converges to the joint distribution of $W_{a_1}, \ldots, W_{a_k}$ where $\{W_a(t)\}_{t\ge 0}$, $a \in [0,\infty]$, is a random field driven by a {\em single} Brownian motion. 
This novel synchronization phenomenon is a key ingredient in our proofs.
It turns out that the convergence of the full measure valued state descriptor $\tZ^r$ can be 
analyzed through the asymptotic properties of these truncated workload processes. This can be heuristically seen from an elementary integration by parts lemma (Lemma \ref{parts})
that expresses the integral of any $C^1$ function, supported on a compact interval of $(0, \infty)$, with respect to the random measure $\tilde{\mathcal{Z}}^r(\cdot)$ in terms of the rescaled, truncated workload processes. This lemma is independent of the scheduling policy and is potentially useful for analyzing other types of  policies for which one has good control over the associated truncated workload processes. 
Using this lemma together with
Theorem \ref{workless} (which characterizes the limits of these truncated workload processes), along with appropriate tightness arguments, we then establish weak convergence of $
Z_f^r(\cdot) := \langle f, \tZ^r(\cdot)\rangle
$ for 
  piecewise $C^1$ functions $f$ supported on a compact interval of $(0, \infty)$ (Theorem \ref{convplus}). The result is then extended to $f$ having support which is bounded below by a positive number $\delta$ but possibly unbounded above (Lemma \ref{Minfty}). Rest of the work is in sending $\delta \rightarrow 0$. 
  This work, which is done in Section \ref{sec:senddtoz}, is technically the most demanding part of the proof as is suggested by the possible singular behavior of the integrand in \eqref{eq:lQ} near $x=0$. The arguments are based on path decompositions of rescaled, truncated workload processes and their limiting versions into excursions and careful analysis of these excursions using martingale arguments; see additional comments at the beginning of Section \ref{sec:senddtoz}. This is done in Lemmas \ref{fifo}-\ref{flowsup}, which finally lead to the proof of Theorem \ref{workfn}. 
As ingredients in the proofs, we also devise some couplings on SRPT systems started from different initial conditions (for example, the `intertwined SRPT queueing systems' analyzed in Subsection \ref{intsec}), which may be of independent interest.

{\color{black}While the idea for the scaling involving $c^r$ is inspired by the prior work \cite{P15}, which considers lighter tailed processing time distributions, the proofs here are not variants or extensions of those in \cite{P15} .  Indeed, in \cite{P15}, the remaining processing times are shown to asymptotically concentrate around the spatial boosting factors $c^r$ as $r$ tends to infinity.  This is not the case for heavier tailed processing time distribution.  Instead, the remaining processing times spread out in a wider window containing $c^r$ and the concentration arguments in \cite{P15} no longer hold.  To address this, we take a different approach by rescaling the measure valued state descriptor such that mass that would otherwise shift toward infinity in a rather spread out fashion around $c^r$ is brought back into a relevant window that spreads out around one. The asymptotic analysis of this rescaled measure-valued process requires an entirely different machinery and approach from the one used in \cite{P15} as was outlined in the previous paragraph.}

{\color{black}We believe our techniques can be extended to SRPT systems with processing time distributions that \emph{depend on $r$}, provided these distributions (indexed by $r$) satisfy certain uniformity conditions required by our techniques. More general $r$-dependence will require significant extensions of our methods and is left for future work.}

\subsection{Organization}\label{sec:org}
The rest of the article is organized as follows. In Section \ref{sec:mathfram}, we rigorously define the sequence of SRPT systems, the heavy traffic conditions,
the associated scaling and assumptions on the initial conditions. In Section \ref{s:main}, we state our main results. Section \ref{prelim} summarizes
some properties of Skorohod maps, regularly varying functions and the functional central limit theorems and tightness criteria used crucially in the
proofs. Section \ref{proofs} is dedicated to the proofs of our main results.

\subsection{Notation}
\label{sec:notat}
The following notation will be used. Let $\NN$ denote the set of positive integers, $\bbZ$ denote the set of integers, $\bbZp$ denote the set of nonnegative integers,
$\RR$ set of real numbers and $\RRp$ the set of nonnegative real numbers. For $a,b\in\RR$, $a\wedge b$ and $a\vee b$ respectively denote the minimum and maximum of the set $\{a,b\}$.
For a Polish space $S$ and $T\in (0,\infty)$, we denote by $\mathcal{D}([0,T]:S)$ (resp. $\mathcal{D}([0,\infty):S)$) the space of functions that are right continuous and have
finite left limits (RCLL)  from $[0,T]$ (resp. $[0,\infty)$) to $S$, equipped with the usual Skorohod topology. Also, denote by $\mathcal{C}([0,T]:S)$ (resp. $\mathcal{C}([0,\infty):S)$) the space
of continuous functions  from $[0,T]$ (resp. $[0,\infty)$) to $S$, equipped with uniform (resp. local uniform) topology. Denote by $\clm_F$  the space of finite nonnegative
Borel measures on $\RRp$ equipped with the topology of weak convergence.  For $\mu\in \clm_F$ and a Borel measurable function $f$ that is integrable with respect to
$\mu$ or nonnegative, we write $\lan f, \mu \ran = \int f d\mu$, which takes the value infinity if $f$ is nonnegative and nonintegrable.
Note that for $\{\mu_n\}_{n\in\NN}\subset\clm_F$ and $\mu\in\clm_F$, as $n\to\infty$, $\mu_n \to \mu$ in $\clm_F$ if and only if $\lan f, \mu_n \ran \to \lan f, \mu \ran$ for
every real valued, bounded, continuous function $f$ on $\RRp$. The topology of weak convergence can be metrized so that $\clm_F$ and hence $\mathcal{D}([0,T]:\clm_F)$
are Polish spaces.  For a Borel subset $A\subseteq\RRp$, $\one_A$ denotes the indicator of set $A$; that is, $\one_A(x)=1$ if $x\in A$ and $\one_A(x)=0$ if $x\not\in A$.
In addition, $\one$ is used as a shorthand notation for $\one_{\RRp}$.  For $x \in \RRp$, $\delta_x$ is the Dirac measure at $x$ that puts a unit atom at $x$ and
$\delta^+_x := \delta_x \one_{\{x>0\}}$ is the measure in $\clm_F$ that equals $\delta_x$ if $x>0$ and is the zero measure otherwise.
For a real valued, bounded function $f$ on $S$, we define $\|f\|_{\infty} := \sup_{x\in S}|f(x)|$.  For $a\in\RRp$, a real valued function $f$ is said to be $C^1$ on $[a,\infty)$ if
it is defined on an open neighborhood of $[a,\infty)$ in $\RRp$ and is continuously differentiable on this neighborhood. For $S$ valued random variables $X_n$, $n\in\NN$, and $X$,
we denote by $X_n \xrightarrow{d} X$ (resp. $X_n \xrightarrow{P} X$) the convergence in distribution (resp. probability) of $X_n$ to $X$ as $n\to\infty$.
For $f \in  \mathcal{D}([0,\infty) : \RR^d)$,  $0 \le s \le t \le \infty$ and $A>0$, we will write $|f(t\#) - f(s\#)| <  A$ to denote that all of
the following inequalities hold: $\left|f(t) - f(s)\right| <  A$, $\left|f(t) - f(s-)\right| <  A$, $\left|f(t-) - f(s)\right| <  A$, $\left|f(t-) - f(s-)\right| <  A$.

\section{Mathematical framework}
\label{sec:mathfram}
\subsection{The sequence of SRPT queues and state descriptor}\label{ss:sd}
We consider a sequence of SRPT queues indexed by $\clr$, a sequence taking values in $(1, \infty)$ tending to infinity.
For each $r\in\clr$, let $\{\breve{v}_l^r, l \in \NN\}$ be a sequence of strictly positive random variables and let $\qq^r$
be a nonnegative integer valued random variable such that $\sum_{l=1}^{\qq^r}\breve{v}_l^r<\infty$ almost surely (with the convention that this sum is zero if $\qq^r$ is zero).
At time zero, there are $\qq^r$ jobs in the $r$-th system with remaining processing times $\breve{v}_l^r$, $l=1, \ldots, \qq^r$.
For $l=1, \ldots, \qq^r$, we refer to the job in system at time zero associated with $\breve{v}_l^r$ as initial job $l$. 
Conditions on $\qq^r$ and  $\{\breve{v}_l^r\}$ will be specified in Section \ref{initconass}.

Jobs arrive to the $r$-th system according to a delayed renewal process $E^r(\cdot)$ with positive, finite rate $\lambda^{r}$ and
finite, positive initial delay. Let $T^r$ (resp.\ $T_1^r$) denote a random variable having the distribution of a typical inter-arrival time (resp.\ the initial delay) in the $r$-th system. We assume that $T^r$ is positive and has finite standard deviation $\sigma_A^r$.
We also assume that $\Expect{(T_1^r)^2}<\infty$. For $j\in{\mathbb N}$,
we refer to the $j$-th job to arrive after time zero as job $j$.

Upon its arrival to the $r$-th system, each job is assigned a processing time, which is the amount of time it takes the server
to process the work associated with that job.  The processing times are taken to be strictly positive and independent and identically distributed.  Also, the processing time distribution does not depend on $r$, i.e., is the same for all $r$, and is given by a
continuous distribution function $F$ on $\RRp$ such that $F(0)=0$.  It is assumed that $\overline{F}(x) = 1 - F(x)$ is  positive for each $x\in\RRp$ and that
$\overline{F}$ is a regularly varying function with index $-(p+1)$ for some $p>1$; namely, for all $t>0$,
\begin{equation}\label{eq:regvar}
 \bar F(t)>0 \mbox{ and } \lim_{x \rightarrow \infty}\frac{\overline{F}(tx)}{\overline{F}(x)} := t^{-(p+1)}.
\end{equation}
The above condition in particular implies that {\color{black}the} processing time distribution has a finite, positive second moment. The Pareto type 1 distribution with
parameters $m>0$ and $p>1$ (i.e. $\overline{F}(x)=\min(m^{p+1}x^{-p-1},1)$ for $x\in\RRp$) is a basic example of a processing time distribution
that satisfies \eqref{eq:regvar}.

For each $r\in\clr$, $\{\qq^r, \breve{v}_l^r, l \in \NN\}$, $E^r(\cdot)$, and the sequence of processing times are assumed to be mutually
independent of one another.

Jobs in the $r$-th system are served in accordance with the SRPT service discipline; that is, at each time the server preemptively
serves the job in system with the shortest remaining processing time.  For $t\ge 0$, $l=1, \ldots, \qq^r$ and $j=1,\dots,E^r(t)$,
$\breve v_l^r(t)$ and $v_j^r(t)$ denote the remaining processing time at time $t$ of initial job $l$ and job $j$ respectively.
For each $r\in\clr$ and $t\ge 0$, define
$$
\Z^r(t) = \sum_{l=1}^{\qq^r} \delta^+_{\breve{v}_l^r(t)} + \sum_{j=1}^{E^r(t)}\delta^+_{v_j^r(t)}.
$$
Then, for each $r\in\clr$ and $t\ge 0$, $\Z^r(t)\in\clm_F$ has a unit atom at the remaining processing time of each job in system.
Furthermore, for each $r\in\clr$, $\Z^r(\cdot)$ is a stochastic process with sample paths in $\mathcal{D}([0,\infty): \clm_F)$.
We will find it convenient to adopt the abbreviated phrases job {\it size} and job {\it sizes} to refer to a given job's remaining processing
time and the collection of all remaining processing times, respectively, at a given time. Also, a job's {\it initial size} refers to
its processing time upon arrival with initial job $l=1, \ldots, \qq^r$ having initial size $\breve{v}_l^r$ by convention.

\subsection{Heavy Traffic Conditions}\label{ss:ht}
Let $v$ denote a random variable having the distribution of the processing time of an incoming job.
For each $r\in\clr$, write 
$$\rho^r :=\lambda^r\mathbb{E}(v)\qquad\hbox{and}\qquad \rho^r_x :=\lambda^r\mathbb{E}(v\one_{[v \le x]})\qquad\hbox{for all }x\in\RRp.$$
It is assumed that there exists $\kappa\in\RR$ and $\sigma_A,\lambda\in(0,\infty)$ such that as $r\to\infty$,
\begin{equation}
	\label{eq:assuht}
r(\rho^r-1) \rightarrow \kappa, \qquad  \lambda^r \rightarrow \lambda,\qquad\text{and}\qquad \sigma_A^r \rightarrow \sigma_A.
\end{equation}
Note that the first limit above implies $ \lambda = 1/\mathbb{E}(v)$. Henceforth, $\kappa\in\RR$ and $\sigma_A,\lambda\in(0,\infty)$
satisfying \eqref{eq:assuht} are fixed. It is also assumed that
\begin{equation}
\label{eq:assuht2}
\limsup_{r\to\infty}{\mathbb E}(T_1^r)\le\lambda^{-1}
\qquad\text{and}\qquad
\limsup_{r\to\infty}\operatorname{Var}(T_1^r)\vee{\mathbb E}\left[\left(T_1^r-(\lambda^r)^{-1}\right)^2\right]\le\sigma_A^2.
\end{equation}
We note here that, for our results to hold, we only need finiteness of the above $\limsup$s. However, the above assumptions are made to treat the first inter-arrival time in a similar fashion as the later ones and thus to make the analysis less notationally cumbersome.
For $r \in \clr$ and $t\ge 0$, define
$$
\overline{E}^r(t) := \frac{E^r(r^2t)}{r^2} \ \text{ and } \ \widehat{E}^r(t) := \frac{E^r(r^2t) - \lambda^r r^2 t}{r} = r(\overline{E}^r(t)-\lambda^rt).
$$
Assume that as $r \rightarrow \infty$,
\begin{equation}\label{eq:estar}
\widehat{E}^r(\cdot) \xrightarrow{d} E^*(\cdot)
\end{equation}
in $\mathcal{D}([0, \infty): \RR)$, where $E^*(\cdot)$ is a one-dimensional Brownian motion starting from zero with zero drift and variance
$\lambda^3 \sigma_A^2$. This also implies that as, $r \rightarrow \infty$,
\begin{equation}\label{eq:aflln}
\overline{E}^r(\cdot) \xrightarrow{d} \lambda(\cdot),\qquad\text{where}\qquad \lambda(t) := \lambda t\qquad\hbox{for all }t \ge 0.
\end{equation}

\subsection{Scaling}
For $x\in\RRp$, let
\begin{equation}\label{def:S}
S(x) = \frac{1}{\mathbb{E}(v \one_{[v>x]})}.
\end{equation}
The function $S(\cdot)$ plays an important role in our analysis. As shown in \cite{DGP1109,K16},
it has {\color{black}the} same order of magnitude as the response time of jobs with remaining processing time $x$ in the system
at time zero, in the fluid limit.
Here, due to the assumptions on $F(\cdot)$, $S(\cdot)$ is a positive, nondecreasing, continuous function such
that $\lim_{x\to\infty}S(x)=\infty$.  In particular, the right continuous inverse $S^{-1}(\cdot)$ exists and is well defined on all of $\RRp$.  Then,
for $y\in\RRp$, {\color{black}we have}
\begin{equation}\label{def:Sinv}
S^{-1}(y) := \inf\{u >0: S(u) > y\},
\end{equation}
and the function $y \mapsto S^{-1}(y)$ is a nonnegative, nondecreasing, right continuous function which is strictly increasing for $y \in [S(0), \infty)$.  Also, for all $y \in [S(0), \infty)$,
\begin{equation}
	S(S^{-1}(y)) = y.\label{eq:ssinv}
\end{equation}
In \cite{DGP1109}, a version of \eqref{def:Sinv} arises as the left edge of the support of the measure valued fluid model solutions studied there.
For each $r\in\clr$, let
\begin{equation}\label{def:cr}
c^r := S^{-1}(r).
\end{equation}
Note that $c^r = 0$ if $r \le S(0)$ and $c^r>0$ if $r > S(0)$. As we are interested in large values of $r$, from now on, we will assume without loss of generality that the elements of $\mathcal{R}$ are all larger than $S(0)$. Then, \eqref{eq:ssinv} and \eqref{def:cr} imply that for all $r\in\clr$,
\begin{equation}\label{eq:Scr}
S(c^r)=r.
\end{equation}
As noted in the introduction, the quantity  $c^r$,  which was introduced in \cite{P15}, identifies the correct scaling needed  in order to obtain a nontrivial limit for the queue length process in the light
tailed case studied there (see \eqref{eq:lighttail}). We will see that this quantity  is  key for the analysis of regularly varying tails as well. For each $r\in\clr$ and $t\ge 0$, define
\begin{equation}\label{ztildef}
\tZ^r(t) =  \frac{c^r}{r} \sum_{l=1}^{\qq^r} \delta^+_{\breve{v}_l^r(r^2t)/c^r} + \frac{c^r}{r}\sum_{i=1}^{E^r(r^2t)}\delta^+_{v_i^r(r^2t)/c^r}.
\end{equation}
Thus $\tZ^r(\cdot)$ is obtained from $\Z^r(\cdot)$ by adding three types of scaling: the time is scaled by $r^2$, the mass is scaled by $c^r/r$ and the space (representing the job sizes) is scaled by $1/c^r$.

{\color{black}To illustrate this scaling, we consider the Pareto type 1 distribution with parameters $m>0$ and $p>1$ (i.e. $\overline{F}(x)=\min(m^{p+1}x^{-p-1},1)$ for $x\in\RRp$).  Then for $c_p:=m^{1+p}(1+p)/p$ and for each $r\in\clr$ such that $c^r \ge m$, we find that $c^r=(c_pr)^{1/p}$ and $c^r/r=\frac{c_p^{1/p}}{r^{(p-1)/p}}$, which respectively tend to the constants $2m^2r$ and $2m^2$ as $p\searrow 1$ and $m$ and $m/r$ as $p\to\infty$.  The latter is traditional diffusion scaling.  Upon noting that the ratio of two regularly varying functions with the same index is slowly varying, we see that for $F$ satisfying \eqref{eq:regvar} for some $p>1$, $c^r$ takes the form $L_p(r)\sqrt[p]{r}$, $r\in\clr$, for some distribution dependent, slowly varying function $L_p$.  See Section \ref{sec:regvarfunc} for a brief summary of the relevant properties of regularly and slowly varying functions.}

For each $r\in\clr$, $t\ge 0$, and  $f: \mathbb{R}_+ \rightarrow \mathbb{R}$,
define
$$
Z_f^r(t) := \langle f, \tZ^r(t)\rangle.
$$
We will also write, for {\color{black}$a\in [0,\infty] := [0, \infty) \cup\{\infty\}$} and $t \ge 0$, 
\begin{equation}\label{zqldef}
Z_a^{r}(t) := Z_{\one_{[0,a]}}^r(t) = \int_0^{a}\tZ^r(t)(dx).
\end{equation}
For each $r\in\clr$ and $t\ge 0$, we adopt the notation $Q^r(t) = Z_{\one}^r(t) = \int_0^{\infty}\tZ^r(t)(dx)$ so that $Q^r(t)$ represents $c^r$ times the diffusion scaled
queue length in the $r$-th system at time instant $t$.

For all $x\in\RRp$, let $\chi(x) = x$ and $\chi_a(x) := \chi(x) \one_{[0,a]}(x)$ for any $a\in\RRp$.
Also, by convention, $\chi_{\infty}=\chi$.  For each $r\in\clr$, $t\ge 0$ and $a\in [0,\infty]$, define
\begin{equation}\label{trunkworkdef}
W_a^r(t) := Z_{\chi_{a}}^r(t) = \langle \chi_a, \tZ^r(t)\rangle.
\end{equation}
For $r\in\clr$, $a\in\RRp$ and $t\ge 0$, $W_a^r(t)$ is equal to {\color{black}the} amount of work associated with jobs
of size
less or equal
to $ac^r$ at time $t$ in the $r$-th system under diffusion scaling. Further note that for each
$r\in\clr$, $W_{\infty}^r(\cdot)$ is the diffusion scaled workload process and $\lim_{a\to\infty}W_a^r(t)=W_{\infty}^r(t)$ for each $t\ge 0$, almost surely.
Observe that for each $r\in\clr$ and each fixed $a\in[0,\infty]$,
$W_a^r(\cdot)\in\mathcal{D}([0,\infty): \RRp)$.  For each $r\in\clr$, we refer to the collection $\{W_a^r(\cdot), a\in\RRp\}$ as the
{\it rescaled, truncated workload processes}, which is a random field on $\RRp^2$ taking values in $\RRp$.
Also note that for $r\in\clr$ and each fixed $t\ge 0$, $W_{\cdot}^r(t)\in\mathcal{D}([0,\infty): \RRp)$. 

\subsection{Asymptotic Conditions for the Sequence of Initial Conditions}
\label{initconass}
We assume that there exists an $\RRp$ valued, continuous, nondecreasing stochastic process
$\{w^*(a): a\in\RRp\}$, with $w^*(\infty) := \lim_{a \rightarrow \infty} w^*(a)$ satisfying $\mathbb{E}(w^*(\infty)) < \infty$,
such that, as $r\to \infty$,
\begin{equation}\label{eq:assuinitcond}
{\color{black}\left( W_{\cdot}^r(0), W_{\infty}^r(0)\right)}\xrightarrow{d} \left(w^*(\cdot), w^*(\infty)\right)
\end{equation}
in $\mathcal{D}([0, \infty): \RRp)\times \RRp$, 
and
\begin{equation}\label{eq:assuinitcondb}
\{{\color{black} W_{\infty}^r(0)}; r \in \clr\} \mbox{ is uniformly integrable.}
\end{equation}
Note that \eqref{eq:assuinitcond} and \eqref{eq:assuinitcondb} imply that, for any $a\in\RRp$,
$$
\lim_{r \rightarrow \infty}\mathbb{E}\left(W_{{\color{black}a}}^r(0)\right) = \mathbb{E}(w^*(a))\qquad\hbox{and}\qquad \lim_{r \rightarrow \infty}\mathbb{E}\left(W_{{\color{black}\infty}}^r(0)\right) = \mathbb{E}(w^*(\infty)).
$$
We further assume that there exist some $\eta^* \in (0, p-1)$, $\astar> 0$ and $\alpha^* \in (0,p]$ such that
\begin{equation}\label{eq:initcondq1}
\limsup_{r \rightarrow \infty}\sup_{a \in [\astar(c^r)^{-1}, 1]} a^{-(p-\eta^*)}\mathbb{E}\left(W_{{\color{black}a}}^r(0)\right) < \infty
\end{equation}
and
\begin{equation}\label{eq:initcondq2}
\limsup_{a \rightarrow \infty}a^{\alpha^*}\mathbb{E}(w^*(\infty) - w^*(a)) < \infty.
 \end{equation}
Assumption \eqref{eq:initcondq1} insures that the work associated with initial jobs with remaining processing times near zero vanishes at a suitable rate as $r$ tends
to infinity.   Assumption \eqref{eq:initcondq2} insures that the limiting work associated with initial jobs with large remaining processing times vanishes at a suitable rate.
Assumptions \eqref{eq:assuinitcondb} and \eqref{eq:initcondq1} imply that
\begin{equation}\label{eq:Cxifinite}
\sup_{a >0}a^{-(p-\eta^*)}\mathbb{E}\left(w^*(a)\right) < \infty.
\end{equation}
Finally, we assume that for any $a\in\RRp$,
\begin{equation}\label{smalljobas}
Z_{a/c^r}^r(0)= \frac{c^r}{r}\sum_{l=1}^{\qq^r}  \one_{[\breve{v}_l^r \le a]} \xrightarrow{P} 0 \ \text{ as } r \rightarrow \infty.
\end{equation}

\begin{remark}\label{naturalin}
A  guideline for whether Assumptions \eqref{eq:assuinitcond}--\eqref{smalljobas}
are natural is to check whether a sequence of systems such that each system starts from zero jobs  at time zero satisfies these assumptions  at
any fixed positive time $t$. It can be checked from the proofs in Section \ref{proofs} that this is indeed the case; namely, if each system starts
with zero jobs  then at any time $t>0$, Assumptions \eqref{eq:assuinitcond}--\eqref{smalljobas}
are satisfied with $\left( W_{\cdot}^r(0), W_{\infty}^r(0)\right)$ replaced by $\left( W_{\cdot}^r(t), W_{\infty}^r(t)\right)$ for each $r\in\clr$ and $\{\breve{v}_l^r\}_{1\le l \le \qq^r}$ replaced with $\{v_i(r^2t) : 1 \le i \le E^r(r^2t),   v_i(r^2t) >0, \} \cup \{\breve{v}_l^r(r^2t), 1 \le l \le \qq^r, \breve{v}_l^r(r^2t) >0\}$. See Appendix \ref{icverify} for a sketch of how to verify this. 
\end{remark}

\subsection{Some initial conditions satisfying Assumptions \eqref{eq:assuinitcond}--\eqref{smalljobas}}\label{iceg}
We give the following two sets of initial conditions which are easily checkable and satisfy Assumptions \eqref{eq:assuinitcond}--\eqref{smalljobas}.

(I) 
Suppose the following hold:
\begin{itemize}
\item[(i)] for each $r\in\clr$, $\{\breve{v}_l^r : l \ge 1\}$ is a sequence of independent and identically distributed random variables that is independent of $\qq^r$;\\
\item[(ii)] For some $\qq^*$ with $\mathbb{E}(\qq^*)< \infty$, $c^r\qq^r/r \xrightarrow{L^1} \qq^*$ as $r \rightarrow \infty$;\\
\item[(iii)]$\sup_{r \in \clr}\mathbb{E}\left[\left(\breve{v}_1^r/c^r\right)^2\right] < \infty$ and $\breve{v}_1^r/c^r \xrightarrow{d} \breve{v}^*$ as $r \rightarrow \infty$, where $\breve{v}^*$ has a continuous distribution;\\ 
\item[(iv)] there is a random variable $\underline{v}$ that stochastically lower bounds $\breve{v}_1^r/c^r$ for all $r \in \clr$ and satisfies
$$
\limsup_{a \downarrow 0}a^{-(p-1-\eta^*)}\mathbb{P}(\underline{v} \le a) < \infty,
$$
for some $\eta^* \in (0, p-1)$.
\end{itemize}
Then Assumptions \eqref{eq:assuinitcond}--\eqref{smalljobas} are satisfied with $\alpha^* = 1$, $\eta^*$ as in part (iv)
above, any $\astar>0$, $w^*(a) = \qq^*\mathbb{E}\left(\breve{v}^* \one_{[\breve{v}^* \le a]}\right)$ for $a\in\RRp$ and $w^*(\infty) = \qq^*\mathbb{E}\left(\breve{v}^*\right)$. See Appendix \ref{icverify} for a sketch of how to check that assumptions
\eqref{eq:assuinitcond}--\eqref{smalljobas} hold for such a sequence of initial conditions.\\

(II) Another set of conditions for which Assumptions \eqref{eq:assuinitcond}--\eqref{smalljobas} hold is that along with (i) in (I) above, for some $\al >0$, $(c^r)^{1+\al} \qq^r/r\to 0$ in $L^1$ and $\{\breve{v}_1^r/c^r, r \in \clr\}$ is $L^1$ bounded.  In particular, it can be checked that under these conditions, Assumptions \eqref{eq:assuinitcond}--\eqref{smalljobas} hold for any $\alpha^* \in (0,p]$, any $\astar>0$, $\eta^* = (p-1-\al)\vee (p-1)/2$ and $w^*(a) = 0$ for $a \in (0, \infty]$. Note that these conditions are trivially satisfied if each system starts from empty, namely $\qq^r=0$ for all $r\in\clr$.

\section{Main results}\label{s:main}
In this section, we state the five main results in this paper. The conditions introduced in Sections \ref{ss:sd}, \ref{ss:ht}, and \ref{initconass} will be assumed to hold throughout this work and will not be noted explicitly in statements of various results. Thus henceforth, we consider
a sequence (or sequences) of SRPT queues indexed by $\mathcal{R}$ satisfying the above
conditions.

\subsection{A random field governing {\color{black}the} limiting behavior}\label{ss:mainrf}
The first theorem (stated below) gives the important observation that for processing time distributions
with regularly varying tails, the joint limiting behavior
of the truncated workload processes is captured by a random field constructed from a single Brownian motion using the Skorohod map.
For $f \in \mathcal{D}([0, \infty): \RR)$ with $f(0)\ge 0$, let
\begin{equation}
	\label{eq:skormap}\Gamma[f](t):= f(t) - \inf_{0\le s \le t} \left(f(s)\wedge 0\right),\;  t \ge 0.
	\end{equation}
The function $\Gamma$  is known as the one-dimensional Skorohod map.

\begin{theorem}\label{workless}
Let $B$ be a standard real Brownian motion and $(\xi(\cdot), \xi(\infty))$ be a $\mathcal{C}([0, \infty): \RR_+) \times \RRp$ valued random variable with same distribution as $(w^*(\cdot),w^*(\infty))$ that is independent of $B$. For any $k\in\NN$ and any $0 \le a_1 <\dots < a_k\le \infty$, as $r \rightarrow \infty$,
$$
(W_{a_1}^r(\cdot), \dots, W_{a_k}^r(\cdot)) \xrightarrow{d} (W_{a_1}(\cdot), \dots, W_{a_k}(\cdot)) 
$$
in $\mathcal{D}([0, \infty): \RRp^k)$, where for $a\in[0,\infty]$,
\begin{equation}\label{eq:GX}
W_a(\cdot) := \Gamma[X_a](\cdot),
\end{equation}
with $\Gamma$ as in \eqref{eq:skormap} and $\{X_a(\cdot) : a\in[0,\infty]\}$, given as follows: for $t\ge 0$,
\begin{align}
X_0(t)&:=\xi(0)=0,\label{eq:x0defn}\\
X_a(t) &:= \xi(a) + \sigma B(t) + \left(\kappa - \frac{\lambda}{a^p}\right)t, \quad\hbox{for }0<a<\infty,\label{eq:xadefn}\\
X_{\infty}(t)&:= \xi(\infty) + \sigma B(t) + \kappa t,\label{eq:xinfinitydefn}
\end{align}
and $\sigma^2 : = \lambda \operatorname{Var}(v) + \lambda\sigma_A^2$, {\color{black}where $v$ is as in Section \ref{ss:ht} and $\sigma_A$ is as in \eqref{eq:assuht}.}
\end{theorem}
Due to \eqref{eq:GX}--\eqref{eq:xinfinitydefn}, $\xi(a)=X_a(0)=W_a(0)$ for all $a\in[0,\infty]$.
The key feature of the above result is that the Brownian motion $B(\cdot)$ that determines $X_a(\cdot)$ is the same for all $a\in[0,\infty]$.
In particular, $a$ only enters in the initial condition and the drift term.  In addition, $W_{\infty}(\cdot)$ is the diffusion limit of the workload process as given in \cite{IW70}.  Theorem \ref{workless} is proved in Section \ref{proofs}
as a consequence of Proposition \ref{comp} and Lemma \ref{jointZ}, stated there. In Proposition \ref{comp}, upper and lower bounds on $W_a^r(t)$ and $Z_a^r(t)$ for each $a\in[0,\infty]$ and $t\ge 0$
are given by coupling it with the workload process and queue length process for a SRPT queueing
system that satisfies all of the assumptions in Section \ref{ss:sd}, except that the renewal arrival process is thinned
to only include jobs with processing time at most $ac^r$. A notion of ordering of two SRPT systems, which we call intertwining, is introduced
in Section \ref{intsec} and used in a crucial way to obtain the queue length bounds in Proposition \ref{comp}.
In Lemma \ref{jointZ}, a functional central limit theorem (FCLT) is established for a finite collection of rescaled, truncated workload processes via the bounds obtained in Proposition \ref{comp} and establishing an FCLT for the bounding processes. Continuity properties of the Skorohod map imply Theorem \ref{workless} as a direct consequence of Lemma \ref{jointZ}.

\subsection{Limits for the queue length process and measure valued state descriptor}\label{ss:qm}
Theorem \ref{workless} can be used in describing the limiting behavior of
$
Z_f^r(\cdot) := \langle f, \tZ^r(\cdot)\rangle
$
for a rich class of functions $f$ as stated in the next theorem. This, in turn, gives distributional asymptotics for the scaled queue length process. Recall that $\chi(x)= x$ and $\one(x)=1$ for $x \in \RR_+$.

\begin{theorem}\label{workfn}
Let $f: [0, \infty) \rightarrow \mathbb{R}$ be any $C^{1}$ function such that $\lim_{x \rightarrow \infty} \frac{f(x)}{x}$ exists
and $\int_1^{\infty} \frac{|f'(x)|}{x^{\alpha^* +1}} dx < \infty$, where $\alpha^*$ is the constant appearing in Assumption \eqref{eq:initcondq2}.
Then, as $r \rightarrow \infty$,
$$
Z_f^r(\cdot) \xrightarrow{d} Z_f(\cdot)
$$
in $\mathcal{D}([0, \infty):\RR)$, where $Z_f$ is a real stochastic process with continuous sample paths, given by the formula
$$
Z_f(t) :=\int_0^{\infty}\left(\frac{f(x)}{x^2} - \frac{f'(x)}{x}\right)W_x(t)dx + \left(\lim_{x \rightarrow \infty}  \frac{f(x)}{x}\right) W_{\infty}(t), \quad t \ge 0.
$$
In particular, as $r\to\infty$,
$$
W_{\infty}^r(\cdot) = Z_{\chi}^r(\cdot) \xrightarrow{d} Z_{\chi}(\cdot)=W_{\infty}(\cdot)  \; \mbox{ and }\; Q^r(\cdot) = Z^r_{\one}(\cdot) \xrightarrow{d} Z_{\one}(\cdot)$$
in $\mathcal{D}([0, \infty):\RRp)$, where $Q(\cdot):=Z_{\one}(\cdot)$ satisfies
\begin{equation}\label{eq:lQ}
Q(t) = \int_0^{\infty}\frac{1}{x^2} W_x(t)dx, \quad t \ge 0.
\end{equation}
\end{theorem}
Theorem \ref{workfn} is proved in Section \ref{ss:pfworkfn}. An overview of this proof is given in Section \ref{ss:methods}.

The result in Theorem \ref{workfn} can be strengthend  to show that $\tZ^r$ converges in distribution to a measure valued process $\tZ$ in $\mathcal{D}([0,\infty): \clm_F)$.
This is stated in the next theorem, which is proved in Section \ref{sec:proofs}. The proof proceeds via integrating the random measure $\tZ^r$ against a class of test functions and analyzing weak convergence of the collection of processes thus obtained.

\begin{theorem}\label{meascon} As $r\to\infty$,
$$
\tZ^r(\cdot)  \xrightarrow{d} \tZ(\cdot)
$$
in $\mathcal{D}([0,\infty): \clm_F)$, where for each $t\ge 0$, the measure $\tZ(t)$ can be characterized as $\tZ(t)(\{0\})=0$,
$\tZ(t)(\RRp)=Q(t)$ and
$$
\tZ(t)[a,b] : = \int_a^b\frac{1}{x^2}W_x(t)dx + \frac{W_b(t)}{b} - \frac{W_a(t)}{a}, \quad 0 < a < b < \infty.
$$
\end{theorem}

\begin{remark}\label{heavylightcomp}
The integral expression \eqref{eq:lQ} in Theorem \ref{workfn} is quite different from the main result (Theorem 3.1) in \cite{P15}, which gives conditions under which light tailed processing times
result in a limit theorem that states $Q(\cdot)=W_{\infty}(\cdot)$ (state space collapse).  While the proofs given here do not cover the light tailed case, the concentration arguments given in \cite{P15}
could be used to argue that the measure valued state descriptors in the light tailed case, scaled as in \eqref{ztildef} above, would
converge to a point mass at one with (random) total mass given by the limiting workload process $W_{\infty}(\cdot)$. Consequently, the rescaled, truncated workload processes $W_x^r(\cdot)$ defined in \eqref{trunkworkdef} above, in the light tailed case, would converge to $W_{\infty}(\cdot)$ for $x> 1$, $W_1(\cdot)$ for $x=1$ and the process that is identically zero otherwise,
and the integral given in \eqref{eq:lQ} would be $W_{\infty}(t)$ for each $t\ge 0$, as it should from the results of \cite{P15}.  The results in Theorem \ref{workfn} and Theorem \ref{meascon} demonstrate that, in contrast to the light tailed
processing time distributions considered in \cite{P15}, heavy tailed processing time distributions do not exhibit state space collapse and the mass of the limiting scaled measure valued state descriptor is distributed as a time-varying random profile over $\RRp$, as opposed to a time-varying randomly sized point mass at one. 
\end{remark}

\subsection{Tail behavior of $\tZ$}\label{ss:tail}

The next result describes the asymptotic behavior of the limiting queue length and limiting workload processes defined in terms of the measure $\tZ$ when attention is restricted to the dynamics of jobs with large remaining processing times.
Let
\begin{equation}
	\label{eqwinftyp}
	W_{\infty}'(t) := t - \sup\{s \le t: W_{\infty}(s) = 0\}, \; t \ge 0,
	\end{equation}
which can be recognized as the duration of the current busy period when $W_{\infty}(t)$ is interpreted as the work in the system at time instant $t$.
We will see in Section \ref{skmpr} that $W_{\infty}'(\cdot)$ arises as the `path-wise derivative' of the Skorohod map with respect to the `drift parameter' of the process on which the map acts, which explains the notation $W_{\infty}'(\cdot)$. We will also assume a stronger version of \eqref{eq:initcondq2} for this result, namely
\begin{equation}\label{initdec}
\lim_{x \rightarrow \infty} x^p(\xi(\infty) - \xi(x)) = 0 \qquad \text{almost surely}.
\end{equation}
In particular, \eqref{initdec} holds when $\xi(\infty)=0$.

\begin{theorem}\label{tailtZ}
Assume \eqref{initdec} holds. For every $t\ge 0$, as $a \rightarrow \infty$,
\begin{align*}
\frac{a^p}{\lambda}\langle \chi \one_{[a, \infty)}, \tZ(t) \rangle &\rightarrow W_{\infty}'(t) \qquad \text{almost surely},\\
\frac{(p+1)a^{p+1}}{p\lambda} \tZ(t)[a, \infty) &\rightarrow W_{\infty}'(t) \qquad \text{almost surely}.
\end{align*}
 In particular, for any $t\ge 0$ such that $W_{\infty}'(t) \neq 0$, as $a \rightarrow \infty$,
\begin{equation}\label{eq:AsymWork}
\frac{\langle \chi \one_{[a, \infty)}, \tZ(t) \rangle}{\mathbb{E}\left(v \ \vert \ v \ge a\right)\tZ(t)[a, \infty)} \rightarrow 1 \qquad \text{almost surely}.
\end{equation}
\end{theorem}
Theorem \ref{tailtZ} is proved in Section \ref{sec:proofs} and proceeds via connecting the tail mass processes \{$\tZ(t)[a,\infty) : t \ge 0\}$ for large $a$ with the process $\{W_{\infty}'(t) : t \ge 0\}$.

\begin{remark}\label{collapsealarge}
The above result says that if, {\color{black}in the diffusion limit,} we restrict attention to jobs {\color{black}in system} of size more than $a$ (for large $a$), the cumulative workload due to these jobs can be {\color{black}approximated} by multiplying the number of such jobs present in the system with the expected size of an incoming job conditional on it being more than {\color{black}size} $a$.
{\color{black}In other words in the diffusion limit, so few large jobs have entered service by a finite time $t$ that the work associated with such jobs satisfies \eqref{eq:AsymWork}.}
This result can be heuristically understood from the SRPT dynamics under which small jobs are given priority and large jobs remain unprocessed at typical time points when the system has small jobs present.  Theorem \ref{tailtZ} can be seen as a form of asymptotic state space collapse when one restricts attention to jobs with large remaining processing times.
\end{remark}

\subsection{Asymptotic state space collapse as $p \rightarrow \infty$}\label{sec:sc}
As stated in Remark \ref{heavylightcomp}, the limiting scaled queue length process given in Theorem \ref{workfn} differs qualitatively from its light tailed analogue treated in \cite{P15} in that, although the limiting scaled queue length and limiting scaled workload processes are driven by the same Brownian motion $B$, there is no state space collapse as in \cite{P15}. However, as $p \rightarrow \infty$ (that is, the tail of the processing time distribution becomes lighter), we obtain a limiting state space collapse as described in Theorem \ref{collapse} below. As we are interested in large values of $p$ here, we will only consider $p \ge 2$. To make the dependence on $p$ explicit, we consider a family of distributions $\{F^{(p)}(\cdot) : p \ge 2\}$ such that for each $p \ge 2$, $\overline{F}^{(p)}(\cdot) := 1- F^{(p)}(\cdot)$ is a regularly varying function; that is, \eqref{eq:regvar} is satisfied by $\overline{F}^{(p)}(\cdot)$. 
For each $p \ge 2$, consider a sequence of SRPT queues indexed by $\mathcal{R}$ such that the initial conditions $\{\qq^{(p),r}, \breve{v}_l^{(p),r}, l \in \NN, r \in \mathcal{R}\}$ satisfy the assumptions of Section \ref{initconass} and the arrival processes $\{E^r(\cdot), r \in \mathcal{R}\}$ do not depend on $p$. Consequently, $\lambda^{(p)} = 1/\mathbb{E}(v^{(p)})$, where $v^{(p)}$ is distributed as $F^{(p)}(\cdot)$, does not depend on $p$ and we will write this quantity as $\lambda$.
The processing times of jobs for each $p \ge 2$ are distributed as $F^{(p)}(\cdot)$. For each $p \ge 2$, write $\sigmap = \sqrt{\lambda \operatorname{Var}(v^{(p)}) + \lambda\sigma_A^2}$.
For each $p \ge 2$, we let $\xi^{(p)}(\infty)$ denote the limiting initial workload
(i.e. the quantity analogous to $\xi(\infty)$ in Theorem \ref{workless} for the $p$-th system) and let
$\xi^{(p)}(\cdot)$ denote the limiting initial truncated workload process (analogous to $\xi(\cdot)$ in Theorem \ref{workless} for the $p$-th system). We will also elucidate the dependence of $\eta^*$ in Assumption \eqref{eq:initcondq1} by writing it as $\eta^*(p)$. We assume that
\begin{equation}\label{varunibdd}
\sup_{p \ge 2}\mathbb{E}\left[(v^{(p)})^2\right] < \infty,\quad \sup_{p \ge 2}\mathbb{E}\left[\xi^{(p)}(\infty)\right] < \infty,\quad\text{and} \quad \sup_{p\ge 2}\cop<\infty,
\end{equation}
where $\cop :=2\sup_{a >0}a^{-(p-\eta^*(p))}\mathbb{E}\left(\xi^{(p)}(a)\right)$ for each $p\ge 2$.
Writing $Q^{(p)}(\cdot)$ for $Q(\cdot)$ and $W^{(p)}_{\infty}(\cdot)$ for $W_{\infty}(\cdot)$ for each $p\ge 2$ to denote the limiting queue length and limiting workload processes respectively, we have that, for $p\ge 2$,
$$
Q^{(p)}(t) = \int_0^{\infty}\frac{1}{x^2}\Gamma \left[ \xi^{(p)}(x) + \sigmap B(\cdot) + (\kappa - \lambda x^{-p})\iota(\cdot)\right](t)dx, \; t \ge 0,
$$
where $\iota$ denotes the identity map on $[0,\infty)$ and $W^{(p)}_{\infty}(t) = \Gamma [X^{(p)}_{\infty}](t)$
for $X^{(p)}_{\infty}(t) = \xi^{(p)}(\infty) + \sigmap B(t) + \kappa t$, $t \ge 0$.
For the state space collapse result, we will require that $\eta^*(\cdot)$ satisfies
\begin{equation}\label{assp}
\liminf_{p \rightarrow \infty}\frac{p-1-\eta^*(p)}{\log p} = \infty.
\end{equation}
Moreover, we will require for any $a \in (1, \infty)$,
\begin{equation}\label{assp2}
\lim_{p \rightarrow \infty} \mathbb{E}(\xi^{(p)}(\infty) - \xi^{(p)}(a)) = 0.
\end{equation}
Then, \eqref{assp} implies that for large $p$, $\mathbb{E}(\xi^{(p)}(a))$ decreases to zero sufficiently fast with $a$ tending to zero. Also, \eqref{assp2} implies that the main contribution to the limiting initial workload process $\xi^{(p)}(\cdot)$ for large values of $p$ comes from initial jobs with size in $(0,1]$. Note that if the system starts from empty, namely $\qq^{(p),r} = 0$ for all $r \in \mathcal{R}$ and $p \ge 2$, then for any $t \ge 0$ and any $a \in (1, \infty)$, by the Lipschitz property of the Skorohod map given in \eqref{eq:lipSM} below,
$$
\lim_{p \rightarrow \infty} \mathbb{E}\left(W^{(p)}_{\infty}(t) - W^{(p)}_{a}(t)\right) \le 2\lambda t a^{-p} \rightarrow 0 \quad \text{as } \quad p \rightarrow \infty.
$$ 
Hence, by the discussion in Remark \ref{naturalin}, Assumption \eqref{assp2} is indeed a natural assumption on $\xi^{(p)}(\cdot)$.
\begin{theorem}\label{collapse}
Assume that \eqref{varunibdd} and \eqref{assp2} hold and we can choose $p \mapsto \eta^*(p)$ such that $\eta^*(\cdot)$ satisfies \eqref{assp}. Then, for any $T>0$,
$$
\sup_{t \in [0,T]}\left|Q^{(p)}(t) -  W^{(p)}_{\infty}(t)\right| \xrightarrow{P} 0 \quad \text{as } \quad p \rightarrow \infty.
$$
\end{theorem}
Theorem \ref{collapse} is proved in Section \ref{sec:proofs}. The proof essentially proceeds by showing that as $p \rightarrow \infty$, the time varying mass profile of the limiting measure valued state descriptor collapses onto a point mass at one.  {\color{black}For the Pareto Type I example with $\overline{F}^{(p)}(x)=\min\left(\left(\frac{\lambda x(p+1)}{p}\right)^{-p-1},1\right)$ for $x\in\RRp$ and $p\ge 2$, we have $\lambda^{(p)}=\lambda$ and $\operatorname{Var}(v^{(p)})=\lambda^{-2}/(p^2-1)$  for all $p\ge 2$, and the latter tends to zero as $p\to\infty$.  In fact, the measure corresponding to the complementary cumulative distribution function (CCDF) $\overline{F}^{(p)}(\cdot)$ converges weakly to the point mass at $\lambda^{-1}$, i.e., the service times are asymptotically deterministic, which makes the state space collapse rather intuitive.  A somewhat more interesting example based on the Lomax distribution that does not have asymptotically deterministic service times has CCDFs $\overline{G}^{(p)}(x)=\left(1+\frac{\lambda x}{p}\right)^{-p-1}$, $x\in\RRp$, with $\lambda^{(p)}=\lambda$ and $\operatorname{Var}(v^{(p)})=\lambda^{-2}(p+1)/(p-1)$ for all $p\ge 2$.  This gives rise to an exponential rate $\lambda$ distribution in the $p\to\infty$ limit.}

\begin{remark}\label{zerocond}
	Consider the initial condition of the form discussed in (II) of Section \ref{iceg}, namely, along with (i) in (I) of Section \ref{iceg}, suppose for some $\al >0$, $(c^r)^{1+\al} \qq^r/r\to 0$ in $L^1$ and $\{\breve{v}_1^r/c^r, r \in \clr\}$ is $L^1$ bounded (note that $\qq^r=0$ for all $r \in \mathcal{R}$ is a special case).
	In this case one can replace $\al^*$ in Theorem \ref{workfn} with $p$. {\color{black}Also, in this case the assumption \eqref{initdec} in Theorem \ref{tailtZ} can be omitted. Moreover, if we consider a sequence of initial conditions indexed by $p \ge 2$ satisfying (II) of Section \ref{iceg} such that the choice of $\alpha= \alpha(p)$ can be made such that $\alpha(p)/\log p \rightarrow \infty$ as $p \rightarrow \infty$, then the assumptions \eqref{varunibdd}, \eqref{assp} and \eqref{assp2} in Theorem \ref{collapse} can be replaced by the single assumption $\sup_{p \ge 2}\mathbb{E}\left[(v^{(p)})^2\right] < \infty$. This applies, in particular, if $\qq^r=0$ for all $r \in \mathcal{R}$.} 
	\end{remark}

\section{Preliminaries}\label{prelim}
In this section we recall some basic facts and record some well known results that will be used several times in this work.

\subsection{Properties of the Skorohod Map}\label{skmpr}
Recall the Skorohod map $\Gamma$ defined in \eqref{eq:skormap}. {\color{black}The properties of $\Gamma$ summarized here can be found in \cite[Chapter 13.5]{W01}, unless noted otherwise.}
Then, denoting $\mathcal{D}_0([0,\infty): \RR)$ as the space of all $f \in \mathcal{D}([0,\infty): \RR)$
with $f(0)\ge 0$, the map $\Gamma$ is a continuous map from $\mathcal{D}_0([0,\infty): \RR)$ to $\mathcal{D}([0,\infty): \RRp)$.
Furthermore, the following Lipschitz property holds: for all  $f_1$, $f_2 \in \mathcal{D}_0([0,\infty): \RR)$ and  $T \in [0,\infty)$,
\begin{equation}\label{eq:lipSM}
	\sup_{t\in[0,T]}|\Gamma[f_1](t)- \Gamma[f_2](t)| \le 2 \sup_{t\in[0,T]}|f_1(t)- f_2(t)|.
\end{equation}
For any $f \in \mathcal{D}_0([0,\infty)$ and any $t_1, t_2$ such that $0 \le t_1 \le t_2 \le T$, defining functions $g_1(s) = \Gamma[f](t_1)$ and $g_2(s) = \Gamma[f](t_1) + f(s) - f(t_1)$ for $s \in [t_1, t_2]$, note that $\Gamma[g_1](t_2) = \Gamma[f](t_1)$ and $\Gamma[g_2](t_2) = \Gamma[f](t_2)$. Using \eqref{eq:lipSM}, we conclude
\begin{equation}\label{skortimecomp}
|\Gamma[f](t_2) - \Gamma[f](t_1)| \le 2\sup_{t_1 \le s \le t_2}|g_2(s) - g_1(s)| = 2\sup_{t_1 \le s \le t_2}|f(s) - f(t_1)|.
\end{equation} 
The following monotonicity property also holds.  Suppose $f_1, f_2 \in \mathcal{D}_0([0,\infty): \RR)$ are such that, for all $0\le s \le t <\infty$
$f_1(t) - f_1(s) \le f_2(t) - f_2(s)$ and $f_1(0)\le f_2(0)$.
Then, it follows that $f_1(t)\le f_2(t)$ and $\sup_{0\le s\le t} (f_1(t)-f_1(s))\le \sup_{0\le s\le t} (f_2(t)-f_2(s))$ for all $t\ge 0$.  Hence,
\begin{equation}\label{eq:monoskor}
	\Gamma[f_1](t) \le \Gamma[f_2](t) \mbox{ for all } t \ge 0.
	\end{equation}
Let $f \in \mathcal{D}_0([0,\infty): \RR)$. For $\eps \in \RR$, let $f_{\eps}(t):= f(t)+\eps t$, $t\ge 0$. Then for every $t\ge 0$
\begin{equation}\label{eq:deriv}
	\eps^{-1} \left[\Gamma[f_{\eps}](t) - \Gamma[f_0](t)\right] \to t - \sup \{0\le s \le t: f_0(s)=0\} \mbox{ as } \eps \to 0.
\end{equation}
For a proof we refer to 
 \cite[ Theorem 1.1]{mandelbaum2010} (see also pages 1921-1922 of \cite{dieker2014}).

\subsection{Regularly Varying Functions}
\label{sec:regvarfunc}
Recall that we assume that the complementary cumulative distribution function $\bar F$ of the processing time distribution is a regularly varying function with index $-(p+1)$
for some $p>1$, namely \eqref{eq:regvar} is satisfied.  Also recall that $S(\cdot)$ is given by \eqref{def:S}.
A function $L: [0,\infty) \to \RRp$ is called a slowly varying function if
$$\lim_{x\to \infty} \frac{L(tx)}{L(x)} = 1 \mbox{ for all } t >0.$$
We will frequently use the following well known properties of regularly varying functions (see \cite[Theorems 1.2.1, 1.2.4, 1.2.6]{mikosch1999}).
\begin{enumerate}[(a)]
\item From \cite[Remark 1.2.3]{mikosch1999}, if $L(\cdot)$ is slowly varying, then for all $\epsilon>0$,
$$
\lim_{x\to\infty} \frac{L(x)}{x^{\epsilon}}=0\qquad\hbox{and}\qquad\lim_{x\to\infty} \frac{L(x)}{x^{-\epsilon}}=\infty.
$$
\item There exists a slowly varying function $L$ such that
$\overline{F}(x) = \frac{1}{x^{p+1}}L(x)$ for all $x \ge 1$.
Henceforth, such a function $L(\cdot)$ is fixed.
\item From Karamata's Theorem \cite[Theorem 1.2.6 (b)]{mikosch1999} with $\alpha = -p-1$,
\begin{equation}\label{intregv}
\lim_{x \rightarrow \infty} \frac{\int_{x}^{\infty}\overline F(t)dt}{x \overline F(x)} = \frac{1}{p}.
\end{equation}
In particular, the function $z \mapsto \mathbb{E}(v \one_{[v > z]})$ is regularly varying with index $-p$
and therefore
for all $a>0$,
\begin{equation}\label{karama}
\lim_{r \rightarrow \infty}\frac{\mathbb{E}(v \one_{[v > ac^r]})}{\mathbb{E}(v \one_{[v > c^r]})} = \frac{1}{a^p}.
\end{equation}
In fact, using   \cite[Theorem 1.2.4 and Theorem 1.2.6 (b)]{mikosch1999} one has that for all $\delta>0$
\begin{equation}\label{karama2}
\lim_{r \rightarrow \infty} \frac{\mathbb{E}(v \one_{[v > u c^r]})}{\mathbb{E}(v \one_{[v > c^r]})} = \frac{1}{u^p} \quad \text{ uniformly for } u \in [\delta, \infty).
\end{equation}
Also, there exists a slowly varying function $\hat{L}$ such that $\mathbb{E}(v \one_{[v > z]}) = z^{-p}\hat{L}(z)$ for all $z >0$. 
By  \cite[Theorem 1.2.1]{mikosch1999}, $\hat{L}$ can be represented as
\begin{equation}\label{def:hatL}
\hat{L}(z) = c(z) \exp\left(\int_1^z\frac{\epsilon(y)}{y}dy\right), \quad z \ge 1,
\end{equation}
where $c$ and $\epsilon$ are nonnegative Borel measurable functions satisfying $\lim_{x \rightarrow \infty}c(x) = c_0 \in (0,\infty)$ and $\lim_{x \rightarrow \infty}\epsilon(x) \rightarrow 0$.

\item By \eqref{def:S} and {\color{black}\eqref{karama}}, $S(\cdot)$ is regularly varying with index $p$.  Then, by Karamata's theorem (\cite[Theorem 1.2.6 (b)]{mikosch1999}), as $x\to \infty$
$ \frac{S(x)L(x)}{x^p} \to \frac{p}{p+1}
$, where $L(\cdot)$ is given in (b).  Combining this with (a), it follows that for any $\epsilon>0$, there exists $x_{\epsilon}>0$ such that for all $x \ge x_{\epsilon}$,
\begin{equation}\label{Srvbd}
\frac{p}{p+1}x^{p-\epsilon} < S(x) < \frac{p}{p+1}x^{p+\epsilon}.
\end{equation}
By \eqref{Srvbd} {\color{black}with $x=S^{-1}(r)$, \eqref{eq:Scr}, and the fact that $S^{-1}(\cdot)$ is strictly increasing,} it follows that for any $\epsilon>0$, there exists $r_{\epsilon}>0$ such that for $r \ge r_{\epsilon}$,
\begin{equation}\label{crscale}
\left(\frac{(p+1)r}{p}\right)^{1/(p+\epsilon)} < c^r < \left(\frac{(p+1)r}{p}\right)^{1/(p-\epsilon)}.
\end{equation}
In particular,
\begin{equation}\label{limitcroverr}
\lim_{r\to\infty}\frac{c^r}{r}=0.
\end{equation}

\item {\color{black}$S^{-1}(\cdot)$ is regularly varying with index $1/p$.}
	\end{enumerate}
\subsection{ A Functional Central Limit Theorem}\label{ss:fclt}
We will need the following well known functional central limit theorem (cf.  \cite[Proposition A.1]{P15}).
For this, recall the definitions of $\lambda^r$, $E^r(\cdot)$, $\bar E^r(\cdot)$, and $\hat E^r(\cdot)$,
for $r\in\clr$, and  $\lambda$, $\lambda(\cdot)$ and $E^*(\cdot)$ given in Section \ref{ss:ht}.
Also, for $r\in\clr$, let $\lambda^r(t)=\lambda^rt$ for $t\ge 0$.
 
\begin{proposition}
	\label{prop:fclt}
For each $r \in \clr$, let $\{x_k^r\}_{k=1}^{\infty}$ be a sequence of nonnegative independent and identically
distributed random variables, with finite mean $m^r$ and finite standard deviation $s^r$, that is independent of
$E^r(\cdot)$. Suppose that for some finite nonnegative constants $m$ and s, $m^r\to m$ and $s^r \to s$,
as $r\to \infty$. Further suppose that, for each $\delta >0$
$$\lim_{r\to \infty} \mathbb{E} \left[ (x_1^r-m^r)^2 \one_{|x^r_1-m^r|> r\delta}\right]=0.$$
For $r \in \clr$, $n \in \NN$ and $t\in [0,\infty)$, let
$$
X^r(n) = \sum_{k=1}^n x_k^r\qquad \mbox{and}\qquad \hat X^r(t) = (X^r(\lfloor r^2t\rfloor) - \lfloor r^2t\rfloor m^r)/r.
$$
Then, as $r\to \infty$, 
$(\hat E^r(\cdot), \hat X^r(\cdot)) \xrightarrow{d} (E^*(\cdot), X^*(\cdot))$ in $\mathcal{D}([0,\infty):\RR^2)$, where $E^*$ is given as in	
\eqref{eq:estar} and $X^*$ is a Brownian motion starting from zero with zero drift and variance $s^2$, that is independent of $E^*$.
Furthermore, as $r\to \infty$,
$$
[X^r(r^2 \bar E^r(\cdot)) - r^2\lambda^r(\cdot) m^r]/r \xrightarrow{d} X^*(\lambda(\cdot)) + m E^*(\cdot),$$
in $\mathcal{D}([0,\infty):\RR)$.
\end{proposition}

\subsection{Tightness and Convergence Criteria}\label{ss:tcc}
We record here certain convenient tools for establishing tightness and proving weak convergence that will be used several times in this article.\\

\noindent{\bf Aldous' Tightness Criterion}.  The following criterion is a useful tool in proving tightness. Let $\{\mathbb{X}^{r}(\cdot): r \in \mathcal{R}\}$ be a collection of random variables in $\mathcal{D}([0,\infty):\RR)$. We will call a random time $\tau$ a $\mathbb{X}^{r}$-stopping time if for each $t \ge 0$, the event $\{\tau \le t\}$ lies in the $\sigma$-field $\sigma\left(\{\mathbb{X}^{r}(s) : s \le t\}\right)$. The collection $\{\mathbb{X}^{r}(\cdot): r \in \mathcal{R}\}$ is tight if and only if the following two conditions hold:
\begin{enumerate}
\item[(A1)] For each $t \ge 0$,
$$
\lim_{a \rightarrow \infty} \limsup_{r \rightarrow \infty} \mathbb{P}\left(|\mathbb{X}^{r}(t)| \ge a\right) = 0.
$$
\item[(A2)] For each $\epsilon, \delta, T >0$, there exists $\eta_0>0$ and $r_0 \in \mathcal{R}$ such that for any $0<\eta \le \eta_0$ and $r \ge r_0$, if $\tau$ is a $\mathbb{X}^{r}$-stopping time having a discrete, finite range satisfying $\tau \le T$, then
$$
\mathbb{P}\left(|\mathbb{X}^{r}(\tau+\eta) - \mathbb{X}^{r}(\tau)| \ge \delta \right) \le \epsilon.
$$
\end{enumerate}
 (cf. \cite[Theorem 16.10 and Corollary to Theorem 16.8]{billingsley2013})

The following elementary lemma will be used several times in the proofs. We provide the short proof for completeness.
\begin{lemma}\label{seqconv}
Suppose that $(\cls, {\mathbold d})$ is a Polish space, $S^0$ is an $\cls$-valued random variable, $\{S_{m}\}_{m\in \NN}$ is a sequence of $\cls$-valued random variables and $\epsilon^*>0$. For each $\epsilon \in (0, \epsilon^*]$, suppose that there is $b(\epsilon)>0$, a  $\cls$-valued random variable $S^{\epsilon}$, and a sequence of random variables $\{S^{\epsilon}_{m}\}_{m\in \NN}$, with $S^{\epsilon}_{m}$ and $S_{m}$ defined on the same probability space for each $m\in\NN$, such that the following hold:
\begin{enumerate} 
\item  $\limsup_{m\to\infty}\mathbb{P}\left({\mathbold d}(S^{\epsilon}_{m}, S_{m}\right) > b(\epsilon)) < b(\epsilon)$ {\color{black}for each $\epsilon\in(0,\epsilon^*]$} and $\lim_{\epsilon\searrow 0}b(\epsilon)=0$;
\item for each $\epsilon \in (0, \epsilon^*]$, $S^{\epsilon}_{m} \xrightarrow{d} S^{\epsilon}$ as $m\to \infty$;
\item $S^{\epsilon} \xrightarrow{d} S^0$ as $\epsilon \rightarrow 0$.
\end{enumerate}
Then $S_{m} \xrightarrow{d} S^0$ as $m\to \infty$.
\end{lemma}
\begin{proof}
	For an $\cls$ valued random variable $X$, denote its probability law as $\mu_X$.
	{\color{black}Let $d_{BL}$ denote the bounded-Lipschitz metric for Borel probability measures on $(\cls, {\mathbold d})$.}  Namely,
	for probability measures $\mu, \nu$ on {\color{black}$\cls$,} $d_{BL}(\mu, \nu) = \sup_g |\int g d\mu - \int g d\nu|$ where the supremum is taken over all Lipschitz functions $g:\cls \to \RR$ that are bounded by $1$ and whose Lipschitz constant is also bounded by $1$. To prove the lemma it suffices to show that $d_{BL}(\mu_{S_{m}}, \mu_{S_{0}}) \to 0$ as $m\to \infty$.
By triangle inequality {\color{black}and (1), for all $\epsilon\in(0,\epsilon^*]$ there exists $m_{\epsilon}$ such that for all $m\ge m_{\epsilon}$,}
	\begin{align*}
		d_{BL}(\mu_{S_{m}}, \mu_{S_{0}}) &\le d_{BL}(\mu_{S_{m}}, \mu_{S^{\epsilon}_{m}}) +
		d_{BL}(\mu_{S^{\epsilon}_{m}}, \mu_{S^{\epsilon}}) + d_{BL}(\mu_{S^{\epsilon}}, \mu_{S_{0}})\\
		&\le d_{BL}(\mu_{S_{m}}, \mu_{S^{\epsilon}_{m}}) + d_{BL}(\mu_{S^{\epsilon}}, \mu_{S_{0}}) + 2 b(\epsilon).
		\end{align*}
		Taking limit as $m \to \infty$ in the above {\color{black}gives that, for all $\epsilon\in(0,\epsilon^*]$,}
	$$\limsup_{m\to \infty}d_{BL}(\mu_{S_{m}}, \mu_{S_{0}}) \le d_{BL}(\mu_{S^{\epsilon}}, \mu_{S_{0}}) + 2 b(\epsilon).$$
	By sending $\epsilon\searrow 0$ {\color{black}and using (3) (which implies that
	$\lim_{\epsilon\searrow 0} d_{BL}(\mu_{S^{\epsilon}}, \mu_{S_{0}})=0$), the result follows.}
\end{proof}

\section{Proofs}\label{proofs}
In this section we prove the main theorems stated in Section \ref{s:main}.  {\color{black}To begin, we recall that} we refer to a job's remaining processing
time as its {\it size}{\color{black}.  In addition,} we refer to a job that arrived to the system after time zero as an {\it external job} and a job already in the system at time zero as an {\it initial job}.  Recall that the processing time distribution
does not depended on $r\in\clr$.  For each $r\in\clr$, we assume that the processing times are determined by a common sequence $\{v_i\}_{i=1}^{\infty}$
of independent and identically distributed random variables with common cumulative distribution function $F$ such that $v_i$ denotes the processing time
of the $i$-th external job arriving to the $r$-th SRPT queue. 
Beginning in Section \ref{sec:pfwork}, $F$ is assumed to satisfy \eqref{eq:regvar} henceforth.
For $r\in\clr$, $t\ge 0$, and $1\le i\le E^r(t)$ (resp.\ $1\le i\le \qq^r$), we recall that $v_i^r(t)$ $(resp.\ \breve v_i^r(t)$) denotes the
remaining processing time (or size) at time $t$ of the $i$-th external (resp.\ initial) job in the $r$-th SRPT queue.

We begin by proving some general comparison results for SRPT queueing systems that hold quite generally in that they do not require condition \eqref{eq:regvar}. These comparison results, besides being of independent interest, will be used in the proofs of our main theorems.  

\subsection{Intertwined SRPT queueing systems}\label{intsec}
In this section we consider  SRPT queues as introduced in Section \ref{ss:sd}. We fix $r$ and suppress it from the notation in this section. Also, as in Section \ref{ss:sd}, we assume that the service time distribution $F$ is continuous, but we do not require $\bar F$ to be regularly varying. In fact, even a finite mean is not needed.
Consider two SRPT queueing systems, say $S_1$ and $S_2$, with a common arrival process $E(\cdot)$ (which, as in Section \ref{ss:sd}, is a delayed renewal process), but with (possibly)  different initial conditions. For each $i=1,2$ and $t \ge 0$, let $n^{(i)}(t)$ be the number of jobs in system $S_i$ at time $t$ and let $\{v^i_{(j)}(t) : 1 \le j \le n^{i}(t)\}$ be the ordered collection of job sizes in system $S_i$ at time $t$, with $v^i_{(1)}(t)$ denoting the smallest job at time $t$, $v^i_{(2)}(t)$ denoting the second smallest job at time $t$, and so on.
For $i=1,2$, define $V^i_{0}(t) =0$ and $V^i_{j}(t) := \sum_{k=1}^{j}v^{i}_{(k)}(t),  \ 1 \le j \le n^{i}(t)$. For each $i=1,2$, the state of the system $S_i$ at time $t$ is completely described by the vector $\left(V^i_{0}(t), \dots,V^i_{n^{i}(t)}(t)\right)$. We say that $S_2$ is {\it intertwined} in $S_1$ at time $t$ if there exist integers $k(t) \ge 0$ and $l(t) \ge 1$ such that the following hold: (i) $S_1$ has $k(t) + l(t) - 1$ or $k(t) + l(t)$ jobs and $S_2$ has $k(t) + l(t) $ jobs at time $t$, (ii) $V^1_{j}(t) = V^2_{j}(t)$ for all $0 \le j \le k(t)$, and (iii) for every $1 \le l \le l(t)$, $V^1_{k(t) + l - 1}(t) < V^2_{k(t) + l}(t) < V^1_{k(t) + l}(t)$ (where, by convention, we take $V^1_{k(t) + l(t)}(t) = \infty$ if $S_1$ has $k(t) + l(t) - 1$ jobs at time $t$).  {\color{black}Thus, if $S_2$ is intertwined in $S_1$ at time $t$, we have
\begin{eqnarray*}
0<V_1^2(t)&=&V_1^1(t)<V_2^2(t)=V_2^1(t)<\dots<V_{k(t)}^2(t)=V_{k(t)}^1(t)\text{ and}\\
V_{k(t)}^1(t)&<&V_{k(t)+1}^2(t)<V_{k(t)+1}^1(t)<\dots\\
\dots&<&V_{k(t)+l(t)-1}^1(t)<V_{k(t)+l(t)}^2(t)<V_{k(t)+l(t)}^1(t).
\end{eqnarray*}
On intervals of time when no arrival or departure takes place in either system, each $V_j^i$ decreases at rate one as each server processes the work associated with the shortest job.  Hence, intertwinement is preserved on such intervals.  In the next two lemmas, we argue that intertwinement is preserved at times of a job arrival and a synchronous departure and swapped at times
of an asynchronous departure in that if $S_2$ is intertwined in $S_1$ immediately before such a departure, then $S_1$ is intertwined in $S_2$ immediately following such a departure. This property, in turn, is used to compare the queue length processes of $S_1$ and $S_2$.} A related, but different, notion for comparing the state of two queueing systems with a common arrival process, called work-dominance, was previously introduced by Smith \cite{S78} to establish optimality of SRPT.  

{\color{black}To begin, w}e have the following lemma which states that if one system is intertwined in the other immediately before a job arrival (which is the same for both systems) then
this  intertwining  is preserved immediately after the arrival.

\begin{lemma}\label{intlem}
Suppose $S_1$ and $S_2$ are two SRPT queueing systems with a common arrival process. Almost surely, if at some $t > 0$ a job arrives in the two systems, and $S_2$ is intertwined in $S_1$ just before time $t$, then $S_2$ is intertwined in $S_1$ at time $t$.
\end{lemma}

\begin{proof}
Denote the processing time of the entering job at time $t$ by $v^*$. Since $F$ is continuous, $P(v^*=c)=0$ for any $c\ge 0$. This property will be used without additional comments in many of the arguments below.
Note that if $v^* < v^2_{(k(t-) + 1)}(t-)$, then $k(t) = k(t-) + 1$
and $l(t) = l(t-)$.
In this case, for $1 \le l \le l(t)$, $V^i_{k(t) + l}(t) = V^i_{k(t-) + l}(t-) + v^*$ for $i = 1,2$ and as $S_2$ was intertwined in $S_1$ just before time $t$, we obtain $V^1_{k(t) + l - 1}(t) < V^2_{k(t) + l }(t) < V^1_{k(t) + l}(t)$ for all $1 \le l \le l(t)$, thus $S_2$ is intertwined in $S_1$ at time $t$. 
Otherwise,  $k(t) = k(t-)$ and $l(t) = l(t-) + 1$, which we assume henceforth.
For $1 \le l \le l(t)$, we consider the 
four possibilities as follows.\\
 (i) $v^* > \max\{v^1_{(k(t-) + l - 1)}(t-), v^2_{(k(t-) + l)}(t-)\}$,
 in which case,
 $$
 V^1_{k(t) + l - 1}(t) = V^1_{k(t) + l - 1}(t-)\qquad\hbox{and}\qquad V^2_{k(t) + l}(t) = V^2_{k(t) + l}(t-).$$
 Thus, by intertwinement before time $t$, $V^2_{k(t) + l}(t) > V^1_{k(t) + l - 1}(t)$.\\
  (ii) $v^1_{(k(t-) + l - 1)}(t-) < v^* < v^2_{(k(t-) + l)}(t-)$, in which case, $V^1_{k(t) + l - 1}(t) = V^1_{k(t) + l - 1}(t-)$ and $V^2_{k(t) + l}(t) = V^2_{k(t) + l - 1}(t-) + v^*$. As $v^* > v^1_{(k(t-) + l - 1)}(t-)= V^1_{k(t) + l - 1}(t-) - V^1_{k(t) + l - 2}(t-)$, $V^1_{k(t) + l - 2}(t-) < V^2_{k(t) + l-1}(t-)$ by intertwinement before time $t$, we obtain
\begin{align*}
V^2_{k(t) + l}(t) & = V^2_{k(t) + l - 1}(t-) + v^* > V^2_{k(t) + l - 1}(t-)  + (V^1_{k(t) + l - 1}(t-) - V^1_{k(t) + l - 2}(t-))\\
& > V^2_{k(t) + l - 1}(t-)  + (V^1_{k(t) + l - 1}(t-) - V^2_{k(t) + l-1}(t-))\\
& = V^1_{k(t) + l - 1}(t-) = V^1_{k(t) + l - 1}(t).
\end{align*}
(iii) $v^2_{(k(t-) + l)}(t-) < v^* < v^1_{(k(t-) + l - 1)}(t-)$, in which case, we have $V^1_{k(t) + l - 1}(t) = V^1_{k(t) + l - 2}(t-) + v^*$ and $V^2_{k(t) + l}(t) = V^2_{k(t) + l}(t-)$.
Also, since 
$k(t) = k(t-)$ and $l(t) = l(t-) + 1$, 
we have $l \ge 2$. As $v^* < v^1_{(k(t-) + l - 1)}(t-)= V^1_{k(t) + l - 1}(t-) - V^1_{k(t) + l - 2}(t-)$,
\begin{align*}
V^1_{k(t) + l - 1}(t)& = V^1_{k(t) + l - 2}(t-) + v^*\\
& <  V^1_{k(t) + l - 2}(t-) + (V^1_{k(t) + l - 1}(t-) - V^1_{k(t) + l - 2}(t-)) = V^1_{k(t) + l - 1}(t-).
\end{align*}
By intertwinement before time $t$, $V^2_{k(t) + l}(t-) > V^1_{k(t) + l-1}(t-)$. Hence,
$$
V^2_{k(t) + l}(t) = V^2_{k(t) + l}(t-) > V^1_{k(t) + l-1}(t-) > V^1_{k(t) + l - 1}(t).
$$
(iv) $v^* < \min\{v^1_{(k(t-) + l - 1)}(t-), v^2_{(k(t-) + l)}(t-)\}$, in which case, we have $V^1_{k(t) + l - 1}(t) = V^1_{k(t) + l - 2}(t-) + v^*$, $V^2_{k(t) + l}(t) = V^2_{k(t) + l-1}(t-) + v^*$, and $l\ge 2$. 
By intertwinement before time $t$,
$$
V^2_{k(t) + l}(t) = V^2_{k(t) + l-1}(t-) + v^* > V^1_{k(t) + l-2}(t-) + v^* = V^1_{k(t) + l - 1}(t).
$$
As, almost surely, the above are the only four possibilities, we have, almost surely, $V^2_{k(t) + l}(t) > V^1_{k(t) + l - 1}(t)$ for all $1 \le l \le l(t)$. By a symmetric argument, we obtain, almost surely, $V^1_{k(t) + l}(t) > V^2_{k(t) + l}(t)$ for all $1 \le l \le l(t)$. This completes the proof of the lemma.
\end{proof}

The following proposition compares the queue length processes for two SRPT systems started from intertwined configurations and having the same arrival process.
\begin{proposition}\label{intqlc}
Suppose $S_1$ and $S_2$ are two SRPT queueing systems with a common arrival process. Moreover, assume that $S_2$ is intertwined in $S_1$ at time zero. Denote the queue length process for $S_i$ by $Q_i(\cdot)$, $i=1,2$, {\color{black}and assume $Q_2(0) = Q_1(0) + 1$}. Then, almost surely, for any $t \ge 0$,
\begin{equation*}
Q_1(t) \le Q_2(t) \le Q_1(t) + 1.
\end{equation*}
\end{proposition} 
\begin{proof}
As $S_2$ is intertwined in $S_1$ at time zero, $Q_1(0) = k(0) + l(0) -1$ and $Q_2(0) = k(0) + l(0)$. Define $\tau^{as}_0=0$ and denote by $\tau^{as}_i$, $i \ge 1$, the time of the $i$-th asynchronous departure, i.e., when there is a departure from one system but not the other. For any $i \ge 0$, on the time interval $[\tau^{as}_i, \tau^{as}_{i+1})$, arrivals and departures happen at the same times (synchronously) from both systems. Clearly, if $S_2$ is intertwined in $S_1$ before a synchronous departure, then it remains so after the departure. Also, after any arrival, by Lemma \ref{intlem}, $S_2$ remains intertwined in $S_1$ if it were the case immediately before the arrival.
Thus, if $S_2$ is intertwined in $S_1$ at time $\tau^{as}_i$, then the same property is true for every $t \in [\tau^{as}_i, \tau^{as}_{i+1})$.
Then, for any $t \in [0, \tau^{as}_1)$, $Q_1(t) = k(t) + l(t) -1$ and $Q_2(t) = k(t) + l(t)$, and hence, $Q_2(t) - Q_1(t) = 1$. Moreover, as for any $t \in [0, \tau^{as}_1)$, $V^2_{k(t) + 1}(t) < V^1_{k(t) + 1}(t)$, the first asynchronous departure happens from $S_2$. Thus, $S_1$ is intertwined in $S_2$ at time $\tau^{as}_1$ (that is, the intertwinement order changes) and $Q_1(\tau^{as}_1) = Q_2(\tau^{as}_1) = k(\tau_1^{as}) + l(\tau_1^{as})$. By the same argument as above, we deduce that $S_1$ remains intertwined in $S_2$ on the time interval $[\tau^{as}_1, \tau^{as}_2)$ and  $Q_1(t) = Q_2(t) = k(t) + l(t)$ for all $t \in [\tau^{as}_1, \tau^{as}_2)$. At time $\tau^{as}_2$, departure happens from $S_1$ and the intertwinement order switches again at $\tau^{as}_2$, and so on. Thus, we conclude that $Q_1(t) = k(t) + l(t) -1, Q_2(t) = k(t) + l(t)$ for all $t \in [\tau^{as}_{2k}, \tau^{as}_{2k+1})$, $k \ge 0$, and $Q_1(t) = Q_2(t) = k(t) + l(t)$ for all $t \in [\tau^{as}_{2k+1}, \tau^{as}_{2k+2})$, $k \ge 0$. In particular, this proves the proposition.
\end{proof}

\subsection{Truncated SRPT queues}\label{ss:TrunSRPTQueues}
For each $r \in \clr$ and $a \in [0,\infty]$, we consider an SRPT queue with a thinned external arrival process $E^r_a(\cdot) := \sum_{i=1}^{E^r(\cdot)}\one_{[v_i \le a c^r]}$, which we refer to as the $r$-th $a$-truncated SRPT queue.  When the $i$-th external job arrives to the $r$-th SRPT queue, it is an external job
for the $r$-th $a$-truncated SRPT queue if and only if its processing time $v_i$ is less or equal to $a c^r$.  Similarly, jobs in the $r$-th $a$-truncated SRPT queue at time zero, namely the initial jobs, are those
that are initial jobs in the $r$-th SRPT queue such that $\check v_l^r\le ac^r$ and $1\le l\le \qq^r$.  Then the $r$-th $a$-truncated SRPT queue evolves in time in accordance with the SRPT
service discipline by preemptively serving the job with the shortest size first. For $r\in\clr$, $t \ge 0$ and $1\le i \le E_a^r(t)$, let $v_i^{r,a}(t)$ be the size at time $t$
in the $r$-th $a$-truncated SRPT queue of the $i$-th external arrival to the $r$-th $a$-truncated SRPT queue. Similarly, for $r\in\clr$, $t\ge 0$ and $1\le l\le \qq^r$, let $\breve{v}_l^{r,a}(t)$ be the size at time $t$ in the $r$-th $a$-truncated SRPT queue of the $l$-th initial job in the $r$-th $a$-truncated SRPT queue if $\breve{v}_l^r \le ac^r$, and zero if $\breve{v}_l^r > ac^r$ (the latter case is vacuous if $a = \infty$).  

Define for each $r\in \clr$, $a \in [0, \infty]$ and $t\ge 0$,
\begin{align*}
V_a^r(t) &:= \sum_{i=1}^{E^r(t)}v_i \one_{[v_i \le a c^r]},\\
\hat{V}_a^r(t) &:= \frac{1}{r}\sum_{i=1}^{E^r(r^2t)}v_i \one_{[v_i \le a c^r]} -  r \lambda^r t\mathbb{E}(v \one_{[v \le a c^r]}),\\
X_a^r(t) &:=  \frac{1}{r} \sum_{l=0}^{\qq^r} \breve{v}_l^r \one_{[\breve{v}_l^r \le ac^r]}  +  \frac{1}{r}V_a^r(r^2t) - rt,\\
Y_a^r(t) &:= \Gamma[X_a^r](t).
\end{align*}
Also, for $r\in\clr$, $a \in [0,\infty]$ and $t\ge 0$,
$$
\tQ^r_a(t) := \frac{c^r}{r}\sum_{l=1}^{\qq^r} \delta^{+}_{\breve{v}_l^{r,a}(r^2t)/c^r} + \frac{c^r}{r}\sum_{i=1}^{E^r_a(r^2t)}\delta^{+}_{v_i^{r,a}(r^2t)/c^r}
$$
denotes the scaled measure describing the state of the  $r$-th $a$-truncated SRPT queue at time $r^2t$ and $Q_a^r(t) := \langle \one,\tQ^r_a(t) \rangle$ denotes the scaled queue length in
the $r$-th $a$-truncated SRPT queue at time $r^2t$.  Recall that, for each $r\in \clr$, $a \in [0, \infty]$ and $t\ge 0$, $Z_a^r(t)$ and $W_a^r(t)$ are defined  in \eqref{zqldef} and \eqref{trunkworkdef} respectively.

We have elected to state the results in this section for truncated SRPT queues in terms of scaled processes defined above.  However, since they hold for each $r\in\clr$, one can obtain unscaled versions from these.  Also, as in Section \ref{intsec}, $F$ is required to be continuous, but $\bar F$ is not required to be regularly varying. The following proposition records a key observation comparing the process $\langle \chi \one_{[0,a]}, \tQ_y^r(\cdot)\rangle$ with $Y_a^r(\cdot)$ and $\langle \one_{[0,a]}, \tQ_y^r(\cdot)\rangle$ with $Q_a^{r}(\cdot)$ for $a \le y \le \infty$.

\begin{proposition}\label{comp}
For any $r \in \clr$, $a \in  (0, \infty)$, $a \le y\le\infty$, and $t \ge 0$, we have, almost surely,
\begin{align}
 Y_a^r(t) &\le \langle \chi \one_{[0,a]}, \tQ_y^r(t)\rangle\le Y_a^r(t) +\frac{ac^r}{r},\label{eq:comp4}\\
Q_a^r(t) &\le \langle \one_{[0,a]}, \tQ_y^r(t)\rangle \le Q_a^r(t) +\frac{c^r}{r}.\label{eq:comp5}
\end{align}
In particular, for any $r \in \clr$, $a \in  (0, \infty)$ and $t \ge 0$, we have, almost surely,
\begin{align}
Y_a^r(t) &\le W_a^r(t) \le Y_a^r(t) + \frac{ac^r}{r},\label{eq:comp2}\\
Q_a^{r}(t) &\le Z_a^{r}(t) \le Q_a^{r}(t) + \frac{c^r}{r}.\label{eq:comp3}
\end{align}
Moreover, almost surely, $W_{0}^r(t) = Y_{0}^r(t) = 0, Z_{0}^{r}(t) = Q_{0}^{r}(t) = 0$, $W_{\infty}^r(t) = Y_{\infty}^r(t)$, and $Z_{\infty}^{r}(t) = Q_{\infty}^{r}(t)$ for any $r\in\clr$ and $t \ge 0$.
\end{proposition}

\begin{proof} Fix $r\in\clr$.
Note that, by definition, $W^r_0(t) = Z^{r}_0(t)=Q^{r}_0(t) = 0$ for all $t \ge 0$. Moreover, almost surely, $X_0^r(t)=-rt$ for all $t \ge 0$ and hence $Y^r_0(t) = \Gamma[X^r_0](t) = 0$
for all $t \ge 0$. Also, as $\tQ^r_{\infty}(t) = \tZ^r(t)$ for all $t \ge 0$, $W_{\infty}^r(t) = Y_{\infty}^r(t)$ and $Z_{\infty}^{r}(t) = Q_{\infty}^{r}(t)$ for all $t \ge 0$.
Thus, the assertions in the last line of the lemma 
hold. Also, for each $a\in(0,\infty)$, \eqref{eq:comp2} follows from \eqref{eq:comp4} and \eqref{eq:comp3} follows from \eqref{eq:comp5} upon setting $y=\infty$, since $\tZ^r(\cdot)=\tQ_{\infty}^r(\cdot)$. So it suffices to verify \eqref{eq:comp4} and \eqref{eq:comp5}.

Fix $a \in (0, \infty)$ and $a\le y\le\infty$.  Define {\color{black}the} stopping times $\sigma_{-1}=0$, and for $k \in \bbZp$,
$$
\sigma_{2k} :=\inf\{s \ge \sigma_{2k-1}: \langle \chi \one_{[0,a]}, \tQ_y^r(s)\rangle =0\}, \quad \quad \sigma_{2k+1} :=\inf\{s \ge \sigma_{2k}: \langle \chi \one_{[0,a]}, \tQ_y^r(s)\rangle >0\}.
$$
To show \eqref{eq:comp4} and \eqref{eq:comp5},
we proceed by induction. Observe that, by definition, $Y_a^r(0)=\langle \chi \one_{[0,a]}, \tQ_y^r(0)\rangle$ and $Q_a^r(0)=\langle \one_{[0,a]}, \tQ_y^r(0)\rangle$ since $a\le y$.
Thus, \eqref{eq:comp4} and \eqref{eq:comp5} hold on $[0,\sigma_{-1}]$. 

First consider the case $\langle \chi \one_{[0,a]}, \tQ_y^r(0)\rangle =0$ (which implies that $\langle \one_{[0,a]}, \tQ_y^r(0)\rangle =0$). Then, $\sigma_0=\sigma_{-1}=0$ and $Y_a^r(t) = \langle \chi \one_{[0,a]}, \tQ_y^r(t)\rangle =0$ for all $t \in [0, \sigma_1)$. The map $t \mapsto \langle \chi \one_{[0,a]}, \tQ_y^r(t)\rangle$ increases at $t = \sigma_1$ due to one of the following two events: (i) an external job 
with processing time less or equal to $ac^r$ arrives 
to the system at time $r^2\sigma_1$ or (ii) an initial job with initial size in $(ac^r, yc^r]$ or an external job with processing time
in $(ac^r, yc^r]$ that arrived during the time interval $(0,r^2\sigma_1)$,
in course of getting served, has its size drop to $ac^r$ at time $r^2\sigma_1$. If (i) occurs, $Q_a^r(\sigma_1) = \langle \one_{[0,a]}, \tQ_y^r(\sigma_1)\rangle = \frac{c^r}{r}$, $Y_a^r(\sigma_1) = \langle \chi \one_{[0,a]}, \tQ_y^r(\sigma_1)\rangle \le \frac{ac^r}{r}$. If (ii) occurs, $Q_a^r(\sigma_1) = 0$, $\langle \one_{[0,a]}, \tQ_y^r(\sigma_1)\rangle = \frac{c^r}{r}$,$Y_a^r(\sigma_1) = 0$, and $\langle \chi \one_{[0,a]}, \tQ_y^r(\sigma_1)\rangle  = \frac{ac^r}{r}$. Thus, when $\langle \chi \one_{[0,a]}, \tQ_y^r(0)\rangle =0$, \eqref{eq:comp4} and \eqref{eq:comp5} hold for all $t \in [0, \sigma_1]$.

Suppose that for some $k \in \bbZp$ \eqref{eq:comp4} and \eqref{eq:comp5} hold for all $t \in [0, \sigma_{2k-1}]$ and $$\langle \chi \one_{[0,a]}, \tQ_y^r(\sigma_{2k-1})\rangle > 0$$
(which implies that $\langle \one_{[0,a]}, \tQ_y^r(\sigma_{2k-1})\rangle >0$). We first show that \eqref{eq:comp4} and \eqref{eq:comp5} hold for all $t \in (\sigma_{2k-1}, \sigma_{2k}]$. By virtue of the SRPT dynamics, no job in the $r$-th $y$-truncated SRPT queue at time $r^2\sigma_{2k-1}$ of size greater than $ac^r$ at time $r^2\sigma_{2k-1}$ is served in the $r$-th $y$-truncated SRPT queue during the time interval $[r^2\sigma_{2k-1}, r^2\sigma_{2k})$. Consequently, for any $t \in (\sigma_{2k-1}, \sigma_{2k})$, the 
following four properties are equivalent: (a) $\langle \chi \one_{[0,a]}, \tQ_y^r(t)\rangle - \langle \chi \one_{[0,a]}, \tQ_y^r(t-)\rangle>0$; (b) $E_a^r(r^2t) - E_a^r(r^2t-)>0$; (c)  $X_a^r(t) - X_a^r(t-)>0$; (d) $Y_a^r(t) - Y_a^r(t-)>0$ and, 
when these equivalent properties hold, $\langle \chi \one_{[0,a]}, \tQ_y^r(t)\rangle - \langle \chi \one_{[0,a]}, \tQ_y^r(t-)\rangle =  Y_a^r(t) - Y_a^r(t-)$.
This also shows that for $t\in[\sigma_{2k-1},\sigma_{2k})$ such that $Y_a^r(t)=0$ and $s\in [t, \inf\{u \ge t: Y_a^r(u) > 0\}\wedge \sigma_{2k}{\color{black})}$, $Y_a^r(s)=0$ and $\langle \chi \one_{[0,a]}, \tQ_y^r(s)\rangle=\langle \chi \one_{[0,a]}, \tQ_y^r(t)\rangle -r(s-t)$.
Moreover, for $t \in [\sigma_{2k-1},\sigma_{2k})$ such that $0< Y_a^r(t) \le \langle \chi \one_{[0,a]}, \tQ_y^r(t)\rangle$ and $s \in [t, \inf\{u \ge t: Y_a^r(u) = 0\}]$,
$$
\langle \chi \one_{[0,a]}, \tQ_y^r(s)\rangle - \langle \chi \one_{[0,a]}, \tQ_y^r(t)\rangle = \frac{1}{r}(V_a^r(r^2s) - V_a^r(r^2t)) - r(s-t) = Y_a^r(s) - Y_a^r(t).
$$
From these observations, we conclude that $t \mapsto \langle \chi \one_{[0,a]}, \tQ_y^r(t)\rangle - Y_a^r(t)$ is nonincreasing on the interval $[\sigma_{2k-1}, \sigma_{2k}]$ and decreases only on the set $\{ u \in [\sigma_{2k-1}, \sigma_{2k}]: Y_a^r(u) = 0\}$. This also implies that either $\langle \chi \one_{[0,a]}, \tQ_y^r(u)\rangle = Y_a^r(u)$ for all $u \in (\sigma_{2k-1}, \sigma_{2k}]$ or the first $t\ge\sigma_{2k-1}$ for which
$\langle \chi \one_{[0,a]}, \tQ_y^r(t)\rangle = Y_a^r(t)$ corresponds to $\sigma_{2k}$ when $\langle \chi \one_{[0,a]}, \tQ_y^r(\sigma_{2k})\rangle = Y_a^r(\sigma_{2k})=0$. We conclude that for any $t \in [\sigma_{2k-1}, \sigma_{2k}]$,
\begin{align*}
0&=\langle \chi \one_{[0,a]}, \tQ_y^r(\sigma_{2k})\rangle - Y_a^r(\sigma_{2k}) \le \langle \chi \one_{[0,a]}, \tQ_y^r(t)\rangle - Y_a^r(t)\\
& \le \langle \chi \one_{[0,a]}, \tQ_y^r(\sigma_{2k-1})\rangle - Y_a^r(\sigma_{2k-1}) \le \frac{ac^r}{r},
\end{align*}
where the last inequality holds by the induction hypothesis. Hence, \eqref{eq:comp4} holds for all $t \in (\sigma_{2k-1}, \sigma_{2k}]$. 

Now we show that \eqref{eq:comp5} holds for all $t \in (\sigma_{2k-1}, \sigma_{2k}]$. If $k \in\NN$, then, by definition of
$\sigma_{2k-1}$, $\langle\one_{[0,a]}, \tQ_y^r(\sigma_{2k-1}-)\rangle=0$, and so,
using the induction hypothesis, $Q_a^r(\sigma_{2k-1}-)=0$. Moreover, the arrival times and processing times of all external jobs with processing time
less than or equal to $ac^r$ into both the $r$-th $a$-truncated SRPT queue and the $r$-th $y$-truncated SRPT queue
on the time interval $[r^2\sigma_{2k-1}, r^2\sigma_{2k}]$
are common to both systems. Further, no job in the $r$-th $y$-truncated SRPT queue at time $r^2\sigma_{2k-1}$ of size greater than $ac^r$ at time $r^2\sigma_{2k-1}$ is served in the $r$-th $y$-truncated SRPT queue during the time interval $[r^2\sigma_{2k-1}, r^2\sigma_{2k}]$. Thus, the processes $t \mapsto Q_a^r(t)$ and $t \mapsto \langle \one_{[0,a]}, \tQ_y^r(t)\rangle$ on the time interval $[r^2\sigma_{2k-1}, r^2\sigma_{2k}]$ can be identified with the (scaled) queue length processes of two $r$-th $a$-truncated SRPT queueing systems having the same arrival process, denoted respectively by $S_1^r$ and $S_2^r$, started at time zero and observed till $S_2^r$ has zero jobs. If $k=0$ or if the increase in $t \mapsto \langle \chi \one_{[0,a]}, \tQ_y^r(t)\rangle$ at time $t= \sigma_{2k-1}$ happens due to the arrival of an external job with processing time less than or equal to $ac^r$, then $\langle  \one_{[0,a]}, \tQ_y^r(\sigma_{2k-1})\rangle = Q_a^r(\sigma_{2k-1})$. Thus, in this case, $S_1^r$ and $S_2^r$ start with the same configuration and hence, $Q_a^r(t) = \langle \one_{[0,a]}, \tQ_y^r(t)\rangle$ for all $t \in [\sigma_{2k-1}, \sigma_{2k}]$. On the other hand, the increase in $t \mapsto \langle \chi \one_{[0,a]}, \tQ_y^r(t)\rangle$ at time $t= \sigma_{2k-1}$ may happen due to a job present in the system at a time $s<r^2\sigma_{2k-1}$, with its size in the range $(ac^r, yc^r]$ at time $s$,
getting served in the $y$-th truncated queue and having its size drop to $ac^r$ at time $r^2\sigma_{2k-1}$. {\color{black}In this case, $S_2^r$ starts with one job of size $ac^r$ and $S_1^r$ starts with zero jobs. Hence, $S_2^r$ is intertwined in $S_1^r$ at time zero in the sense of Subsection \ref{intsec} with $k(0)=0$ and $l(0)=1$, and $S_2^r$ has one more job at time zero than $S_1^r$.} By Proposition \ref{intqlc}, for any $t \in [\sigma_{2k-1}, \sigma_{2k}]$,
$$
Q_a^r(t) \le \langle \one_{[0,a]}, \tQ_y^r(t)\rangle \le Q_a^r(t) + \frac{c^r}{r}.
$$
Hence, \eqref{eq:comp5} holds for all $t \in (\sigma_{2k-1}, \sigma_{2k}]$.

To see that \eqref{eq:comp4} and \eqref{eq:comp5} hold for all $t \in (\sigma_{2k}, \sigma_{2k+1}]$, first note that $Y_a^r(t) = \langle \chi \one_{[0,a]}, \tQ_y^r(t)\rangle =0$ for all $t \in (\sigma_{2k}, \sigma_{2k+1})$. Moreover, observe that either $Q_a^r(\sigma_{2k+1}) = \langle \one_{[0,a]}, \tQ_y^r(\sigma_{2k+1})\rangle = \frac{c^r}{r}$ and $Y_a^r(\sigma_{2k+1}) = \langle \chi \one_{[0,a]}, \tQ_y^r(\sigma_{2k+1})\rangle \le \frac{ac^r}{r}$, or $Q_a^r(\sigma_{2k+1}) = 0, \langle \one_{[0,a]}, \tQ_y^r(\sigma_{2k+1})\rangle = \frac{c^r}{r}$, $Y_a^r(\sigma_{2k+1}) = 0$, and $\langle \chi \one_{[0,a]}, \tQ_y^r(\sigma_{2k+1})\rangle  = \frac{ac^r}{r}$.
In both cases, \eqref{eq:comp4} and \eqref{eq:comp5} hold for all $t \in (\sigma_{2k}, \sigma_{2k+1}]$.

Thus, by induction, \eqref{eq:comp4} and \eqref{eq:comp5} hold for all $t\in[0,\lim_{k\to\infty}\sigma_{2k})$.  To complete
the proof, we show that $\lim_{k\to\infty}\sigma_{2k}=\infty$. Suppose first that $\mathbb{E}(v\one_{[v \le ac^r]})>0$. For each $k \in \bbZp$, let $v^*_k$ be the processing time of the first external job to arrive to the $r$-th $y$-truncated SRPT queue
after time $\sigma_{2k}$. Then it is easy to see that for each $k \in \bbZp$, $\sigma_{2k+2} - \sigma_{2k+1} \ge r^{-2}v^*_k\one_{[v^*_k \le ac^r]}$. As $\{v^*_k\one_{[v^*_k \le ac^r]}\}_{k \ge 0}$ is a sequence of independent and identically distributed random variables, where each element has the same distribution as $v\one_{[v \le ac^r]}$, and since $\mathbb{E}(v\one_{[v \le ac^r]})>0$, almost surely, 
$$
\lim_{k \rightarrow \infty} \sigma_{2k}  \ge \lim_{k \rightarrow \infty} \sum_{j=0}^{k-1}(\sigma_{2j+2} - \sigma_{2j+1}) \ge r^{-2}\lim_{k \rightarrow \infty} \sum_{j=0}^{2k-2}v^*_j\one_{[v^*_j \le ac^r]} = \infty.
$$
If $\mathbb{E}(v\one_{[v \le ac^r]})=0$ and $\mathbb{E}(v\one_{[v \le yc^r]})>0$, then almost surely, no external job with processing time less or equal to $ac^r$ arrives into the system and thus almost surely, $\sigma_{2k+2} - \sigma_{2k+1} = r^{-2}ac^r$ for all $k \in \bbZp$, and hence $\lim_{k \rightarrow \infty} \sigma_{2k} = \infty$, as desired. Finally, if
$\mathbb{E}(v\one_{[v \le yc^r]}) = 0$, which implies that $\mathbb{E}(v\one_{[v \le ac^r]})=0$ since $a\le y$, then there exists $k_0 \in \bbZp$ such that $Q^r_y(\sigma_{2k_0})=0$ and thus $\sigma_{2k_0+1} = \infty$. Hence \eqref{eq:comp4} and \eqref{eq:comp5} hold for all $t \in [0,\infty)$. 
\end{proof}
The following lemma compares queue length processes for truncated SRPT queues with different truncations.

\begin{lemma}\label{qlxy}
For all $r\in\clr$, $0 \le x \le y \le \infty$ and $t \ge 0$,
$$
0 \le Q_{y}^r(t) - Q_{x}^r(t) \le \frac{c^r}{r} + x^{-1}Y_{y}^r(t).
$$
\end{lemma}

\begin{proof} Fix $r\in\clr$, $0 \le x \le y \le \infty$ and $t \ge 0$.
Note that, almost surely,
\begin{align}\label{newqdel}
0 \le Q_{y}^r(t) - Q_{x}^r(t) 
&= \int_{0}^{x}\tQ^r_{y}(t)(dz) - {\color{black} Q_{x}^r(t)} + \int_{x}^{y}\tQ^r_{y}(t)(dz)\\
&={\color{black}\langle \one_{[0,x]}, \tQ^r_{y}(t)\rangle - Q_{x}^r(t)}+ \int_{x}^{y}\tQ^r_{y}(t)(dz).\nonumber
\end{align}
By \eqref{eq:comp5} in Proposition \ref{comp} with $a=x$, almost surely,
$$
0 \le {\color{black}\langle \one_{[0,x]}, \tQ^r_{y}(t)\rangle - Q_{x}^r(t)} \le \frac{c^r}{r}.
$$
Using this observation in \eqref{newqdel}, we obtain
\begin{eqnarray}\label{un1}
0 &\le& Q_{y}^r(t) - Q_{x}^r(t) \le \frac{c^r}{r} + \int_{x}^{y}\tQ^r_{y}(t)(dz)\\
&\le& \frac{c^r}{r} + x^{-1}\int_{x}^{y}z\tQ^r_{y}(t)(dz) \le \frac{c^r}{r} + x^{-1}Y_{y}^r(t),\nonumber
\end{eqnarray}
as desired.\end{proof}

\subsection{Proof of Theorem \ref{workless}}\label{sec:pfwork}

The following lemma is a functional central limit theorem for $\{X_{\cdot}^r(\cdot): r\in\clr\}$, which is used below in conjunction with the result in Proposition \ref{comp} to prove Theorem \ref{workless}. For this, recall the definition of {\color{black}$X_a^r$ and $X_a$, $a\in[0,\infty]$, from Section \ref{ss:TrunSRPTQueues} and \eqref{eq:xadefn} respectively}.

\begin{lemma}\label{jointZ}
There exists a probability space on which we are given a Brownian motion $B$ and a $\mathcal{C}([0,\infty):\RRp) \times \RRp$ valued random variable $(\xi(\cdot), \xi(\infty))$ independent of $B$, with same distribution as $(w^*(\cdot),w^*(\infty))$, such that for any $k\in\NN$ and any $0 < a_1 <\dots < a_k \le \infty$, as $r \rightarrow \infty$,
$$
(X_{a_1}^r(\cdot), \dots, X_{a_k}^r(\cdot)) \xrightarrow{d} (X_{a_1}(\cdot), \dots, X_{a_k}(\cdot))
$$
in $\mathcal{C}([0,\infty):\RR^k)$.
\end{lemma}

\begin{proof}
Note that for any $r\in\clr$, $a\in(0,\infty)$ and $t\ge 0$,
\begin{eqnarray}
X_{\infty}^r(t) &=& X_{\infty}^r(0)+ \frac{1}{r}V_{\infty}^r(r^2t) - rt = X_{\infty}^r(0)+\hat{V}_{\infty}^r(t) +r(\rho^r - 1)t,\label{infwork}\\
X_{a}^r(t) &=& X_{a}^r(0)+ \frac{1}{r}V_a^r(r^2t) - rt = X_{a}^r(0)+\hat{V}_a^r(t) +r(\rho^r_{ac^r} - 1)t,\label{Zdec}
\end{eqnarray}
where $X_{a}^r(0) = \frac{1}{r} \sum_{l=0}^{\qq^r} \breve{v}_l^r \one_{[\breve{v}_l^r \le ac^r]} $.
Note that for any $r\in\clr$, $a \in (0,\infty)$, and $t\ge 0$, 
\begin{multline}
r(\rho^r_{ac^r} - 1) = r(\rho^r_{ac^r} - \rho^r) + r(\rho^r - 1) = -r\lambda^r \mathbb{E}(v \one_{[v > ac^r]}) + r(\rho^r - 1)\\
= -\lambda^r\frac{\mathbb{E}(v \one_{[v > ac^r]})}{\mathbb{E}(v \one_{[v > c^r]})}r \mathbb{E}(v \one_{[v > c^r]}) + r(\rho^r - 1)\\
 = -\lambda^r\frac{\mathbb{E}(v \one_{[v > ac^r]})}{\mathbb{E}(v \one_{[v > c^r]})} \frac{r}{S(c^r)} + r(\rho^r - 1) = 
 -\lambda^r\frac{\mathbb{E}(v \one_{[v > ac^r]})}{\mathbb{E}(v \one_{[v > c^r]})}  + r(\rho^r - 1),
 \label{eq:htcr}
\end{multline}
where we have used \eqref{eq:Scr} in the final equality. By  \eqref{karama},
\begin{equation*}
\lim_{r \rightarrow \infty}\frac{\mathbb{E}(v \one_{[v > ac^r]})}{\mathbb{E}(v \one_{[v > c^r]})} = \frac{1}{a^p}.
\end{equation*}
Using this and assumption \eqref{eq:assuht} in the above equation, we obtain that for each $a\in(0,\infty)$,
\begin{equation}\label{rhor}
r(\rho^r_{ac^r} - 1) \rightarrow \kappa - \frac{\lambda}{a^p}, \qquad \text{ as } r \rightarrow \infty.
\end{equation}
For $a\in(0,\infty)$, let $m_a^r=\mathbb{E}(v \one_{[v \le ac^r]})$ and $(s_a^r)^2=\text{Var}(v \one_{[v \le ac^r]})$.
Then, finiteness of the second moment of $v$ and the fact that $\lim_{r\to\infty}c^r=\infty$ give that,
for $a\in(0,\infty)$, $\lim_{r\to\infty}m_a^r=\mathbb{E}(v)$, $\lim_{r\to \infty} (s_a^r)^2=\text{Var}(v)$, and for each $\delta >0$
$$
\lim_{r\to \infty} \mathbb{E} \left[ (v \one_{[v \le ac^r]}-m_a^r)^2 \one_{|v \one_{[v \le ac^r]}-m_a^r|> r\delta}\right]=0.
$$
Thus, by Proposition \ref{prop:fclt}, for each $a\in(0,\infty)$,
$
\hat{V}_a^r(\cdot) \xrightarrow{d} \sigma B(\cdot)
$
where $\sigma^2 = \lambda \operatorname{Var}(v) + (\mathbb{E}(v))^2\lambda^3\sigma_A^2 = \lambda \operatorname{Var}(v) + \lambda\sigma_A^2$ and $B$ is a standard Brownian motion. 
Note that, from \eqref{eq:assuinitcond} and assumed mutual independence in Section \ref{sec:mathfram}, we in fact have that, for each $a\in(0,\infty)$,
\begin{equation}\label{martconv}
(X_{\cdot}^r(0), \hat{V}_a^r(\cdot)) \xrightarrow{d} (\xi(\cdot), \sigma B(\cdot))
\end{equation}
in $\mathcal{D}([0,\infty): \RRp)\times \mathcal{D}([0,\infty): \RR)$, where $\xi$ is distributed as $w^*$ and is independent of $B$.

For each $0 < a < b \le \infty$,
$$
\hat{V}_b^r(t) - \hat{V}_a^r(t) = \frac{1}{r}\sum_{i=1}^{E^r(r^2t)}v_i \one_{[ac^r < v_i \le b c^r]} -  r \lambda^rt\mathbb{E}(v \one_{[ac^r < v \le b c^r]}).
$$
Note that by the finiteness of the second moment of $v$ and $\lim_{r\to\infty}c^r=\infty$, for each $0 < a < b \le \infty$,
as $r \rightarrow \infty$,
\begin{equation}\label{eq:vabcr}
\mathbb{E}\left(v \one_{[ac^r < v \le b c^r]}\right) \rightarrow 0 \quad \text{and} \quad \operatorname{Var}(v \one_{[ac^r < v \le b c^r]}) \le \mathbb{E}\left(v^2 \one_{[ac^r < v]}\right) \rightarrow 0.
\end{equation}
Thus, by Proposition \ref{prop:fclt}, for each $0 < a < b \le \infty$,
$$
\hat{V}_b^r(\cdot) - \hat{V}_a^r(\cdot) \xrightarrow{d} 0 \text{ as } r \rightarrow \infty.
$$
This, combined with \eqref{infwork}, \eqref{Zdec} and \eqref{rhor}, gives for each $0 < a < b \le \infty$,
\begin{equation}\label{sup}
(X_b^r(\cdot) - X_b^r(0))  - (X_a^r(\cdot)- X_a^r(0)) + \left(\frac{\lambda}{b^p} - \frac{\lambda}{a^p}\right)(\cdot) \xrightarrow{d} 0 \quad \text{ as } r \rightarrow \infty,
\end{equation}
where $\lambda/b^p$ is taken to be zero if $b = \infty$.

The above convergence together with \eqref{rhor} shows that, for each $i = 1, \ldots, k$
$$X^r_{a_i}(\cdot) = X^r_{a_i}(0) + \hat V^r_{a_i}(\cdot) + (\kappa - \frac{\lambda}{a_i^p})\iota(\cdot) + \eta_i^r(\cdot),$$
where $\eta_i^r(\cdot)\xrightarrow{d} 0 $ as $r\to \infty$ for each $i$.
The result now follows on combining the above convergence with \eqref{martconv}.
\end{proof}

\begin{proof}[Proof of Theorem \ref{workless}]
{\color{black}Lemma \ref{jointZ}, continuity of the Skorohod map $\Gamma$ and the continuous mapping theorem, imply that for all $k\in\NN$
and $0\le a_1<a_2<\cdot<a_k\le\infty$, $(Y_{a_1}^r,Y_{a_2}^r,\dots,Y_{a_k}^r)\xrightarrow{d} (W_{a_1},W_{a_2},\dots,W_{a_k})$.}
The theorem follows from {\color{black}this,} Proposition \ref{comp} and
{\color{black}\eqref{limitcroverr}.}
\end{proof}

\subsection{Proof of Theorem \ref{workfn}}\label{ss:pfworkfn}
Before proceeding, the reader may wish to review the overview of the proof of Theorem \ref{workfn} given in {\color{black}Section \ref{ss:methods}}.
We begin by establishing the result in Lemma \ref{parts} below as a elementary consequence of integration by parts.
In what follows, {\color{black}for $0\le \delta < M<\infty$}, we will write `$\int_{\delta}^M$' to denote integration over the interval $(\delta,M]$. We will also write for any function $h: (\delta,M] \rightarrow \mathbb{R}$ and any $\delta\ge 0$, $h(\deltap) := \lim_{x \searrow \delta} h(x)$, whenever this limit exists.

\begin{lemma}\label{parts}
Suppose that $0<\delta<M< \infty$ and $f: (\delta,M] \rightarrow \mathbb{R}$ is a $C^1$ function such that $f(\deltap)$ and $f'(\deltap)$ exist.
Then, writing $g(x) = f(x)/x$ for $x\in(\delta,M]$, for any $r\in\clr$ and $t \ge 0$, the following holds:
$$
\int_{\delta}^Mf(x) \tZ^r(t)(dx) = - \int_{\delta}^Mg'(x)W_x^r(t)dx + g(M)W_M^r(t) - g(\deltap)W_{\delta}^r(t).
$$
\end{lemma}

\begin{proof} Fix $r\in\clr$ and $t\ge 0$.
Define the finite nonnegative Borel measure $\mu^r(t)$ on $\RRp$ by $\mu^r(t)(dx) := x \tZ^r(t)(dx)$ for $x\in\RRp$. Then, for $0\le a<b$, $\mu^r(t)(a,b] = W_b^r(t) - W_a^r(t)$. Therefore,
\begin{eqnarray*}
\int_{\delta}^Mf(x) \tZ^r(t)(dx) &=& \int_{\delta}^Mg(x) \mu^r(t)(dx) = \int_{\delta}^M\left(\int_{\delta}^xg'(y)dy + g(\deltap)\right) \mu^r(t)(dx)\\
 &=& \int_{\delta}^M\int_{y}^{M}\mu^r(t)(dx)g'(y)dy + g(\deltap)\mu^r(t)(\delta,M]\\
 &=& \int_{\delta}^M\mu^r(t)(y,M]g'(y)dy + g(\deltap)\mu^r(t)(\delta,M]\\
& =&\int_{\delta}^M(W_M^r(t) - W_y^r(t))g'(y)dy + g(\deltap) (W_M^r(t) - W_{\delta}^r(t))\\
 & = &-\int_{\delta}^MW_y^r(t)g'(y)dy + W_M^r(t)(g(M) - g(\deltap))\\
 &\qquad& + g(\deltap) (W_M^r(t) - W_{\delta}^r(t))\\
& =& - \int_{\delta}^Mg'(y)W_y^r(t)dy + g(M)W_M^r(t) - g(\deltap)W_{\delta}^r(t),
\end{eqnarray*}
which proves the lemma.
\end{proof}

Next, the result in Lemma \ref{parts}, along with tightness arguments, is used to establish Theorem \ref{convplus}, which gives convergence in distribution to the desired limit for certain
compactly supported functions with support bounded away from zero.

\begin{theorem}\label{convplus}
Suppose that $J\in\NN$, $0<a_1 < b_1 \le a_2 < b_2 \dots \le a_J < b_J < \infty$, and $f: [0, \infty) \rightarrow \mathbb{R}$ is a $C^1$ function on $(a_j,b_j]$
for each $1 \le j \le J$ and zero on $\left(\cup_{j=1}^J(a_j,b_j]\right)^c$. Also, assume $\lim_{x\searrow a_j}f(x)$ and $\lim_{x\searrow a_j}f'(x)$ exist for each $1 \le j \le J$.
Then, writing $g(x) = f(x)/x$ for $x\in(0,\infty)$, as $r\to\infty$,
\begin{equation}\label{eq:limpt}
\int_0^{\infty}f(x) \tZ^r(\cdot)(dx) \xrightarrow{d} \sum_{j=1}^J\left(- \int_{a_j}^{b_j}g'(x)W_x(\cdot)dx + g(b_j)W_{b_j}(\cdot) - \lim_{x\searrow a_j}g(x)W_{a_j}(\cdot)\right).
\end{equation}
in  $\mathcal{D}([0,\infty) : \mathbb{R})$. 
The limiting process defined by the right side of \eqref{eq:limpt}, in fact, has sample paths in $\mathcal{C}([0,\infty) : \mathbb{R})$ almost surely.
\end{theorem}
\begin{remark}
	\label{rem:takez}
The proof of Theorem \ref{workfn} will show that we can also take $a_1=0$ in Theorem \ref{convplus}. See Remark \ref{rem:aeqze} for details.
\end{remark}
\begin{proof}[Proof of Theorem \ref{convplus}]
 We will prove the theorem for $J=1$. The proof for $J \ge 2$ follows along the same lines (with more cumbersome notation) and is, therefore, omitted. We will write the interval $(a_1, b_1]$ as $(\delta, M]$ with $0< \delta < M < \infty$. Assume $f$ is not identically zero (otherwise the result is trivial). \\
 \emph{Proof of Tightness:} We will use Aldous' tightness criterion stated in Section \ref{ss:tcc}. Note that, for $r\in\clr$ and $t\ge 0$,
$$
\left|\int_{\delta}^Mf(x) \tZ^r(t)(dx) \right| \le \sup_{z \in [\delta, M]} |g(z)|\int_{\delta}^Mx \tZ^r(t)(dx)
= \sup_{z \in [\delta, M]} |g(z)|(W_M^r(t) - W_{\delta}^r(t)).
$$
By Theorem \ref{workless}, $\left\{W_M^r(\cdot) - W_{\delta}^r(\cdot)\right\}_{r\in\clr}$ is tight, which implies that $\left\{\int_{\delta}^Mf(x) \tZ^r(t)(dx)\right\}_{r\in\clr}$ is tight for each fixed $t\ge 0$.
Thus, (A1) of Aldous' tightness criterion holds for\\
$\left\{\int_{\delta}^Mf(x) \tZ^r(\cdot)(dx)\right\}_{r\in\clr}$.

Next we show that (A2) of Aldous' tightness criterion holds for the above sequence as well.
Fix $T \in (0,\infty)$, $\eta \in (0,1)$ and a stopping time $\tau$ that takes values in $[0,T]$. Then, by Lemma \ref{parts}, for $r\in\clr$,
\begin{multline}\label{fcomp}
\left|\int_{\delta}^Mf(x) \tZ^r(\tau + \eta)(dx) - \int_{\delta}^Mf(x) \tZ^r(\tau)(dx)\right|\\
 \le C_g\left(\int_{\delta}^M|W_x^r(\tau+\eta) - W_x^r(\tau)|dx + |W_M^r(\tau+\eta) - W_M^r(\tau)| + |W_{\delta}^r(\tau+\eta) - W_{\delta}^r(\tau)|\right),
\end{multline}
where $C_g := \left(\sup_{z \in [\delta, M]} |g'(z)|\right) + |g(M)| + |g(\deltap)|$. By \eqref{eq:comp2} in Proposition \ref{comp} and \eqref{skortimecomp}, for any $r\in\clr$ and
$x \in [\delta,M]$,
\begin{eqnarray}
|W_x^r(\tau+\eta) - W_x^r(\tau)|
&\le& |Y_x^r(\tau+\eta) - Y_x^r(\tau)| + \frac{xc^r}{r}\label{xzcomp}\\
&\le& 2\sup_{\tau \le s \le \tau+\eta}|X_x^r(s) - X_x^r(\tau)| + \frac{xc^r}{r}.\nonumber
\end{eqnarray}
Thus, for $r\in\clr$,
\begin{eqnarray}
\int_{\delta}^M|W_x^r(\tau+\eta) - W_x^r(\tau)|dx &\le& \int_{\delta}^M\left(|Y_x^r(\tau+\eta) - Y_x^r(\tau)| + \frac{xc^r}{r}\right)dx\label{tight2}\\
&\le& 2\int_{\delta}^M\sup_{\tau \le s \le \tau+\eta}|X_x^r(s) - X_x^r(\tau)|dx + \frac{M^2c^r}{2r}.\nonumber
\end{eqnarray}
Note that for $r\in\clr$, $s \in [\tau, \tau+\eta]$ and $x\in\RRp$,
\begin{equation*}
X_x^r(s) - X_x^r(\tau) = \hat{V}_x^r(s) - \hat{V}_x^r(\tau) +r(\rho^r_{xc^r} - 1)(s-\tau),
\end{equation*}
and hence
\begin{equation}\label{Zd}
\sup_{\tau \le s \le \tau+\eta}|X_x^r(s) - X_x^r(\tau)| \le \sup_{\tau \le s \le \tau+\eta}|\hat{V}_x^r(s) - \hat{V}_x^r(\tau)| +\left|r(\rho^r_{xc^r} - 1)\right|\eta.
\end{equation}
For each $r\in\clr$, define a process $\hat{U}^r(\cdot)$ as follows:
\begin{equation*}
\hat{U}^r(t) := \frac{1}{r}\sum_{i=1}^{\lfloor r^2t\rfloor}\left(v_i \one_{[v_i > \delta c^r]} - \mathbb{E}(v \one_{[v > \delta c^r]})\right),\qquad\hbox{for }t\ge 0.
\end{equation*}
Note that for any $r\in\clr$, $s \in [\tau, \tau+\eta]$ and $x \in [\delta, M]$,
\begin{align}\label{supx}
|\hat{V}_x^r(s) - \hat{V}_x^r(\tau)| &\le |\hat{V}_{\infty}^r(s) - \hat{V}_{\infty}^r(\tau)| + \frac{1}{r}\sum_{i=E^r(r^2\tau)+1}^{E^r(r^2(\tau+\eta))}v_i \one_{[v_i > \delta c^r]} + r \lambda^r\eta\mathbb{E}(v \one_{[v > \delta c^r]})\nonumber\\
&\le  |\hat{V}_{\infty}^r(s) - \hat{V}_{\infty}^r(\tau)| + |\hat{U}^r(\overline{E}^r(\tau + \eta)) - \hat{U}^r(\overline{E}^r(\tau))|\\
&\qquad + \frac{1}{r}\left(E^r(r^2(\tau+\eta)) - E^r(r^2\tau)\right)\mathbb{E}(v \one_{[v > \delta c^r]}) + r \lambda^r\eta\mathbb{E}(v \one_{[v > \delta c^r]}).\nonumber
\end{align}
By Proposition \ref{prop:fclt} as $r\to\infty$, $\hat{V}_{\infty}^r(\cdot) \xrightarrow{d} V^*(\cdot)$  in $\mathcal{D}([0,T+1] : \mathbb{R})$ for some Brownian motion $V^*$ with zero drift and finite variance. Fix $\gamma\in(0,1/2)$. 
Recall the notation $|f(t\#) - f(s\#)| <  A$, from Section \ref{sec:notat},
for  a RCLL function $f$,  $0 \le s \le t \le \infty$ and $A>0$.
For $K>0$, define the set
$$
\Omega(K) := \{|V^*(t\#) - V^*(s\#)| < K\eta^{\gamma} \text{ for all } 0 \le s \le t \le T+ 1 \text{ with } t-s \le \eta\}.
$$
Fix $\epsilon \in (0,1/8)$. Since $V^*$ is Holder continuous with exponent $\gamma$, there exists $K_{\epsilon}$ (not depending on $\eta$) large enough such that $\mathbb{P}(\Omega(K_{\epsilon})) \ge 1-\epsilon.$
Since for any $K>0$, the set
$$
A(K) := \{f \in \mathcal{D}([0,T+1] : \mathbb{R}): |f(t\#) - f(s\#)| < K\eta^{\gamma} \text{ for all } 0 \le s \le t \le T+ 1 \text{ with } t-s \le \eta\}
$$
is nonempty and open in the Skorohod topology by \cite[Chapter 3, Proposition 6.5]{ethier2009markov} and $\hat{V}_{\infty}^r(\cdot) \xrightarrow{d} V^*(\cdot)$ as $r\to\infty$, the Portmanteau theorem implies that there exists $r_0 > 0$ such that for all $r \ge r_0$,
$$
\mathbb{P}\left(\hat{V}_{\infty}^r(\cdot) \in A(K_{\epsilon})\right) \ge 1-2\epsilon,
$$
and consequently, for all $r \ge r_0$,
\begin{equation}\label{munif1}
\mathbb{P}\left(\sup_{\tau \le s \le \tau+\eta}|\hat{V}_{\infty}^r(s) - \hat{V}_{\infty}^r(\tau)| \ge K_{\epsilon} \eta^{\gamma}\right) \le 2\epsilon.
\end{equation}
Recall that 
$
\overline{E}^r(\cdot) \xrightarrow{d} \lambda(\cdot),
$
where $\lambda(t) = \lambda t$ for $t \ge 0$,
and by Proposition \ref{prop:fclt},
$
\hat{U}^r(\cdot) \xrightarrow{d} 0
$
as $r \rightarrow \infty$. Therefore, as $r \rightarrow \infty$,
$
\hat{U}^r(\overline{E}^r(\cdot)) \xrightarrow{d} 0
$
and consequently, there exists $r_1 \ge r_0$ such that for $r \ge r_1$,
\begin{equation}\label{hatXsmall}
\mathbb{P}\left(|\hat{U}^r(\overline{E}^r(\tau + \eta)) - \hat{U}^r(\overline{E}^r(\tau))| > \eta^{\gamma}\right) \le 2\mathbb{P}\left(\sup_{t \in [0,T+1]}|\hat{U}^r(\overline{E}^r(t))| > \eta^{\gamma}/2\right) < \epsilon.
\end{equation}
Now, using the fact that $r\mathbb{E}[v\one_{[v>c^r]}]=1$ due to \eqref{def:S}, \eqref{eq:ssinv} and \eqref{def:cr}, we write the sum of the third and the fourth terms on the right side of \eqref{supx} as
\begin{eqnarray}\label{kara}
&&\frac{1}{r}\left(E^r(r^2(\tau+\eta)) - E^r(r^2\tau)\right)\mathbb{E}(v \one_{[v > \delta c^r]}) + r \lambda^r\eta\mathbb{E}(v \one_{[v > \delta c^r]})\nonumber\\
&&\qquad= \frac{E^r(r^2(\tau +\eta)) - E^r(r^2\tau)}{r^2}\frac{\mathbb{E}(v \one_{[v > \delta c^r]})}{\mathbb{E}(v \one_{[v > c^r]})} + \lambda^r\eta\frac{\mathbb{E}(v \one_{[v > \delta c^r]})}{\mathbb{E}(v \one_{[v > c^r]})}\nonumber\\
&&\qquad= \left(\overline{E}^r(\tau +\eta) - \overline{E}^r(\tau)\right)\frac{\mathbb{E}(v \one_{[v > \delta c^r]})}{\mathbb{E}(v \one_{[v > c^r]})} + \lambda^r\eta\frac{\mathbb{E}(v \one_{[v > \delta c^r]})}{\mathbb{E}(v \one_{[v > c^r]})}.
\end{eqnarray}
As the set
$$
\Omega^*:= \{f \in \mathcal{D}([0,T+1] : \mathbb{R}): |f(t\#) - f(s\#)| < 2\lambda \eta \text{ for all } 0 \le s \le t \le T+ 1 \text{ with } t-s \le \eta\}
$$
is nonempty and open in the Skorohod topology and 
$\overline{E}^r(\cdot) \xrightarrow{d} \lambda(\cdot)$ as $r\to\infty$, there exists $r_2 \ge r_1$ such that for all $r \ge r_2$,
$$
\mathbb{P}\left(\overline{E}^r(\tau +\eta) - \overline{E}^r(\tau) \ge 2\lambda \eta\right) < \epsilon.
$$
Moreover, $\lambda^r \rightarrow \lambda$ as $r \rightarrow \infty$ and \eqref{karama} implies
$$
\lim_{r \rightarrow \infty} \frac{\mathbb{E}(v \one_{[v > \delta c^r]})}{\mathbb{E}(v \one_{[v > c^r]})} = \frac{1}{\delta^p}.
$$
Using these observations in \eqref{kara} gives that there is an $r_3 \ge r_2$ such that for all $r \ge r_3$,
\begin{equation}\label{munif2}
\mathbb{P}\left(\frac{1}{r}\left(E^r(r^2(\tau+\eta)) - E^r(r^2\tau)\right)\mathbb{E}(v \one_{[v > \delta c^r]}) + r \lambda^r\eta\mathbb{E}(v \one_{[v > \delta c^r]}) > \frac{8\lambda \eta}{\delta^p}\right) < \epsilon.
\end{equation}
Using \eqref{supx}, \eqref{munif1}, \eqref{hatXsmall} and \eqref{munif2}, we obtain for $r \ge r_3$,
\begin{multline}\label{munif3}
\mathbb{P}\left(\sup_{x \in [\delta, M]}\sup_{\tau \le s \le \tau+\eta}|\hat{V}_x^r(s) - \hat{V}_x^r(\tau)| > \left(K_{\epsilon} + 1 +\frac{8\lambda}{\delta^p}\right)\eta^{\gamma}\right)\\
\le \mathbb{P}\left(\sup_{\tau \le s \le \tau+\eta}|\hat{V}_{\infty}^r(s) - \hat{V}_{\infty}^r(\tau)| > K_{\epsilon} \eta^{\gamma}\right)
+ \mathbb{P}\left(|\hat{U}^r(\overline{E}^r(\tau + \eta)) - \hat{U}^r(\overline{E}^r(\tau))| > \eta^{\gamma}\right)\\
\quad\quad+ \mathbb{P}\left(\frac{1}{r}\left(E^r(r^2(\tau+\eta)) - E^r(r^2\tau)\right)\mathbb{E}(v \one_{[v > \delta c^r]}) + r \lambda^r\eta\mathbb{E}(v \one_{[v > \delta c^r]})) > \frac{8\lambda\eta}{\delta^p}\right)< 4\epsilon.
\end{multline}
Moreover, by \eqref{eq:htcr} and the uniform convergence in \eqref{karama2},
 $r(\rho^r_{xc^r} - 1) \rightarrow \kappa - \frac{\lambda}{x^p}$ as $r \rightarrow \infty$ uniformly for $x \in [\delta, \infty)$. Thus, there exists $C_1>0$ and $r_4 \ge r_3$ such that for all $r \ge r_4$,
\begin{equation}\label{rhouni}
\sup_{x \in [\delta, M]}|r(\rho^r_{xc^r} - 1)| \le C_1.
\end{equation}
Using \eqref{Zd}, \eqref{munif3} and \eqref{rhouni}, for some $C_2 \in (0,\infty)$ and all $r \ge r_4$,
\begin{equation}\label{Ztight}
\mathbb{P}\left(\sup_{x \in [\delta, M]}\sup_{\tau \le s \le \tau+\eta}|X_x^r(s) - X_x^r(\tau)| > \left(K_{\epsilon} + 1 + \frac{8\lambda}{\delta^p} + C_2\right)\eta^{\gamma}\right) < 4\epsilon.
\end{equation}
Take $r_5 \ge r_4$ such that $\max\{M^2c^r/(2r), Mc^r/r\} < \eta^{\gamma}$ for all $r\ge r_5$ and define $C_3:= 2(M-\delta)\left(K_{\epsilon} + 1 + \frac{8\lambda}{\delta^p} + C_2\right) + 1$. Then, using  \eqref{tight2} and \eqref{Ztight}, we obtain, for all $r\ge r_5$,
\begin{equation}\label{tightint}
\mathbb{P}\left(\int_{\delta}^M|W_x^r(\tau+\eta) - W_x^r(\tau)|dx > C_3\eta^{\gamma}\right) < 4\epsilon.
\end{equation}
Similarly, using \eqref{xzcomp} and \eqref{Ztight} and writing $C_4 := 2\left(K_{\epsilon} + 1 + \frac{8\lambda}{\delta^p} + C_2\right) + 1$,
for $r \ge r_5$, we can show that
\begin{equation}\label{tightint2}
\mathbb{P}\left(|W_M^r(\tau+\eta) - W_M^r(\tau)| + |W_{\delta}^r(\tau+\eta) - W_{\delta}^r(\tau)| > C_4\eta^{\gamma}\right) < 4\epsilon.
\end{equation}
Finally, using \eqref{fcomp}, \eqref{tightint} and \eqref{tightint2}, and the fact that $T$, $\eta$, $\epsilon$ and $\tau$ were arbitrary,
we conclude that for any $T>0$, $\eta \in (0,1)$, $\epsilon \in (0,1/8)$, and stopping time $\tau$ taking values in $[0,T]$, there exists $C^*>0$ and $r^*>0$
such that for any $r \ge r^*$,
\begin{equation}\label{tightII}
\mathbb{P}\left(\left|\int_{\delta}^Mf(x) \tZ^r(\tau + \eta)(dx) - \int_{\delta}^Mf(x) \tZ^r(\tau)(dx)\right|> C^*\eta^{\gamma}\right) < 8\epsilon.
\end{equation}
For instance, $C^*=C_g(C_3 + C_4)$ and $r^*=r_5$. Equation \eqref{tightII} implies that condition (A2) of Aldous' tightness criterion also holds.
Thus,
$\left\{\int_{\delta}^Mf(x) \tZ^r(\cdot)(dx)\right\}_{r\in\clr}$ is tight in $\mathcal{D}([0,T]:\RR)$  by Aldous' tightness criterion.\\

\noindent\emph{Proof of finite dimensional joint convergence: } For $r\in\clr$ and $t\ge 0$, write
\begin{align*}
\Psi^r(t) &:= 
  - \int_{\delta}^Mg'(x)W_x^r(t)dx + g(M)W_M^r(t) - g(\deltap)W_{\delta}^r(t),\\
 \Psi(t) &:= - \int_{\delta}^Mg'(x)W_x(t)dx + g(M)W_M(t) - g(\deltap)W_{\delta}(t).
\end{align*}
Fix $k\in\NN$, $T>0$, and $0\le t_1<\dots<t_k\le T$.  We will use Lemma \ref{seqconv} and Proposition \ref{comp} to show that
\begin{equation}\label{finconmain}
\mathbf{A}^r :=(\Psi^r(t_1),\dots, \Psi^r(t_k)) \xrightarrow{d} \mathbf{A} :=(\Psi(t_1), \dots, \Psi(t_k))
\end{equation}
as $r \rightarrow \infty$. For this, for each $n \in\NN$, let $\delta = x_0 < x_1 < \dots < x_{K_n} = M$ be a partition of mesh $n^{-1}$. For $r\in\clr$, $n\in\NN$, and $t\ge 0$, define
\begin{align*}
\Psi_n^r(t) &:= \sum_{j=0}^{K_n-1}W_{x_j}^r(t) (g(x_j) - g(x_{j+1})) + g(M)W_M^r(t) - g(\deltap)W_{\delta}^r(t),\\
\Psi_n(t) &:= \sum_{j=0}^{K_n-1}W_{x_j}(t) (g(x_j) - g(x_{j+1})) + g(M)W_M(t) - g(\deltap)W_{\delta}(t).
\end{align*}
Observe that for each $n \in\NN$, by Theorem \ref{workless} and the continuous mapping theorem,
\begin{equation}\label{nconv}
\Psi_n^r(\cdot) \xrightarrow{d} \Psi_n(\cdot) \quad \text{ in } \mathcal{D}([0,T]:\mathbb{R}) \quad \text{ as } r \rightarrow \infty.
\end{equation}
By \eqref{nconv}, for each $n \in\NN$,
\begin{equation}\label{nconvfin}
\mathbf{A}_n^r :=(\Psi_n^r(t_1),\dots, \Psi_n^r(t_k))\xrightarrow{d} \mathbf{A}_n:=(\Psi_n(t_1), \dots, \Psi_n(t_k)) \quad \text{ as } r \rightarrow \infty.
\end{equation}
For each $r\in\clr$, $n\in\NN$, and $t\ge 0$, note that
\begin{align}
|\Psi_n^r(t) - \Psi^r(t)| &\le \sum_{j=0}^{K_n -1}\int_{x_j}^{x_{j+1}}|g'(x)|\left(W_x^r(t) - W_{x_j}^r(t)\right)dx,\label{eq:PsinrPsir}\\
|\Psi_n(t) - \Psi(t)| &\le \sum_{j=0}^{K_n -1}\int_{x_j}^{x_{j+1}}|g'(x)|\left(W_x(t) - W_{x_j}(t)\right)dx.\label{eq:PsinPsi}
\end{align}
By \eqref{eq:PsinrPsir}, \eqref{eq:comp2} in Proposition \ref{comp} and the Lipschitz property \eqref{eq:lipSM} of the Skorohod map {\color{black}$\Gamma$}, for any $r\in\clr$, $n\in\NN$, and $t \in [0,T]$,
\begin{multline}\label{psicom}
|\Psi_n^r(t) - \Psi^r(t)| \le \sum_{j=0}^{K_n -1}\int_{x_j}^{x_{j+1}}|g'(x)|\left(\left|Y_x^r(t) - Y_{x_j}^r(t)\right| + \frac{xc^r}{r}\right)dx\\
 \le \frac{Mc^r}{r}\int_{\delta}^{M}|g'(x)|dx + 2\sum_{j=0}^{K_n -1}\int_{x_j}^{x_{j+1}}|g'(x)|\left(\sup_{s \in [0,T]}|X_x^r(s) - X_{x_j}^r(s)|\right)dx.
\end{multline}
Now, for any $0 \le j \le K_n-1$ and any $x \in [x_j, x_{j+1}]$, 
\begin{eqnarray}\label{eq:sumtwo}
	\sup_{s \in [0,T]}|X_x^r(s) - X_{x_j}^r(s)| &\le&  \frac{1}{r}\sum_{l=1}^{\qq^r} \breve{v}_l^r \one_{[x_j c^r < \breve{v}_l^r \le x_{j+1} c^r]}\\
	&&+  \frac{1}{r}\sum_{i=1}^{E^r(r^2T)}v_i \one_{[x_jc^r < v_i \le x_{j+1}c^r]}.\nonumber
\end{eqnarray}
Hence, by \eqref{psicom} and \eqref{eq:sumtwo}, for each $r\in\clr$ and $n\in\NN$,
\begin{equation}
\sup_{0\le t\le T} \left| \Psi_n^r(t)-\Psi^r(t)\right|\le\Delta_{n,1}^r+\Delta_{n,2}^r,\label{PhiBnd}
\end{equation}
where
\begin{align*}
\Delta_{n,1}^r &:= \frac{Mc^r}{r}\int_{\delta}^{M}|g'(x)|dx + 2\sum_{j=0}^{K_{n} -1}\int_{x_j}^{x_{j+1}}|g'(x)| dx\left(\frac{1}{r}\sum_{i=1}^{E^r(r^2T)}v_i \one_{[x_jc^r < v_i \le x_{j+1}c^r]}\right)\\
\Delta_{n,2}^r &:= 2\sum_{j=0}^{K_{n} -1}\int_{x_j}^{x_{j+1}}|g'(x)| dx\left(\frac{1}{r}\sum_{l=1}^{\qq^r} \breve{v}_l^r \one_{[x_j c^r < \breve{v}_l^r \le x_{j+1} c^r]}\right).
\end{align*}
Observe that there exists $C>0$ such that for all $r\in\clr$, $n\in\NN$, and $0\le j\le K_n$,
\begin{align}
\mathbb{E}\left(\frac{1}{r}\sum_{i=1}^{E^r(r^2T)}v_i \one_{[x_jc^r < v_i \le x_{j+1}c^r]}\right) \le CrT\mathbb{E}\left(v \one_{[x_jc^r < v \le x_{j+1}c^r]}\right).\label{inq:1}
\end{align}
Recalling that $r\mathbb{E}\left(v \one_{[v > c^r]}\right) =r/S(c^r)= 1$ for each $r\in\clr$, we can write, for $r\in\clr$ and $0\le j\le K_n$,
\begin{eqnarray}
\mathbb{E}\left(v \one_{[x_jc^r < v \le x_{j+1}c^r]}\right) &=& \mathbb{E}\left(v \one_{[v > x_jc^r]}\right) - \mathbb{E}\left(v \one_{[v > x_{j+1}c^r]}\right)\nonumber\\
&=& \frac{1}{r}\left(\frac{\mathbb{E}\left(v \one_{[v > x_jc^r]}\right)}{\mathbb{E}\left(v \one_{[v > c^r]}\right)} - \frac{\mathbb{E}\left(v \one_{[v > x_{j+1}c^r]}\right)}{\mathbb{E}\left(v \one_{[v > c^r]}\right)}\right).\label{eq:2}
\end{eqnarray}
From the uniform convergence in \eqref{karama2},
for each $n\in\NN$, there exists $r_1(n) >0$ such that for all $r \ge r_1(n)$,  
\begin{equation}\label{inq:3}
\left|\frac{\mathbb{E}(v \one_{[v > u c^r]})}{\mathbb{E}(v \one_{[v > c^r]})} - \frac{1}{u^p}\right| < \frac{1}{n} \quad \text{for all } u \in [\delta, \infty).
\end{equation}
By combining \eqref{inq:1}, \eqref{eq:2}, and \eqref{inq:3}, for any $n\in\NN$, $r \ge r_1(n)$ and any $0 \le j \le K_n-1$,
\begin{eqnarray}
\mathbb{E}\left(\frac{1}{r}\sum_{i=1}^{E^r(r^2T)}v_i \one_{[x_jc^r < v_i \le x_{j+1}c^r]}\right)
&\le &CT\left|\frac{\mathbb{E}(v \one_{[v > x_j c^r]})}{\mathbb{E}(v \one_{[v > c^r]})} - \frac{1}{x_j^p}\right|\nonumber\\
&& + CT\left|\frac{\mathbb{E}(v \one_{[v > x_{j+1} c^r]})}{\mathbb{E}(v \one_{[v > c^r]})} - \frac{1}{x_{j+1}^p}\right|\nonumber\\
&& + CT\left(\frac{1}{x_j^p} - \frac{1}{x_{j+1}^p}\right) \le \frac{2CT}{n} + \frac{CTp}{\delta^{p+1}n}.\nonumber
\end{eqnarray}
Thus, for $n\in\NN$ and $r\ge r_1(n)$,
\begin{equation}
\mathbb{E}\left[ \Delta_{n,1}^r\right]\le \frac{Mc^r}{r}\int_{\delta}^M \left| g'(x)\right| dx+2\left(2CT+\frac{CTp}{\delta^{p+1}}\right)\frac{1}{n}\int_{\delta}^M \left| g'(x)\right| dx.\label{rentight}
\end{equation}
Fix $\epsilon \in (0,1)$. Choose $n_{\epsilon} \in \NN$ such that $2\left(2CT + \frac{CTp}{\delta^{p+1}}\right)\frac{1}{n_{\epsilon}}\int_{\delta}^{M}|g'(x)|dx < \epsilon^2/(4\sqrt{k})$ and
$$
\mathbb{P}\left(\sup_{\left\{\delta<y, z\le M : |z-y|\le n_{\epsilon}^{-1}\right\}}\left|\xi(z) - \xi(y)\right| < \frac{\epsilon}{4\sqrt{k}\int_{\delta}^M|g'(x)|dx}\right) \ge 1-\epsilon/4,
$$
which can be ensured to exist since $\xi(\cdot)$ is continuous and $[\delta,M]$ is compact. Noting that
\begin{align*}
S^{(\epsilon)} &:= \left\lbrace f \in \mathcal{D}([\delta,M] : \mathbb{R}) : |f(z\#) - f(y\#)| < \frac{\epsilon}{4\sqrt{k}\int_{\delta}^M|g'(x)|dx}\right.\\
&\qquad\qquad \left. \forall\ \delta \le y, z \le M \text{ with } \left|z-y\right| \le n_{\epsilon}^{-1}\right\rbrace
\end{align*}
is nonempty and open in the Skorohod topology and by assumption \eqref{eq:assuinitcond}, we obtain $r_2\ge r_1(n_{\epsilon})$ such that for all $r \ge r_2$, 
\begin{equation}\label{intight}
\mathbb{P}\left(\frac{1}{r}\sum_{l=1}^{\qq^r} \breve{v}_l^r \one_{[x_j c^r < \breve{v}_l^r \le x_{j+1} c^r]} \ge \frac{\epsilon}{4\sqrt{k}\int_{\delta}^M|g'(x)|dx} \text{ for some } 0 \le j \le K_{n_{\epsilon}}-1\right) < \epsilon/2.
\end{equation}
Using \eqref{PhiBnd}, \eqref{rentight}, \eqref{intight}, and the choice of $n_{\epsilon}$,
we conclude that for $r \ge r_2$,
\begin{align*}
\mathbb{P}&\left(\sup_{t \in [0,T]}|\Psi_{n_{\epsilon}}^r(t) - \Psi^r(t)| > \frac{\epsilon}{\sqrt{k}} \right) \le \mathbb{P}\left(\Delta_{n_{\epsilon},1}^r(t) > \frac{\epsilon}{2\sqrt{k}}\right) + \mathbb{P}\left(\Delta_{n_{\epsilon},2}^r > \frac{\epsilon}{2\sqrt{k}} \right)\\
 &\qquad \qquad \qquad  \le \frac{2\sqrt{k}Mc^r}{\epsilon r}\int_{\delta}^{M}|g'(x)|dx + \frac{4\sqrt{k}}{\epsilon}\left(2CT + \frac{CTp}{\delta^{p+1}}\right)\frac{1}{n_{\epsilon}}\int_{\delta}^{M}|g'(x)|dx + \frac{\epsilon}{2}\\
 &\qquad \qquad \qquad  \le \frac{2\sqrt{k}Mc^r}{\epsilon r}\int_{\delta}^{M}|g'(x)|dx + \epsilon.
\end{align*}
Therefore, 
\begin{equation}\label{acon}
\limsup_{r}\mathbb{P}\left(\|\mathbf{A}_{n_{\epsilon}}^r - \mathbf{A}^r\|_2 > \epsilon\right) \le \limsup_{r}\mathbb{P}\left(\sqrt{k}\sup_{t \in [0,T]}|\Psi_{n_{\epsilon}}^r(t) - \Psi^r(t)| > \epsilon\right)
\le \epsilon,
\end{equation}
where $\|\cdot\|_2$ denotes the $L^2$-norm in $\mathbb{R}^k$. Thus, condition (1) of Lemma \ref{seqconv} holds with $r$ in place of $m$, $S^{\epsilon}_r =\mathbf{A}_{n_{\epsilon}}^r$, $S_r =\mathbf{A}^r$ and $b(\epsilon) = 2\epsilon$. Condition (2) of Lemma \ref{seqconv} with $S^{\epsilon} = \mathbf{A}_{n_{\epsilon}}$ follows from \eqref{nconvfin}.

Next, recall $W_a(\cdot) = \Gamma[X_a](\cdot)$ and for $\theta>0$, $b>a>0$, write 
$$
\omega(\xi, \theta; [a,b]) := \sup\{|\xi(y) - \xi(x)| : a \le x, y \le b, \  |y-x| \le \theta\}.
$$
By the continuity of $\xi(\cdot)$, for any fixed $a,b$, $\lim_{\theta \rightarrow 0} \omega(\xi, \theta ; [a,b]) = 0$ almost surely.
Using this observation along with \eqref{eq:PsinPsi}, the Lipschitz property \eqref{eq:lipSM} of the Skorohod map
and \eqref{eq:x0defn}-\eqref{eq:xadefn}, for each $n\in\NN$,
\begin{align*}
&\sup_{t \in [0,T]}|\Psi_n(t) - \Psi(t)|\\
&\qquad \le 2T\lambda\sum_{j=0}^{K_n -1}\int_{x_j}^{x_{j+1}}|g'(x)|\left(x^{-p}_j - x^{-p}\right)dx + 2\omega(\xi, n^{-1}; [\delta, M])\int_{\delta}^M|g'(x)|dx \\
&\qquad\le 2T\lambda\sum_{j=0}^{K_n -1}\int_{x_j}^{x_{j+1}}|g'(x)|\frac{p(x-x_j)}{x_j^{p+1}}dx + 2\omega(\xi, n^{-1}; [\delta, M])\int_{\delta}^M|g'(x)|dx\\
&\qquad \le \frac{2 T \lambda p}{\delta^{p+1}n}\int_{\delta}^{M}|g'(x)|dx + 2\omega(\xi, n^{-1}; [\delta, M])\int_{\delta}^M|g'(x)|dx. 
\end{align*}
Thus, $\sup_{t \in [0,T]}|\Psi_n(t) - \Psi(t)|\to 0$ almost surely as  $n \rightarrow \infty$, which implies that
\begin{equation}\label{ancon}
\|\mathbf{A}_n - \mathbf{A}\| \le \sqrt{k}\sup_{t \in [0,T]}|\Psi_n(t) - \Psi(t)| \rightarrow 0 \quad \text{ almost surely as } n \rightarrow \infty.
\end{equation} 
Thus, condition (3) of Lemma \ref{seqconv} holds with $S^{\epsilon} = \mathbf{A}_{n_{\epsilon}}$ and $S^0 = \mathbf{A}$.
The weak convergence in \eqref{finconmain} now follows from \eqref{nconvfin}, \eqref{acon}, \eqref{ancon} and Lemma \ref{seqconv}. 

This completes the proof of the convergence claimed in the theorem. The continuity of the limiting process in the theorem follows from that of Brownian motion and \eqref{skortimecomp}.
\end{proof}

To prove Theorem \ref{workfn} as a consequence of Theorem \ref{convplus}, we need the following result.
\begin{lemma}\label{Minfty}
For $0<\delta < \infty$, consider any $C^1$ function $f: [\delta, \infty) \rightarrow \mathbb{R}$ such that $\lim_{x \rightarrow \infty} \frac{f(x)}{x}$ exists and $\int_1^{\infty} \frac{|f'(x)|}{x^{\alpha^* +1}} < \infty$, where $\alpha^*$ is the constant appearing in Assumption \eqref{eq:initcondq2}. Then, writing $g(x) = f(x)/x$ for $x\in[\delta,\infty)$ and $g(\infty) = \lim_{x \rightarrow \infty} g(x)$, 
the following distributional convergence holds in $\mathcal{D}([0,\infty) : \mathbb{R})$:
\begin{equation}\label{delinf}
\int_{\delta}^{\infty}f(x) \tZ^r(\cdot)(dx) \xrightarrow{d} - \int_{\delta}^{\infty}g'(x)W_x(\cdot)dx + g(\infty)W_{\infty}(\cdot) - g(\delta)W_{\delta}(\cdot),
\end{equation}
as $r\to\infty$, where the right side of \eqref{delinf} defines a stochastic process with
sample paths in $\mathcal{C}([0,\infty) : \mathbb{R})$.
\end{lemma}

\begin{proof} Fix $T,\delta>0$.
For $\delta<M< \infty$ and $t\ge 0$, define the following:
$$
\Phi^r_{M}(t) := \int_{\delta}^{M}f(x)\tZ^r(t)(dx), \quad \Phi^r_{\infty}(t) := \int_{\delta}^{\infty}f(x)\tZ^r(t)(dx),
$$
and
$$
\Upsilon_{M}(t) := - \int_{\delta}^Mg'(x)W_x(t)dx + g(M)W_{M}(t) - g(\delta)W_{\delta}(t).
$$
We first show that on the time interval $[0,T]$, almost surely, $\int_{\delta}^Mg'(x)W_x(\cdot)dx$ converges uniformly {\color{black}(with respect to time) on the time interval $[0,T]$} as $M\to\infty$,
and hence, the limit
$\int_{\delta}^{\infty}g'(x)W_x(\cdot)dx$ is well defined and continuous on $[0,T]$. Note that for any $M' > M > \delta$,
\begin{align*}
\sup_{t \in [0,T]}\left| \int_M^{M'}g'(x)W_x(t)dx\right| &\le \sup_{t \in [0,T]}\left| \int_M^{M'}g'(x)(W_{\infty}(t) - W_x(t))dx\right|\\
&\qquad + \sup_{t \in [0,T]}\left| \int_M^{M'}g'(x)W_{\infty}(t)dx\right|.
\end{align*}
By \eqref{eq:xadefn}, \eqref{eq:xinfinitydefn}, the Lipschitz property \eqref{eq:lipSM} of the Skorohod map and recalling that $\xi(\infty) := \lim_{u \rightarrow \infty}\xi(u) < \infty$ almost surely by assumption, for any $M' > M > \delta$,
\begin{multline}\label{limc1}
\sup_{t \in [0,T]}\left| \int_M^{M'}g'(x)(W_{\infty}(t) - W_x(t))dx\right|\\
 \le 2T\lambda \int_M^{\infty}|g'(x)|x^{-p}dx + 2\int_M^{\infty}|g'(x)|(\xi(\infty) - \xi(x))dx.
\end{multline}
By Assumption \eqref{eq:initcondq2}, $0<\alpha^*\le p$, and so, by the assumptions on $f$ in the theorem,
$$
\int_1^{\infty}|g'(x)|x^{-p}dx \le \int_1^{\infty}\left(\frac{|f(x)|}{x^{p+2}} + \frac{|f'(x)|}{x^{p+1}}\right)dx \le \int_1^{\infty}\left(\frac{|f(x)|}{x^{\alpha^*+2}} + \frac{|f'(x)|}{x^{\alpha^*+1}}\right)dx < \infty.
$$
Also, by Assumption \eqref{eq:initcondq2}, there exists $C>0$ such that
\begin{equation}\label{eq:haifa1}
\mathbb{E}\left(\xi(\infty) - \xi(x)\right) \le C x^{-\alpha^*} \mbox{ for all } x\ge 1.
\end{equation}
Hence, by Fubini's theorem,
$$
\mathbb{E}\left(\int_1^{\infty}|g'(x)|(\xi(\infty) - \xi(x))dx\right) \le C\int_{1}^{\infty}\left(\frac{|f(x)|}{x^{2}} + \frac{|f'(x)|}{x}\right)x^{-\alpha^*}dx < \infty.
$$
Moreover, for any $M' > M>\delta$,
\begin{equation}\label{limc2}
\sup_{t \in [0,T]}\left| \int_M^{M'}g'(x)W_{\infty}(t)dx\right| = \sup_{t \in [0,T]}W_{\infty}(t)|g(M') - g(M)|.
\end{equation}
Hence, \eqref{limc1}, \eqref{limc2} {\color{black}and the finiteness of $\lim_{x\to\infty}g(x)=g(\infty)$} imply that, almost surely, {\color{black}the sequence $\left\{\int_{\delta}^{M_n}g'(x)W_x(\cdot)dx\right\}_{n=1}^{\infty}$} is uniformly Cauchy {\color{black}on $[0,T]$ for any sequence $\{M_n\}_{n=1}^{\infty}$ such that $\lim_{n\to\infty}M_n=\infty$},
which proves the uniform convergence to the limit $\int_{\delta}^{\infty} g'(x)W_x(\cdot)dx$ as $M\to\infty$. Moreover, by \eqref{eq:xadefn}, \eqref{eq:xinfinitydefn}, and the Lipschitz property \eqref{eq:lipSM} of the Skorohod map, for any $M>\delta$,
\begin{align*}
&\sup_{t \in [0,T]}|g(M)W_M(t) - g(\infty)W_{\infty}(t)|\\
&\le |g(M)|\sup_{t \in [0,T]}|W_M(t) - W_{\infty}(t)|
 + \left(\sup_{t \in [0,T]}|W_{\infty}(t)|\right)|g(\infty) - g(M)|\\
&\le 2\left(\sup_{x \ge \delta} |g(x)|\right) T\lambda M^{-p} + 2\left(\sup_{x \ge \delta} |g(x)|\right)(\xi(\infty) - \xi(M))\\
&\qquad +  \left(\sup_{t \in [0,T]}|W_{\infty}(t)|\right)|g(\infty) - g(M)|. 
\end{align*}
The upper bound in the display immediately above tends to zero as $M\to\infty$.
Thus, we conclude that
$$
\Upsilon_{\infty}(t) := - \int_{\delta}^{\infty}g'(x)W_x(t)dx + g(\infty)W_{\infty}(t) - g(\delta)W_{\delta}(t), \quad t\in[0,T],
$$
is well defined and {\color{black}$\Upsilon_{\infty}:[0,T]\to\RR$ is} continuous and, almost surely,
\begin{equation}\label{Minf1}
\lim_{M \rightarrow \infty}\sup_{t \in [0,T]}\left| \Upsilon_{M}(t) - \Upsilon_{\infty}(t)\right| = 0.
\end{equation}
By Theorem \ref{convplus}, $\Phi_M^r(\cdot)\to\Upsilon_M(\cdot)$ in $\mathcal{D}([0,T]:\RR)$ as $r\to\infty$ for each $M>\delta$.  This
together with \eqref{Minf1} implies that conditions (2) and (3) of Lemma \ref{seqconv} hold.  Thus, in order to show that $\Phi_{\infty}^r(\cdot)\to\Upsilon_{\infty}(\cdot)$
in $\mathcal{D}([0,T]:\RR)$ as $r\to\infty$, it suffices to show that condition (1) of Lemma \ref{seqconv} holds.  For this,
observe that as $\lim_{x \rightarrow \infty} \frac{f(x)}{x}$ exists, there exists a constant $C'>0$ such that $|f(x)| \le C'x$ for all $x \ge \delta$.
Hence, for all $r\in\clr$, $t\in[0,T]$ and $M>\delta$,
\begin{align*}
|\Phi^r_{\infty}(t) - \Phi^r_{M}(t)| &\le \int_M^{\infty}|f(x)|\tZ^r(t)(dx) \le C'\int_M^{\infty}x\tZ^r(t)(dx)\\
&= \frac{C'}{r} \sum_{l=1}^{\qq^r} \breve{v}_l^r(r^2t) \one_{[\breve{v}_l^r(r^2t) > Mc^r]}  + \frac{C'}{r}\sum_{i=1}^{E^r(r^2t)}v_i(r^2t)\one_{[v_i(r^2t) > Mc^r]}\\
&\le \frac{C'}{r} \sum_{l=1}^{\qq^r} \breve{v}_l^r \one_{[\breve{v}_l^r > Mc^r]}  + \frac{C'}{r}\sum_{i=1}^{E^r(r^2T)}v_i\one_{[v_i > Mc^r]}.
\end{align*}
Thus, due to {\color{black}\eqref{eq:assuinitcond} and \eqref{eq:assuinitcondb} (in particular, the limits displayed after \eqref{eq:assuinitcondb}),} the independence of $E^r(\cdot)$ and $\{v_i\}_{i\in\NN}$ for each $r\in\clr$, Wald's lemma, and \eqref{eq:aflln}, there exists $C''>0$
such that for all $M>\delta$,
\begin{align*}
\limsup_{r \rightarrow \infty}\mathbb{E}&\left(\sup_{t \in [0,T]}|\Phi^r_{\infty}(t) - \Phi^r_{M}(t)|\right)\\
 &\le C'\limsup_{r \rightarrow \infty}\mathbb{E}\left(\frac{1}{r}\sum_{l=0}^{\qq^r} \breve{v}_l^r \one_{[\breve{v}_l^r > Mc^r]}\right) +  \limsup_{r \rightarrow \infty}\frac{C'}{r}\mathbb{E}\left(\sum_{i=1}^{E^r(r^2T)}v_i\one_{[v_i > Mc^r]}\right)\\ 
 &{\color{black}= C'\limsup_{r \rightarrow \infty} \left(W_{\infty}^r(0) - W_M^r(0)\right) + \limsup_{r \rightarrow \infty}\frac{C'}{r}\mathbb{E}\left(\sum_{i=1}^{E^r(r^2T)}v_i\one_{[v_i > Mc^r]}\right)}\\
& \le C'\mathbb{E}(\xi(\infty) - \xi(M)) + \limsup_{r \rightarrow \infty}C''rT\mathbb{E}\left(v\one_{[v > Mc^r]}\right).
\end{align*}
Thus, for all $M>\delta$,
\begin{align*}
\limsup_{r \rightarrow \infty}\mathbb{E}&\left(\sup_{t \in [0,T]}|\Phi^r_{\infty}(t) - \Phi^r_{M}(t)|\right)\\
 &\le
C'\mathbb{E}(\xi(\infty) - \xi(M)) + C''T\limsup_{r \rightarrow \infty}\frac{\mathbb{E}\left(v\one_{[v > Mc^r]}\right)}{\mathbb{E}\left(v\one_{[v > c^r]}\right)} \le \frac{CC'}{M^{\alpha^*}} + \frac{C''T}{M^p},
\end{align*}
where we have used \eqref{eq:Scr} in the first inequality, and \eqref{eq:haifa1} and \eqref{karama} in the second inequality.  From this bound it follows from Markov's inequality that,
for all $M>\delta$,
\begin{equation}\label{Minf2}
\limsup_{r \rightarrow \infty}\mathbb{P}\left(\sup_{t \in [0,T]}|\Phi^r_{\infty}(t) - \Phi^r_{M}(t)| > \frac{1}{M^{\al^*/2}}\right) \le \frac{CC'}{M^{\alpha^*/2}} + \frac{C''T}{M^{p-\alpha^*/2}}.
\end{equation}
As previously noted, by Theorem \ref{convplus}, $\Phi^r_{M}(\cdot) \xrightarrow{d} \Upsilon_M(\cdot)$ as $r \rightarrow \infty$ in $\mathcal{D}([0,T]:\RR)$ for each $M>\delta$. This, along with \eqref{Minf1}, \eqref{Minf2} and Lemma \ref{seqconv}, proves that $\Phi^r_{\infty}(\cdot)\xrightarrow{d} \Upsilon_{\infty}(\cdot)$ as $r \rightarrow \infty$
in $\mathcal{D}([0,T]:\RR)$ and that $\Upsilon_{\infty}(\cdot)$ is continuous in $[0,T]$, which proves the lemma since $T>0$ was arbitrary.
\end{proof}

\subsubsection{Sending $\delta \to 0$}\label{sec:senddtoz}

Next we show that the result in Lemma \ref{Minfty} holds for $\delta=0$.
The strategy involved is again to use Lemma \ref{seqconv} to send $\delta\to 0$ in \eqref{delinf}.
In particular, by \eqref{delinf} in Lemma \ref{Minfty}, condition (2) of Lemma \ref{seqconv} holds. Hence, we can apply Lemma \ref{seqconv}
 after we have shown that the left hand side of \eqref{delinf} is close to $\int_0^{\infty} f(x)\tZ^r(\cdot)dx$ in a uniform sense as required to verify condition (1) of
 Lemma \ref{seqconv}, and the right hand side of \eqref{delinf} converges to the appropriate limit as $\delta \rightarrow 0$ to verify condition (3) of
 Lemma \ref{seqconv}. However, showing that these two conditions hold becomes quite technical. To control the left hand side of \eqref{delinf}, we first show that for any $a>0$ (not depending on $r$), the maximum number of jobs
 in the $r$-th $a(c^r)^{-1}$-truncated queue in the time interval $[0,T]$ is small (Lemma \ref{fifo}) by performing an excursion analysis of the workload process. 
However, the estimates obtained by such an analysis turn out to be too crude to show that the number of jobs of size $\le \delta c^r$ is uniformly small on the time interval $[0,T]$ for small $\delta$. For this, we need much more involved analysis making careful use of the SRPT dynamics. Roughly, we show that the workload process corresponding to jobs of size in the interval $[a, \delta c^r]$ can be bounded above by a (reflected) martingale with large negative drift, quantified in \eqref{XdeltaU}. We then decompose the workload process path into excursions between appropriately chosen levels and control these excursions using the upper bounding process to bound the maximum on $[0,T]$. Finally, bounding the queue-length process by a (sufficiently large) multiple of the workload process (see \eqref{un1}), we obtain a `continuity estimate' in Lemma \ref{unifsmallint} which, in turn, gives the bound on $\sup_{t \in [0,T]}Z^r_{\delta}(t)$ required by Lemma \ref{seqconv} to control the left hand side of \eqref{delinf} (see Lemma \ref{smallql}). To control the right hand side of \eqref{delinf}, we again use excursion analysis to show that the integral $\int_{\delta}^{\infty}g'(x)W_x(\cdot)dx$ indeed converges to a finite random variable as $\delta \rightarrow 0$ (Lemmas \ref{BMsup} and \ref{flowsup}). Together, these estimates complete the proof of Theorem \ref{workfn}.

Recall that for $r\in\clr$, $a > 0$, and $t\ge 0$, $r(c^r)^{-1}Q_{a(c^r)^{-1}}^{r}(t)$ is the queue length of the $r$-th $a(c^r)^{-1}$-truncated SRPT queue at time $r^2t$.
Denote by $\Theta$ the collection of all functions $\theta: \clr \to \RRp$ such that $\theta(r)\to 0$ as $r\to \infty$.
\begin{lemma}\label{fifo}
For any $a, T>0$, there exist $\theta \in \Theta$ and $r_0>0$ such that for $r \ge r_0$,
$$
\mathbb{P}\left(\sup_{t \in [0, T]} Q^r_{a(c^r)^{-1}}(t) > \theta(r)\right) \le \theta(r).
$$
\end{lemma}
\begin{proof} Fix $a,T>0$.  For $r\in\clr$ and $i\in\NN$,
let $T_i^r$ and $v_i$ respectively denote the inter-arrival time and processing time of the $i$th external job in the SRPT queue. For $r\in\clr$, define $T_0^r=0$. As in Section \ref{ss:sd}, $T_1^r$ is strictly positive and has a finite second moment, but does not necessarily have the same distribution as $T_i^r$ for $i\ge 2$, and $T^r$ is a random variable that is equal in distribution to $T_i^r$ for each $i\ge 2$.
Also, for $r\in\clr$ and $t\ge 0$, denote by $\hat{W}_a^r(t) := rW_{a(c^r)^{-1}}^{r}(r^{-2}t)$,
the (unscaled) workload process of the $r$-th $a(c^r)^{-1}$-truncated SRPT queue. Finally, for $r\in\clr$, define the stopping times $\{K_i^r\}_{i=-1}^{\infty}$ with respect to the filtration $\{\mathcal{F}^r_n\}_{n\in\bbZp}$, where $\mathcal{F}^r_0 := \sigma(\{\qq^r, \breve{v}_l^r: l \in \NN\})$ and, for $n \in\NN$, $\mathcal{F}^r_n := \sigma\left(\{\qq^r, \breve{v}_l^r, T^r_i, v_i : l \in \NN, \le i \le n\}\right)$, as follows:
$$
{\color{black}K_{-1}^r =0,}\quad
K_0^r =0 \text{ if } \hat{W}^r_a(0)=0,\ \text{otherwise }K_0^r = \inf\left\lbrace k \in\bbZp: 
\hat{W}_a^r\left(\left(\sum_{i=0}^{k}T_i^{r}\right)-\right) = 0\right\rbrace,
$$
and, for $j \in\bbZp$,
$$
K_{2j+1}^r := K_{2j}^r+1, \quad K_{2j+2}^r := \inf\left\lbrace k \ge K_{2j+1}^r: 
\hat{W}_a^r\left(\left(\sum_{i=0}^{k}T_i^{r}\right)-\right) = 0\right\rbrace.
$$ 
For $r \in \clr$, $l>0$ and $i\in\NN$, write $T_i^{r,l} := T_i^r \wedge l$ and $v_i^a := v_i \one_{[v_i \le a]}$.
Recall that the processing time distribution does not depend on $r\in\clr$. Moreover, as the distribution function $F$ of the processing time is assumed to satisfy $F(x) <1$ for all $x \in \mathbb{R}$, and as $a < \infty$, we have that $\lambda \mathbb{E}(v_1^a) < 1$.  Also, using \eqref{eq:assuht},
\begin{equation}\label{unifint}
\limsup_{l \rightarrow \infty} \limsup_{r \rightarrow \infty} \mathbb{E}\left(T^r\one_{[T^r > l]}\right) = 0.
\end{equation}
Using these observations, there exist $r_0\in\clr$ and $l,\eta>0$ such that $\lambda_l^r := \left(\mathbb{E}\left(T^r\wedge l\right)\right)^{-1}$, $\mathbb{E}(v_1^a)$, $\lambda^r$ and $\sigma_A^r$ satisfy that for all $r \ge r_0$
\begin{equation}\label{fifo1}
\lambda_l^r\mathbb{E}(v_1^a) < 1- 4\eta, \quad \lambda_l^r \le 2\lambda, \quad \lambda^r \ge \lambda/2, \quad (\sigma_A^r)^2 \le 2\sigma_A^2, \quad \lambda r^2T>1.
\end{equation}
Fix $r_0\in\clr$, and $l,\eta>0$ such that \eqref{fifo1} holds. Note that for any $r\ge r_0$ and $j \in\bbZp$
such that $K^r_{2j} \ge K^r_{2j-1} + 2$ and for any $k \in [K^r_{2j-1}+1, K^r_{2j}-1]$, 
\begin{align}\label{woub}
\hat{W}_a^r\left(\sum_{i=0}^{k}T_i^{r}\right)
&= \hat{W}_a^r\left(\sum_{i=0}^{K^r_{2j-1}}T_i^{r}\right) + \sum_{i=K^r_{2j-1} + 1}^k v_i^a - \sum_{i=K^r_{2j-1} + 1}^k T_i^r\nonumber\\
&\le \hat{W}_a^r\left(\sum_{i=0}^{K^r_{2j-1}}T_i^{r}\right) + \sum_{i=K^r_{2j-1} + 1}^k v_i^a - \sum_{i=K^r_{2j-1} + 1}^k T_i^{r,l}\nonumber\\
&\le \hat{W}_a^r\left(\sum_{i=0}^{K^r_{2j-1}}T_i^{r}\right) + M_1(k) - M_1\left(K^r_{2j-1}\right) + M_2^r(k) - M_2^r\left(K^r_{2j-1}\right)\nonumber\\
&\qquad - 2\eta \lambda^{-1}\left(k-K^r_{2j-1}\right) +\left((\lambda_l^r)^{-1}-\Expect{T_1^{r,l}}\right)\one_{\{j=0\}},\nonumber\\
&\le \hat{W}_a^r\left(\sum_{i=0}^{K^r_{2j-1}}T_i^{r}\right) + M_1(k) - M_1\left(K^r_{2j-1}\right) + M_2^r(k) - M_2^r\left(K^r_{2j-1}\right)\nonumber\\
&\qquad - 2\eta \lambda^{-1}\left(k-K^r_{2j-1}\right) +2\lambda^{-1}\one_{\{j=0\}},
\end{align}
where, for $r\ge r_0$, $M_1$ and $M_2^r$ are martingales (with respect to the filtration $\{\mathcal{F}_n^{r}\}_{n\in\bbZp}$ defined above) given by $M_1(0) = M_2^r(0) = 0$, and for $k \in\NN$,
$$
M_1(k) := \sum_{i=1}^k \left(v_i^a - \mathbb{E}(v_i^a) \right)\quad \hbox{and } \quad M_2^r(k) := -\sum_{i=1}^k \left(T_i^{r,l} - \mathbb{E}(T_i^{r,l})\right).
$$
For $r\ge r_0$, write $\mathcal{T}^r_l(k) := \sum_{i=1}^k T_i^{r,l}$ for $k\in\NN$.
As $v_i^a \le a$ for all $i \in\NN$ and there are at most two jobs in the $r$-th $a(c^r)^{-1}$-truncated SRPT queue at time {\color{black}$\sum_{i=0}^{K^r_{2j-1}}T_i^r$} for each $j \in\NN$, and $r\ge r_0$,
$\hat{W}_a^r\left(\sum_{i=0}^{K^r_{2j-1}}T_i^{r}\right) \le 2a$ for all $j \in\NN$.
For $k \ge k_1 :=2(1+ \lambda a / \eta)$, we have
$2\eta \lambda^{-1}k-2a\ge \eta \lambda^{-1}k+\eta \lambda^{-1}k_1-2a\ge \eta \lambda^{-1}k+ 2\eta \lambda^{-1}\ge \eta \lambda^{-1}k$.
Also, observe that $M_1(\cdot)$ (resp.\ $\tilde M_2^r(\cdot):=M_2^r(\cdot+1)-M_2^r(1)$) is equal in distribution to $M_1(\cdot+j)-M_1(j)$ (resp.\ $M_2^r(\cdot+j)-M_2^r(j))$ for all $j\in\NN$.
Hence, as  $M_1$ and $M^r_2$, $r\ge r_0$, are martingales with bounded increments such that the bounds on the increments do not
depend on $r\ge r_0$, using \eqref{woub} and the Azuma-Hoeffding inequality,  we obtain that for $k \ge k_1 :=2(1+ \lambda a / \eta)$, $j\in\NN$, and $r\ge r_0$, 
\begin{align*}
\mathbb{P}\left(K^r_{2j} - K^r_{2j-1} > k\right) &\le \mathbb{P}\left(2a + M_1(k) + \tilde{M}_2^r(k) - 2\eta \lambda^{-1}k> 0\right)\\
&\le \mathbb{P}\left(M_1(k) > \eta k/(2\lambda)\right) + \mathbb{P}\left(\tilde{M}^r_2(k) > \eta k/(2\lambda)\right) \le 2e^{-Ck}
\end{align*}
for some positive constant $C$ depending on $a$, $\eta$, $\lambda$, and $l$, but not $k\ge k_1$ and $r\ge r_0$.
Note that for any $r\ge r_0$ and $j\in\NN$, the queue
length of the $r$-th $a$-truncated SRPT queue in the time interval $[\mathcal{T}^r_l(K^r_{2j-2}), \mathcal{T}^r_l(K^r_{2j-1})]$
is bounded above by 2 and in the time interval
 $[\mathcal{T}^r_l(K^r_{2j-1}), \mathcal{T}^r_l(K^r_{2j})]$ is bounded above
by $K^r_{2j} - K^r_{2j-1}+1$.
Thus, for any $r\ge r_0$, $k \ge k_1 + 1$ and  $N \in\NN$,
\begin{align}\label{fifo2}
&\mathbb{P}\left(\sup_{t \in [\mathcal{T}^r_l(K^r_{0}), \mathcal{T}^r_l(K^r_{2N})]}\frac{r}{c^r}Q_{a(c^r)^{-1}}^{r}(r^{-2}t) > k\right)\\
&\qquad \le \sum_{j=1}^{N}\mathbb{P}\left(\sup_{t \in [\mathcal{T}^r_l(K^r_{2j-1}), \mathcal{T}^r_l(K^r_{2j})]}\frac{r}{c^r}Q_{a(c^r)^{-1}}^{r}(r^{-2}t) > k\right)\nonumber\\
&\qquad \le \sum_{j= 1}^{N}\mathbb{P}\left(K^r_{2j} - K^r_{2j-1} > k-1\right) \le 2N e^{C}e^{-Ck}.\nonumber
\end{align} 
For all $r\ge r_0$, $K^r_{2j} - K^r_{2j-2} \ge 1$ for all $j \in\NN$ and $\lambda_l^r \le 2\lambda$. Hence, for all $r\ge r_0$ and integers
$N \ge 4 r^2\lambda T +1$, 
\begin{multline}\label{fifo3}
\mathbb{P}\left(\mathcal{T}^r_l(K^r_{2N}) < r^2T\right)\\ \le \mathbb{P}\left(\mathcal{T}^r_l(N) -T_1^{r,l} < r^2T\right)
\le \mathbb{P}\left(\sum_{i= 2}^{N} \left(T_i^{r,l} - \mathbb{E}(T_i^{r,l})\right) < r^2T - (N-1)\lambda^{-1}/2\right)\\
\le \mathbb{P}\left(\sum_{i= 2}^{N} \left(T_i^{r,l} - \mathbb{E}(T_i^{r,l})\right) < -\lambda^{-1}(N-1)/4\right)
\le \frac{16\lambda^2(2\sigma_A^2 + 4\lambda^{-2})}{N-1},
\end{multline}
where we have used $\operatorname{Var}\left(T_i^{r,l}\right) \le \mathbb{E}(T^r)^2 = (\sigma_A^r)^2 + (\lambda^r)^{-2}  \le 2\sigma_A^2 + 4\lambda^{-2}$ for $i\ge 2$ in the last bound.
From \eqref{fifo2} and \eqref{fifo3}, for $r \ge r_0$ and any integers $k \ge k_1$ and $N -1 \in [4 \lambda r^2 T, 5 \lambda r^2 T]$,
\begin{multline*}
\mathbb{P}\left(\sup_{t \in [\mathcal{T}^r_l(K^r_{0}), r^2T]}\frac{r}{c^r}Q_{a(c^r)^{-1}}^{r}(r^{-2}t) > k\right)\\ \le \mathbb{P}\left(\sup_{t \in [\mathcal{T}^r_l(K^r_{0}), \mathcal{T}^r_l(K^r_{2N})]}\frac{r}{c^r}Q_{a(c^r)^{-1}}^{r}(r^{-2}t) > k\right)  + \mathbb{P}\left(\mathcal{T}^r_l(K^r_{2N}) < r^2T\right)\\
\le 2Ne^{C}e^{-Ck} + \frac{16\lambda^2(2\sigma_A^2 + 4\lambda^{-2})}{N -1} \le 10 e^{C} \lambda r^2 Te^{-Ck} + \frac{4\lambda(2\sigma_A^2 + 4\lambda^{-2})}{r^2T}.
\end{multline*}
Taking $r_1\ge r_0$ such that $\lfloor 3\log r_1/C \rfloor + 1\ge k_1$ and $k = \lfloor 3\log r/C \rfloor + 1$, we obtain that for some $r_1\ge r_0$ and all  $r\ge r_1$,
\begin{equation}\label{endsup}
\mathbb{P}\left(\sup_{t \in [r^{-2}\mathcal{T}^r_l(K^r_{0}),T]}Q_{a(c^r)^{-1}}^{r}(t) > \frac{3c^r\log r}{Cr} + \frac{c^r}{r}\right) \le \frac{10e^{C}\lambda T}{r} + \frac{4\lambda(2\sigma_A^2 + 4\lambda^{-2})}{r^2T}.
\end{equation}
Note that $\hat{W}_a^r\left(0\right) = \sum_{l=1}^{\qq^r}\breve{v}_l^r\one_{[\breve{v}_l^r \le a]}$ for $r\in\clr$. Using this in \eqref{woub} (with $j=0$), for any $r\ge r_0$ and $k \in\NN$,
\begin{align}\label{intloweq}
\mathbb{P}\left(K_0^r > k\right) &{\color{black} \ = \mathbb{P}\left(\hat W^r_a\left(\sum_{i=0}^kT_i^r\right)>0, \ K_0^r>k\right)}\nonumber\\
&\le \mathbb{P}\left(\sum_{l=1}^{\qq^r}{\color{black}\breve{v}_l^r}\one_{[\breve{v}_l^r \le a]} + M_1(k) + M^r_2(k) - 2\eta\lambda^{-1}k{\color{black}+2\lambda^{-1}}> 0\right).
\end{align}
Further, note that by Assumption \eqref{smalljobas}, there exists $\theta \in \Theta$ such that $r\theta(r)/c^r \rightarrow \infty$ as $r \rightarrow \infty$ and
for all $r\in\clr$,
\begin{equation}\label{intloweq2}
\mathbb{P}\left(\sum_{l=1}^{\qq^r}\one_{[\breve{v}_l^r \le a]} > r \eta \la^{-1}\theta(r)/(a+1)c^r\right) \le \theta(r).
\end{equation}
Using these observations, we conclude that there exists $r_2\ge r_1$ such that for all $r\ge r_2$, 
\begin{align}\label{startsup}
\mathbb{P}&\left(\sup_{t \in [0, r^{-2}\mathcal{T}^r_l(K^r_{0})]}\frac{r}{c^r}Q_{a(c^r)^{-1}}^{r}(t) > (\eta \la^{-1} +1)\left(1+\frac{r \theta(r)}{c^r}\right)\right)\\
&\le \mathbb{P}\left(\sum_{l=1}^{\qq^r}\one_{[\breve{v}_l^r \le a]} + K^r_0 > (\eta \la^{-1} +1)\left(1+\frac{r \theta(r)}{c^r}\right)\right)\nonumber\\
&\le \mathbb{P}\left(\sum_{l=1}^{\qq^r}\one_{[\breve{v}_l^r \le a]} > \eta \la^{-1}\left(1+\frac{r \theta(r)}{c^r}\right)\right) + \mathbb{P}\left(K^r_0 >  1+\frac{r \theta(r)}{c^r}\right)\nonumber\\
& \le \mathbb{P}\left(\sum_{l=1}^{\qq^r}\one_{[\breve{v}_l^r \le a]} > r \eta \la^{-1}\theta(r)/c^r\right)\nonumber\\
&\qquad+ \mathbb{P}\left(\sum_{l=1}^{\qq^r}\breve{v}_l^r \one_{[\breve{v}_l^r \le a]} + M_1\left(\left\lfloor 1+\frac{r \theta(r)}{c^r}\right\rfloor\right)\right.\nonumber\\
&\qquad\qquad\qquad \left. + M_2^r\left(\left\lfloor 1+\frac{r \theta(r)}{c^r}\right\rfloor\right) - 2\eta \lambda^{-1}r \theta(r)/c^r +2\lambda^{-1}> 0\right)\nonumber\\
&\le \mathbb{P}\left(\sum_{l=1}^{\qq^r} \one_{[\breve{v}_l^r \le a]} > \eta \lambda^{-1}r \theta(r)/c^r\right) + \mathbb{P}\left(\sum_{l=1}^{\qq^r} \breve{v}_l^r\one_{[\breve{v}_l^r \le a]} > \eta \lambda^{-1}r \theta(r)/c^r\right)\nonumber\\
&\qquad + \mathbb{P}\left(M_1\left(\left\lfloor 1+\frac{r \theta(r)}{c^r}\right\rfloor\right) > \eta \lambda^{-1}r \theta(r)/(2c^r) +\lambda^{-1}\right)\nonumber\\
&\qquad \qquad  + \mathbb{P}\left(M^r_2\left(\left\lfloor 1+\frac{r \theta(r)}{c^r}\right\rfloor\right) > \eta \lambda^{-1}r \theta(r)/(2c^r) +\lambda^{-1}\right)\nonumber\\
&\le \mathbb{P}\left(\sum_{l=1}^{\qq^r} \one_{[\breve{v}_l^r \le a]} > \eta \lambda^{-1}r \theta(r)/c^r\right) + \mathbb{P}\left(\sum_{l=1}^{\qq^r} \one_{[\breve{v}_l^r \le a]} > \eta \lambda^{-1}r \theta(r)/a c^r\right)\nonumber\\
&\qquad + \mathbb{P}\left(M_1\left(\left\lfloor 1+\frac{r \theta(r)}{c^r}\right\rfloor\right) > \eta \lambda^{-1}r \theta(r)/(2c^r) +\lambda^{-1}\right)\nonumber\\
&\qquad \qquad + \mathbb{P}\left(M^r_2\left(\left\lfloor 1+\frac{r \theta(r)}{c^r}\right\rfloor\right) > \eta \lambda^{-1}r \theta(r)/(2c^r) +\lambda^{-1}\right)\nonumber\\
&\le 2\theta(r) + 2e^{-Cr \theta(r)/c^r},\nonumber
\end{align}
where we used \eqref{intloweq} in the third inequality and \eqref{intloweq2} and the Azuma-Hoeffding inequality in the last inequality.
Since the upper bounds in \eqref{endsup} and \eqref{startsup} tend to zero as $r \rightarrow \infty$, the lemma follows from \eqref{endsup} and \eqref{startsup}.
\end{proof}

Recall the parameter $\eta^*$ specified in Section \ref{initconass}.  We will need the following technical lemma in what follows.

\begin{lemma}\label{lem:Mrd}
Let {\color{black} $D'\ge 8p$} and $\eta\in(\eta^*,p-1)$. There exist $M_*(\eta)>1$, $r_*(\eta) \ge 1$ and $\delta_*(\eta) \in (0,1)$
such that for all $r \ge r_*(\eta)$ and $\delta \in [2M_*(\eta)(c^r)^{-1}, \delta_*(\eta)]$ the following hold:
\begin{eqnarray}
\frac{2^{3} -2}{2^{2D'\log(1/\delta)} -2}&\le& \delta^{D'},\label{SRWcomp2}\\
\lambda/2 \le \lambda^r &\le& 8\lambda/7,\label{lambdarBnd}\\
{\mathbb E}(T_1^r) &\le& 2\lambda^{-1}\label{lambdarBnd2}\\
\sigma_A^2/2 \le (\sigma_A^r)^2 &\le& 2\sigma_A^2,\label{sigmarBnd}\\
{\mathbb E}\left[\left(T_1^r-(\lambda^{r})^{-1}\right)^2\right] &\le& 2\sigma_A^2,\label{sigmarBnd2}\\
\mathbb{E}\left[\left(v\one_{[v \le \delta c^r]} - \lambda^r\mathbb{E}(v\one_{[v \le \delta c^r]})T_i^r\right)^2\right]&\le& C:=\mathbb{E}\left[v^2\right]+\frac{128\sigma_A^2}{49}, \hbox{ for all }i\in\NN, \label{C}\\
c^r &<& \left(\frac{(p+1)r}{p}\right)^{1/(p-\eta/2)},\label{crup}\\
 \left(\frac{p+1}{p}\right)^{2(p-\eta)/(p-\eta/2)}\frac{1}{Cr^{\eta/(p-\eta/2)}}&\le& \min\left(\frac{2}{\lambda^2\sigma_A^2},194\right),\label{one6th}\\
 \left(\frac{p}{p+1}\right)^{2(p-\eta)/(p-\eta/2)}r^{\eta/(p-\eta/2)}&\le&
 r^2\delta^{2(p-\eta)} ,\label{un20}\\
-\lambda^r\frac{\mathbb{E}(v \one_{[v > \delta c^r]})}{\mathbb{E}(v \one_{[v > c^r]})}  + r(\rho^r - 1) &\le& - \frac{\lambda }{4\delta^{p-\eta}}.\label{eq:NegDrift}
\end{eqnarray}
Moreover, for any $b_0>0$ and any $\eta\in(\eta^*,p-1)$, there exists $\tilde{r}(\eta,b_0) \ge r_*(\eta)$ such that for any $b \ge b_0$, $r \ge \tilde{r}(\eta,b_0)$ and $\delta \in [2M_*(\eta)(c^r)^{-1}, \delta_*(\eta)]$,
\begin{equation}
\mathbb{P}\left(E^r(3br^2\delta^{2(p-\eta)}/4) > \lfloor b \lambda r^2\delta^{2(p-\eta)}\rfloor\right)
\le \left(\frac{p+1}{p}\right)^{2(p-\eta)/(p-\eta/2)}\frac{2^9\lambda\sigma_A^2}{br^{\eta/(p-\eta/2)}}.\label{Econt1}
\end{equation}
\end{lemma}

\begin{proof}
By \eqref{eq:assuht}, \eqref{eq:assuht2}, \eqref{crscale} and other elementary considerations, there exist $M_2(\eta)>1$, $r_2(\eta) \ge 1$ and $\delta_2(\eta)\in (0,1)$ such that
\eqref{SRWcomp2}--\eqref{one6th} hold for all $r \ge r_2(\eta)$ and $\delta \in [2M_2(\eta)(c^r)^{-1}, \delta_2(\eta)]$.  Then \eqref{un20} holds for all $r \ge r_2(\eta)$
and $\delta \in [2M_2(\eta)(c^r)^{-1}, \delta_2(\eta)]$ as well, since $2M_2(\eta)\ge 2$ and \eqref{crup} imply that for all $r \ge r_2(\eta)$
and $\delta \in [2M_2(\eta)(c^r)^{-1}, \delta_2(\eta)]$,
 $$
 \left(\frac{p}{p+1}\right)^{\frac{2(p-\eta)}{p-\eta/2}}r^{\eta/(p-\eta/2)} =r^2\left(\frac{p}{(p+1)r}\right)^{\frac{2(p-\eta)}{p-\eta/2}} \le r^2(c^r)^{-2(p-\eta)}\le r^2\delta^{2(p-\eta)}.
 $$
From Section \ref{sec:regvarfunc} (c),
$\mathbb{E}(v \one_{[v > z]}) = z^{-p}\hat{L}(z)$ for all $z >0$, where $\hat{L}$ satisfies \eqref{def:hatL}
for some nonnegative Borel measurable functions $c(\cdot)$ and $\epsilon(\cdot)$, with $c(\cdot)$ satisfying $\lim_{x \rightarrow \infty}c(x) = c_0 \in (0,\infty)$ and
$\epsilon(\cdot)$ satisfying $\epsilon(y) \rightarrow 0$ as $y \rightarrow \infty$.
For any $\eta>0$, we can obtain $M_3(\eta)\ge M_2(\eta)$ such that for all $y,z\ge M_3(\eta)$, $\frac{c(y)}{c(z)} \ge 1/2$ and $\epsilon(y) < \eta$. Hence,
for all 
$\delta \in (0,1)$ and $z \ge M_3(\eta)/\delta$, 
$$
\frac{\hat{L}(\delta z)}{\hat{L}(z)} = \frac{c(\delta z)}{c(z)}\exp\left(-\int_{\delta z}^z\frac{\epsilon(y)}{y}dy\right) \ge \frac{\exp\left(-\eta\int_{\delta z}^z\frac{1}{y}dy\right)}{2} = \frac{\exp\left(-\eta\log (1/\delta)\right)}{2} = \frac{\delta^{\eta}}{2}.
$$
Thus, for all $\delta \in (0,1)$ and $z \ge M_3(\eta)/\delta$,
$$
\frac{\mathbb{E}(v \one_{[v > \delta z]})}{\mathbb{E}(v \one_{[v > z]})} = \frac{(\delta z)^{-p}\hat{L}(\delta z)}{z^{-p}\hat{L}(z)} \ge \frac{1}{2\delta^{p-\eta}}.
$$
From this it follows that for some $0<\delta_3(\eta)\le\delta_2(\eta)$, \eqref{eq:NegDrift} holds for all $r \ge r_2(\eta)$ and $\delta \in [2M_3(\eta)(c^r)^{-1}, \delta_3(\eta)]$. Setting $r_*(\eta)=r_2(\eta)$, $M_*(\eta)=M_3(\eta)$ and $\delta_*(\eta)=\delta_3(\eta)$ completes the proof of \eqref{SRWcomp2}-\eqref{eq:NegDrift}.

To prove \eqref{Econt1}, note that for any $b_0>0$, by \eqref{un20}, we can choose $\tilde{r}(\eta,b_0) \ge r_*(\eta)$ such that for all $r \ge \tilde{r}(\eta,b_0)$ and $\delta \in [2M_*(\eta)(c^r)^{-1}, \delta_*(\eta)]$,
\begin{equation}\label{lalala}
r^2\delta^{2(p-\eta)} \ge 28(\lambda b_0)^{-1}. 
\end{equation}
Using \eqref{lambdarBnd} in the third line below, \eqref{lalala} in the fifth line below, and Chebychev's inequality, \eqref{sigmarBnd}, and \eqref{un20} in the sixth line below,
for all $b \ge b_0$, $r\ge \tilde{r}(\eta,b_0)$, $\delta \in [2M_*(\eta)(c^r)^{-1}, \delta_*(\eta)]$,
\begin{align*}
\mathbb{P}&\left(E^r(3br^2\delta^{2(p-\eta)}/4) > \lfloor b \lambda r^2\delta^{2(p-\eta)}\rfloor\right)
\le \mathbb{P}\left(\sum_{i=2}^{\lfloor b \lambda r^2\delta^{2(p-\eta)} \rfloor}T_i^r <  \frac{3b r^2\delta^{2(p-\eta)}}{4}\right)\nonumber\\
&= \mathbb{P}\left(\sum_{i=2}^{\lfloor b\lambda r^2\delta^{2(p-\eta)} \rfloor}\left(T_i^r - \mathbb{E}(T_i^r)\right)<  \frac{3b r^2\delta^{2(p-\eta)}}{4} - \left(\lfloor b \lambda r^2\delta^{2(p-\eta)} \rfloor -1\right)(\lambda^r)^{-1}\right)\nonumber\\
&\le \mathbb{P}\left(\sum_{i=2}^{\lfloor b \lambda r^2\delta^{2(p-\eta)} \rfloor}\left(T_i^r - \mathbb{E}(T_i^r)\right)< \frac{3b r^2\delta^{2(p-\eta)}}{4} - \frac{7\lambda^{-1}}{8}\left(b \lambda r^2\delta^{2(p-\eta)} - 2\right)\right)\nonumber\\
&= \mathbb{P}\left(\sum_{i=2}^{\lfloor b \lambda r^2\delta^{2(p-\eta)} \rfloor}\left(T_i^r - \mathbb{E}(T_i^r)\right)<  -\frac{br^2\delta^{2(p-\eta)}}{8} + \frac{14\lambda^{-1}}{8}\right)\nonumber\\
&\le \mathbb{P}\left(\sum_{i=2}^{\lfloor b \lambda r^2\delta^{2(p-\eta)} \rfloor}\left(T_i^r - \mathbb{E}(T_i^r)\right)<  -\frac{b r^2\delta^{2(p-\eta)}}{16}\right)\nonumber\\
& \le \frac{2^8 \lambda (\sigma_A^r)^2}{br^2\delta^{2(p-\eta)}} \le \left(\frac{p+1}{p}\right)^{2(p-\eta)/(p-\eta/2)}\frac{2^9\lambda\sigma_A^2}{br^{\eta/(p-\eta/2)}}.\nonumber
\end{align*}
Hence \eqref{Econt1} holds for all $b \ge b_0$, $r\ge \tilde{r}(\eta,b_0)$, $\delta \in [2M_*(\eta)(c^r)^{-1}, \delta_*(\eta)]$.
\end{proof}

\begin{lemma}\label{unifsmallint}
Fix $T>0$. There exist $D_1$, $D_2$, $D_3>0$ such that the following holds: For any $\eta \in (\eta^*, p-1)$, there exist $M(\eta)>1$, $r(\eta) \ge 2$, and $\delta(\eta) \in (0,1)$ such that for all $r \ge r(\eta)$ and $\delta \in [2M(\eta)(c^r)^{-1}, \delta(\eta)]$,
\begin{multline*}
\mathbb{P}\left(\sup_{t \in [0,T]}(Q_{\delta}^r(t) - Q_{\delta/2}^r(t)) > D_1 \delta^{p-1-\eta}\log (\delta^{-1}) + \frac{c^r}{r}\right)\\
 \le 35\delta^{D_2} + \mathbb{P}\left(\frac{1}{r}\sum_{l=1}^{\qq^r}\breve{v}_l^r\one_{[\breve{v}_l^r \le \delta c^r]} > D_3\delta^{p-\eta}\right).
\end{multline*}
\end{lemma}

\begin{proof} Fix {\color{black}$D'\ge 8p$}. For $\eta \in (\eta^*, p-1)$ and $b_0>0$, recall $M_*(\eta)$, $r_*(\eta)$, $\delta_*(\eta)$ and $\tilde{r}(\eta,b_0)$ from Lemma \ref{lem:Mrd}. Set $M(\eta)=M_*(\eta)$ and take $r\ge \max\{r_*(\eta), \tilde{r}(\eta, \lambda^{-1})\}$ and
$\delta\in[2M(\eta)(c^r)^{-1}, \delta_*(\eta)]$. For all $t\ge 0$, by Lemma \ref{qlxy} with $x=\delta/2$ and $y = \delta$, 
\begin{equation}\label{un2}
0 \le Q_{\delta}^r(t) - Q_{\delta/2}^r(t) \le \frac{c^r}{r} + 2\delta^{-1}Y_{\delta}^r(t).
\end{equation}
The major effort of the proof will be to obtain bounds on the probability
that $\sup_{0\le t\le T'}2\delta^{-1}Y_{\delta}^r(t)$ exceeds certain bounds for a suitable $T'\ge T$, which entails a detailed analysis of its excursions. 
To get an overview of the strategy for this, the reader may wish to look ahead to \eqref{eq:chiless1} (where $r(\eta)$, $B\ge 1$ and $\epsilon\in(0,1)$
are constants to be determined in what follows), definitions \eqref{def:taueven}, \eqref{def:tauodd} and \eqref{def:scrN}, \eqref{eq:tauineq} and
\eqref{eq:chiless10}.  In what follows, each of the three terms on the right side of \eqref{eq:chiless10} is bounded above using estimates in \eqref{Nless}, \eqref{chiless}, and \eqref{eq:initterm}.

Recall that $Y_{\delta}^r(t) = \Gamma[X_{\delta}^r](t)$ for $t\ge 0$.  From \eqref{Zdec} and \eqref{eq:htcr}, for $t\ge 0$,
\begin{align*}
X_{\delta}^r(t) 
=  X_{\delta}^r(0)+\hat{V}_{\delta}^r(t) -\lambda^rt\frac{\mathbb{E}(v \one_{[v > \delta c^r]})}{\mathbb{E}(v \one_{[v > c^r]})}  + rt(\rho^r - 1).
\end{align*}
By \eqref{eq:NegDrift}, for all
$0\le s \le t$, 
\begin{equation}\label{XdeltaU}
X_{\delta}^r(t)- X_{\delta}^r(s) \le U^r_{\delta}(t) - U^r_{\delta}(s), \mbox{ where } U^r_{\delta}(t) := \hat{V}_{\delta}^r(t) - \frac{\lambda t}{4\delta^{p-\eta}}.
\end{equation}
For 
$k \in\NN$, write
\begin{eqnarray*}
\widetilde{V}^r_{\delta}(k) &:=& r\hat{V}_{\delta}^r\left(r^{-2}\sum_{i=1}^{k} T_i^r\right)= \sum_{i=1}^{k}v_i\one_{[v_i \le \delta c^r]} - \lambda^r\mathbb{E}(v\one_{[v \le \delta c^r]})\sum_{i=1}^kT_i^r
\end{eqnarray*}
By \eqref{C}, for each $k \in\NN$,
$$
\mathbb{E}\left[\left(\widetilde{V}^r_{\delta}(k)\right)^2\right] \le Ck.
$$
Take any $B\ge1$.
Thus, as 
$\{\widetilde{V}^r_{\delta}(k)\}_{k \in\NN}$ is a martingale (with respect to the filtration $\{\mathcal{F}^r_n\}_{n\in\bbZp}$, where $\mathcal{F}^r_0 := \{\qq^r, \breve{v}_l^r: l \in \NN\}$ and $\mathcal{F}^r_n := \sigma\left(\qq^r, \breve{v}_l^r, T^r_i, v_i : l \in \NN, i \le n\right)$, $n \ge 1$),
using Doob's maximal inequality, \eqref{crup}, recalling $r \ge \tilde{r}(\eta, \lambda^{-1})$ and using \eqref{Econt1} with $b=16B\lambda^{-1}$, 
\begin{align}\label{vhatsup}
\mathbb{P}&\left(\sup_{t \in [0, 12B\lambda^{-1}\delta^{2(p-\eta)}]}\hat{V}_{\delta}^r(t) > B\delta^{p-\eta}/2\right)\\
& = \mathbb{P}\left(\sup_{1 \le k \le E^r\left(12B\lambda^{-1}r^2\delta^{2(p-\eta)}\right)}\hat{V}_{\delta}^r\left(r^{-2}\sum_{i=1}^{k} T_i^r\right) > B\delta^{p-\eta}/2\right)\nonumber\\
&\le \mathbb{P}\left(\sup_{1 \le k \le \lfloor16Br^2\delta^{2(p-\eta)}\rfloor}\hat{V}_{\delta}^r\left(r^{-2}\sum_{i=1}^{k} T_i^r\right) > B\delta^{p-\eta}/2\right)\nonumber\\
&\qquad \qquad \qquad \qquad + \mathbb{P}\left(E^r(12Br^2\lambda^{-1}\delta^{2(p-\eta)}) > \lfloor16Br^2\delta^{2(p-\eta)}\rfloor\right)\nonumber\\
&= \mathbb{P}\left(\sup_{1 \le k \le \lfloor16Br^2\delta^{2(p-\eta)}\rfloor}\widetilde{V}^r_{\delta}(k) > Br\delta^{p-\eta}/2\right)\nonumber\\
&\qquad \qquad \qquad \qquad + \mathbb{P}\left(E^r(12Br^2\lambda^{-1}\delta^{2(p-\eta)}) > \lfloor16Br^2\delta^{2(p-\eta)}\rfloor\right)\nonumber\\
&\le \frac{{\color{black} 16}\mathbb{E}\left[\left(\widetilde{V}^r_{\delta}(\lfloor16Br^2\delta^{2(p-\eta)}\rfloor)\right)^2\right]}{B^2r^2\delta^{2(p-\eta)}}\nonumber\\
&\qquad \qquad \qquad \qquad + \mathbb{P}\left(E^r(12Br^2\lambda^{-1}\delta^{2(p-\eta)}) > \lfloor16Br^2\delta^{2(p-\eta)}\rfloor\right)\nonumber\\
&\le \frac{{\color{black} 256} C}{B} + \left(\frac{p+1}{p}\right)^{2(p-\eta)/(p-\eta/2)}\frac{32\lambda^2\sigma_A^2}{Br^{\eta/(p-\eta/2)}}.\nonumber
\end{align}
Next, from \eqref{XdeltaU}, we see that for any 
integer $i \ge 2$ and $s\ge 0$,
\begin{align}\label{un2.5}
\mathbb{P}&\left(X_{\delta}^r(\cdot + s) \text{ crosses } (i+1)B\delta^{p-\eta} \text{ before } (i-2)B\delta^{p-\eta}\  \Big| \ X_{\delta}^r(s) = iB\delta^{p-\eta},\right.\\
 &\hspace{9cm} \left. E^r(r^2s) - E^r(r^2s-)>0\right)\nonumber\\
&\le \mathbb{P}\left(\sup_{t \in [0, 12B\lambda^{-1}\delta^{2(p-\eta)}]}\hat{V}_{\delta}^r(t+s)-\hat{V}_{\delta}^r(s)  > B\delta^{p-\eta}/2\ \Big| \ X_{\delta}^r(s) = iB\delta^{p-\eta},\right.\nonumber\\
 &\hspace{9cm} \left. E^r(r^2s) - E^r(r^2s-)>0\right)\nonumber\\
&=  \mathbb{P}\left(\sup_{t \in [0, 12B\lambda^{-1}\delta^{2(p-\eta)}]}\hat{V}_{\delta}^r(t) > B\delta^{p-\eta}/2\right)\nonumber
\end{align}
where, in the last step, we have used the strong Markov property of the process $\hat{V}_{\delta}^r(\cdot)$ at the jump times of the process $t \mapsto E^r(r^2 t)$.
Combining  \eqref{vhatsup} and \eqref{un2.5}, setting {\color{black}$B=960C\vee1$}, and using \eqref{one6th},
we conclude that 
for all 
integers $i \ge 2$, $0\le x \le iB\delta^{p-\eta}$, and $s\ge 0$,
\begin{align}\label{un3}
&\mathbb{P}\left(X_{\delta}^r(\cdot + s) \text{ crosses } (i+1)B\delta^{p-\eta} \text{ before } (i-2)B\delta^{p-\eta}\  \Big| \ X_{\delta}^r(s) = x,\right.\\
&\hspace{9.5cm} \left. E^r(r^2s) - E^r(r^2s-)>0\right)\nonumber\\
&\le \mathbb{P}\left(X_{\delta}^r(\cdot + s) \text{ crosses } (i+1)B\delta^{p-\eta} \text{ before } (i-2)B\delta^{p-\eta}\  \Big| \ X_{\delta}^r(s) = iB\delta^{p-\eta},\right.\nonumber\\
 &\hspace{9.5cm} \left. E^r(r^2s) - E^r(r^2s-)>0\right)\nonumber\\
&\le \frac{1}{3}.\nonumber
\end{align}
Using $M(\eta)/c^r\le 2M(\eta)/c^r\le\delta$, $M(\eta)=M_*(\eta)>1$, \eqref{crup}, \eqref{one6th}, $r_*(\eta) \ge 1$ and {\color{black} $B\ge 960 C> 2C$},  
\begin{align}\label{delbd}
\delta c^r/r &= \delta^{p-\eta}\frac{c^r}{r\delta^{p-\eta-1}} \le \delta^{p-\eta}\frac{(c^r)^{p-\eta}}{M(\eta)^{p-\eta-1}r}\\
& \le \delta^{p-\eta}\left(\frac{p+1}{p}\right)^{(p-\eta)/(p-\eta/2)}\frac{r^{(p-\eta)/(p-\eta/2)}}{r}\nonumber\\
& = \delta^{p-\eta}\left(\frac{p+1}{p}\right)^{(p-\eta)/(p-\eta/2)}r^{-\eta/(2p-\eta)} < B\delta^{p-\eta}/2.\nonumber
\end{align}
For
$s\ge 0$, define the following stopping times with respect to the filtration $\{\mathcal{H}_t\}_{t \ge 0}$ given by
$\mathcal{H}_t := \{\qq^r, \breve{v}_l^r, V^r_{\delta}(r^2s), E^r(r^2s) : l \in \NN, s \le t\}$ for $t\ge 0$: 
$\beta_0=s$ and for $k \in\bbZp$, 
\begin{align*}
\beta_{k+1} &:=\inf\{t \ge \beta_k: X_{\delta}^r(t) - X_{\delta}^r(\beta_k) \ge B\delta^{p-\eta}\\
&\qquad \qquad \text{ or } E^r(r^2t) - E^r(r^2t-)>0 \text{ and } X_{\delta}^r(t-) - X_{\delta}^r(\beta_k) \le -2B\delta^{p-\eta}\},
\end{align*}
and write $\widetilde{X}^r_{\delta}(k) := X_{\delta}^r(\beta_k)$.
For any 
$k\in\bbZ_+$,
note that if $X_{\delta}^r(\beta_{k+1}) - X_{\delta}^r(\beta_k) \ge B\delta^{p-\eta}$; that is, if $\beta_{k+1}$ corresponds to an up-crossing of $X_{\delta}^r$, then,  using \eqref{delbd} and that jumps up of $X_{\delta}^r(\cdot)$ are at most of size $\delta c^r/r$, $X_{\delta}^r(\beta_{k+1}) - X_{\delta}^r(\beta_k) \le B\delta^{p-\eta} + \delta c^r/r \le 3B\delta^{p-\eta}/2$. Similarly, for any 
$k\in\bbZ_+$, if $X_{\delta}^r(\beta_{k+1}-) - X_{\delta}^r(\beta_k) \le -2B\delta^{p-\eta}$, then, by 
the same line of reasoning,
$X_{\delta}^r(\beta_{k+1}) - X_{\delta}^r(\beta_k) \le -2B\delta^{p-\eta} + \delta c^r/r \le -3B\delta^{p-\eta}/2$.
Let $\{S_{\delta}(k)\}_{k \in\bbZp}$ be a random walk with $S_{\delta}(0) = 9B\delta^{p-\eta}/2$ and for $k \in\bbZp$,
$$
\mathbb{P}(S_{\delta}(k+1) - S_{\delta}(k) = 3B\delta^{p-\eta}/2) = 1/3 \ \ \ \hbox{and} \ \ \ \mathbb{P}(S_{\delta}(k+1) - S_{\delta}(k)  =-3B\delta^{p-\eta}/2) = 2/3.
$$
Recall that $D'\ge 8p$ was fixed at the onset.  Also note that \eqref{SRWcomp2} implies that $\frac{2^{3} -2}{2^{2D'\log(1/\delta)} -2}\le 1$, which in turn implies that $9/2\le 3D'\log(1/\delta)$.
Then, from \eqref{un3}, the above observations, \eqref{SRWcomp2}, and 
by comparing the sequence 
$\{\widetilde{X}^r_{\delta}(k)\}_{k\in\bbZp}$ with $\{S_{\delta}(k)\}_{k\in\bbZp}$, it follows that, for any $t\ge 0$ and any $x_0 \in [4B\delta^{p-\eta}, 9B\delta^{p-\eta}/2]$,
\begin{align}\label{SRWcomp}
&\mathbb{P}\left(Y_{\delta}^r(t+\cdot) \text{ crosses } 3D'B\delta^{p-\eta}\log(\delta^{-1}) \text{ before } \frac{3B\delta^{p-\eta}}{2}\  \Big| \ Y_{\delta}^r(t) = x_0,\right.\\
&\hspace{10cm} \left.E^r(r^2t) - E^r(r^2t-)>0\right)\nonumber\\
&=\mathbb{P}\left(X_{\delta}^r(t+\cdot) \text{ crosses } 3D'B\delta^{p-\eta}\log(\delta^{-1}) \text{ before } \frac{3B\delta^{p-\eta}}{2}\  \Big| \ X_{\delta}^r(t) = x_0,\right.\nonumber\\
&\hspace{10cm} \left.E^r(r^2t) - E^r(r^2t-)>0\right)\nonumber\\
&\le \mathbb{P}\left(X_{\delta}^r(t+\cdot) \text{ crosses } 3D'B\delta^{p-\eta}\log(\delta^{-1}) \text{ before } \frac{3B\delta^{p-\eta}}{2}\  \Big| \ X_{\delta}^r(t) = \frac{9B\delta^{p-\eta}}{2},\right.\nonumber\\ 
&\hspace{10cm} \left.E^r(r^2t) - E^r(r^2t-)>0\right)\nonumber\\
&\le \mathbb{P}\left(S_{\delta}(\cdot) \text{ crosses } 3D'B\delta^{p-\eta}\log(\delta^{-1}) \text{ before } \frac{3B\delta^{p-\eta}}{2} \right)
\le \frac{2^{3} -2}{2^{2D'\log(1/\delta)} -2}\le \delta^{D'},\nonumber
\end{align}
where, in the second to the last inequality above, we have used the fact that, for the biased random walk $S_{\delta}$, $n \mapsto 2^{2S_{\delta}(n)/(3B\delta^{p-\eta})}$ is a martingale (with respect to the natural filtration generated by $S_{\delta}$) to compute the probability via optional stopping theorem.  
Define the following stopping times (with respect to the filtration $\{\mathcal{H}_t\}_{t \ge 0}$ defined above): $\tau_{-1} = 0$ 
and for $k \in\bbZp$,
\begin{align}
	\tau_{2k} &:=\inf\{t \ge \tau_{2k-1} : E^r(r^2t) - E^r(r^2t-)>0
	 \text{ and } Y^r_{\delta}(t-)\le 2B\delta^{p-\eta}\},\label{def:taueven}\\
\tau_{2k+1} &:=\inf\{t \ge \tau_{2k} :  Y^r_{\delta}(t) \ge 4B\delta^{p-\eta}\},\label{def:tauodd}
\end{align}
and let 
\begin{equation}\label{def:scrN}
\mathcal{N} := \inf\left\lbrace k \in\NN: \sup_{t \in [\tau_{2k-1}, \tau_{2k}]}Y^r_{\delta}(t) \ge 3D'B\delta^{p-\eta}\log(1/\delta)\right\rbrace.
\end{equation}
Due to \eqref{delbd} and since $Y_{\delta}^r$ has upward jumps of size at most $c^r\delta/r$, for each $k\in\NN$, $Y_{\delta}^r(\tau_{2k-1})\in[4B\delta^{p-\eta},9B\delta^{p-\eta}/2]$. As $\delta\le \delta_*(\eta)<1$, by \eqref{SRWcomp}, 
\begin{equation}\label{Nless}
\mathbb{P}(\mathcal{N} \le \lfloor \delta^{-D'/2} \rfloor + 1 ) \le \sum_{k=1}^{\lfloor \delta^{-D'/2} \rfloor + 1} \mathbb{P}\left( \sup_{t \in [\tau_{2k-1}, \tau_{2k}]}Y^r_{\delta}(t) \ge 3D'B\delta^{p-\eta}\log(1/\delta)\right)\le 2\delta^{D'/2}.
\end{equation}
Using \eqref{delbd}, $Y^r_{\delta}(\tau_{2k}) \le 3B\delta^{p-\eta}$ for all $k \in\bbZp$. From \eqref{XdeltaU}, 
it follows that, for each 
$k \in\bbZp$,
$$
t \mapsto (\hat{V}^r_{\delta}(t + \tau_{2k}) - \hat{V}^r_{\delta}(\tau_{2k})) - (X^r_{\delta}(t + \tau_{2k}) - X^r_{\delta}(\tau_{2k})), \ t \ge 0,
$$ 
is nondecreasing in $t$. Thus, by the monotonicity property noted in \eqref{eq:monoskor},
for each
$k \in\bbZp$ and $t\ge 0$,
\begin{eqnarray*}
Y^r_{\delta}(t + \tau_{2k}) &=& \Gamma\left[Y^r_{\delta}(\tau_{2k}) + (X^r_{\delta}(\cdot + \tau_{2k}) - X^r_{\delta}(\tau_{2k}))\right](t)\\
&\le& \Gamma\left[Y^r_{\delta}(\tau_{2k}) + (\hat{V}^r_{\delta}(\cdot + \tau_{2k}) - \hat{V}^r_{\delta}(\tau_{2k}))\right](t)\\
&\le&  \Gamma\left[3B\delta^{p-\eta} + (\hat{V}^r_{\delta}(\cdot + \tau_{2k}) - \hat{V}^r_{\delta}(\tau_{2k}))\right](t).
\end{eqnarray*}
For each 
$k \in\bbZp$, a job arrives to the $r$-th system at time $\tau_{2k}$.  Hence,
by the strong Markov property, $\{\Gamma\left[3B\delta^{p-\eta} + (\hat{V}^r_{\delta}(\cdot + \tau_{2k}) - \hat{V}^r_{\delta}(\tau_{2k}))\right](t)\ : \ t \ge 0 \}$ has the same distribution as the process
$\{\Gamma\left[3B\delta^{p-\eta} + \hat{V}^r_{\delta}(\cdot)\right](t) \ : t \ge 0\}$.
Thus, for each
$d \in\NN$ and $t \ge 0$,
\begin{equation}\label{dom}
\mathbb{P}\left(\sum_{j=0}^{d}(\tau_{2j+1} -\tau_{2j}) \le t\right) \le \mathbb{P}\left(\sum_{j=0}^{d}\chi_j\le t\right),
\end{equation}
where $\{\chi_0,\chi_1,\dots\}$ are independent and identically distributed random variables distributed as
$$
\mathbb{P}\left(\chi_0 \le s\right) = \mathbb{P}\left(\sup_{t \in [0, s]} \Gamma\left[3B\delta^{p-\eta} + \hat{V}^r_{\delta}(\cdot)\right](t) \ge 4B \delta^{p-\eta}\right), \quad s \ge 0.
$$
Recalling $\delta \le \delta_*(\eta) < 1$ and using the Lipschitz property of the Skorohod map noted in \eqref{eq:lipSM}, we obtain that for all
$\epsilon \in (0,1)$,
\begin{equation}\label{taugap}
\mathbb{P}\left(\chi_0 \le \epsilon \delta^{2(p-\eta)}\right) \le \mathbb{P}\left(\sup_{t \in [0, \epsilon \delta^{2(p-\eta)}]} |\hat{V}^r_{\delta}(t)| \ge B \delta^{p-\eta}/2 \right).
\end{equation}
Then given $\epsilon\in(0,1)$, following the same line of reasoning used to obtain
\eqref{vhatsup} and using \eqref{Econt1} with $b=2\epsilon$ (noting $3b/4 > \epsilon$), we obtain for $r \ge \hat{r}(\eta, \epsilon) := \max\{r_*(\eta), \tilde{r}(\eta, \lambda^{-1}), \tilde{r}(\eta,\epsilon)\}$,
\begin{align}\label{vsup}
\mathbb{P}&\left(\sup_{t \in [0, \epsilon \delta^{2(p-\eta)}]}\hat{V}_{\delta}^{r}(t) \ge B\delta^{p-\eta}/4\right)\\
&\le \mathbb{P}\left(\sup_{1 \le k \le \lfloor2\epsilon \lambda r^2\delta^{2(p-\eta)}\rfloor}\hat{V}_{\delta}^{r}\left(r^{-2}\sum_{i=1}^{k} T_i^r\right) > B\delta^{p-\eta}/4\right)\nonumber\\ 
&\qquad + \mathbb{P}\left(E^{r}(\epsilon {r}^2\delta^{2(p-\eta)}) > \lfloor2\epsilon \lambda r^2\delta^{2(p-\eta)}\rfloor\right)\nonumber\\
 &\le \frac{128C \lambda \epsilon}{B^2} + \left(\frac{p+1}{p}\right)^{2(p-\eta)/(p-\eta/2)}\frac{2^8\lambda\sigma_A^2}{\epsilon r^{\eta/(p-\eta/2)}}.\nonumber
\end{align}
Moreover, as $\hat{V}^r_{\delta}(\cdot)$ decreases between successive arrivals of jobs and increases at the
arrival times, for each $\epsilon\in(0,1)$, we have the following lower bound on $\hat{V}^r_{\delta}(\cdot)$
on the time interval $[0, \epsilon \delta^{2(p-\eta)}]$:
\begin{align}\label{vinf0}
\inf_{t \in [0, \epsilon \delta^{2(p-\eta)}]} \hat{V}^{r}_{\delta}(t) &\ge \inf_{0 \le k \le E^{r}(\epsilon r^2 \delta^{2(p-\eta)})} \left(r^{-1}\sum_{i=1}^{k}v_i\one_{[v_i \le \delta c^{r}]} - \lambda^{r}r^{-1}\mathbb{E}(v\one_{[v \le \delta c^{r}]})\sum_{i=1}^{k+1}T_i^{r}\right)\nonumber\\
&\ge \inf_{0 \le k \le E^{r}(\epsilon r^2 \delta^{2(p-\eta)})} \left(r^{-1}\sum_{i=1}^{k}v_i\one_{[v_i \le \delta c^{r}]} - \lambda^{r}r^{-1}\mathbb{E}(v\one_{[v \le \delta c^{r}]})\sum_{i=1}^{k}T_i^{r}\right)\nonumber\\
&\qquad \qquad \qquad - \frac{8\lambda \mathbb{E}(v)}{7r}\sup_{1 \le k \le E^{r}(\epsilon r^2 \delta^{2(p-\eta)})+1}T^{r}_k\nonumber\\
&= \frac{1}{r}\inf_{1 \le k \le E^{r}(\epsilon r^2 \delta^{2(p-\eta)})}\widetilde{V}^{r}_{\delta}(k)- \frac{8}{7r}\sup_{1 \le k \le E^{r}(\epsilon r^2 \delta^{2(p-\eta)})+1}T^{r}_k,
\end{align}
where the bound \eqref{lambdarBnd} was used in the last term.
Once again, following the arguments for obtaining \eqref{vhatsup} in a manner similar to those that arrive at \eqref{vsup}, for $\epsilon\in(0,1)$ and $r \ge \hat{r}(\eta, \epsilon)$,
\begin{multline}\label{vinf1}
\mathbb{P}\left(\frac{1}{r}\inf_{1 \le k \le E^r(\epsilon r^2 \delta^{2(p-\eta)})}\widetilde{V}^r_{\delta}(k) < -B\delta^{p-\eta}/8\right)\\
 \le \frac{512C\lambda \epsilon}{B^2}
+ \left(\frac{p+1}{p}\right)^{2(p-\eta)/(p-\eta/2)}\frac{2^8\lambda\sigma_A^2}{\epsilon r^{\eta/(p-\eta/2)}}.
\end{multline}
Moreover, for any $\epsilon\in(0,1)$,
\begin{multline}\label{vinf2}
\mathbb{P}\left(\frac{8}{7r}\sup_{1 \le k \le E^r(\epsilon r^2 \delta^{2(p-\eta)})+1}T^r_k> B\delta^{p-\eta}/8\right) 
= \mathbb{P}\left(\sup_{1 \le k \le E^r(\epsilon r^2 \delta^{2(p-\eta)})+1}T^r_k> \frac{7B\delta^{p-\eta}r}{64}\right)\\
 \le \mathbb{P}\left(E^r(\epsilon r^2 \delta^{2(p-\eta)}) > \lfloor 2\epsilon \lambda r^2 \delta^{2(p-\eta)}\rfloor\right) + \mathbb{P}\left(\sup_{1 \le k \le \lfloor 2\epsilon \lambda r^2 \delta^{2(p-\eta)}\rfloor+1}T^r_k> \frac{7B\delta^{p-\eta}r}{64}\right).
\end{multline}
Applying a union bound, Chebychev's inequality, and \eqref{lambdarBnd}--\eqref{sigmarBnd2}, it follows that for any $\epsilon\in(0,1)$,
\begin{multline*}
\mathbb{P}\left(\sup_{1 \le k \le \lfloor 2\epsilon \lambda r^2 \delta^{2(p-\eta)}\rfloor+1}T^r_k> \frac{7B\delta^{p-\eta}r}{64}\right) \le (2\epsilon \lambda r^2 \delta^{2(p-\eta)}+1)\max_{k=1,2}\mathbb{P}\left(T^r_k> \frac{7B\delta^{p-\eta}r}{64}\right)\\
 \le (2\epsilon \lambda r^2 \delta^{2(p-\eta)}+1)\left(\frac{64}{7B\delta^{p-\eta}r}\right)^2 \left(\mathbb{E}\left[\left(T^r_1\right)^2\right]\vee\mathbb{E}\left[\left(T^r\right)^2\right]\right) \le 
  \frac{ (2\lambda + (\epsilon r^2 \delta^{2(p-\eta)})^{-1}) \epsilon C_1}{B^2},
 \end{multline*}
 where $C_1=10^2 (2\sigma_A^2 + (2\lambda^{-1})^2)$. Thus, for $\epsilon\in(0,1)$ and $r \ge \hat{r}(\eta, \epsilon)$, by using the above bound and \eqref{Econt1} with $b=2\epsilon$ in \eqref{vinf2},
 we obtain
\begin{multline}\label{vinf3}
\mathbb{P}\left(\frac{8}{7r}\sup_{1 \le k \le E^r(\epsilon r^2 \delta^{2(p-\eta)})}T^r_k> B\delta^{p-\eta}/8\right) \\
\le \left(\frac{p+1}{p}\right)^{2(p-\eta)/(p-\eta/2)}\frac{2^8\lambda\sigma_A^2}{\epsilon r^{\eta/(p-\eta/2)}} +  
 \frac{ (2\lambda + (\epsilon r^2 \delta^{2(p-\eta)})^{-1}) \epsilon C_1}{B^2}.
\end{multline}
From \eqref{vinf0}, \eqref{vinf1} and \eqref{vinf3}, for $\epsilon\in(0,1)$ and $r \ge \hat{r}(\eta, \epsilon)$,
\begin{multline}\label{vinf4}
\mathbb{P}\left(\inf_{t \in [0, \epsilon \delta^{2(p-\eta)}]} \hat{V}^r_{\delta}(t)< -B\delta^{p-\eta}/4\right)\\
 \le\frac{512 C\lambda \epsilon}{B^2} + 2\left(\frac{p+1}{p}\right)^{2(p-\eta)/(p-\eta/2)}\frac{2^8\lambda\sigma_A^2}{\epsilon r^{\eta/(p-\eta/2)}} +
  \frac{ (2\lambda + (\epsilon r^2 \delta^{2(p-\eta)})^{-1}) \epsilon C_1}{B^2}.
\end{multline}
From \eqref{vsup}, \eqref{vinf4}, and as $B\ge1$, we can fix $\epsilon\in(0,1)$ and find $\hat r(\eta) \ge \hat{r}(\eta, \epsilon)$ such that for all $r \ge \hat r(\eta)$,
$$
\mathbb{P}\left(\sup_{t \in [0, \epsilon \delta^{2(p-\eta)}]} |\hat{V}^{r}_{\delta}(t)| \ge B \delta^{p-\eta}/2\right) \le 1/2,
$$
 and hence, from \eqref{taugap},
\begin{equation}\label{chibd}
\mathbb{P}\left(\chi_0 \le \epsilon \delta^{2(p-\eta)}\right) \le 1/2.
 \end{equation}
Henceforth, we fix such an $\epsilon$ and assume $r\ge \hat r(\eta)$.
Applying the Azuma-Hoeffding inequality on the martingale (with respect to its natural filtration)
$$
M^{\chi}_{\ell} := \sum_{k=1}^{\ell}\left(\one_{[\chi_k > \epsilon \delta^{2(p-\eta)}]} - \mathbb{P}\left(\chi_0 > \epsilon \delta^{2(p-\eta)}\right)\right),\qquad \ell\in\bbZp,
$$
and using \eqref{dom} and \eqref{chibd}, for any $d \ge 1$, we obtain 
\begin{multline}\label{chiless}
\mathbb{P}\left(\sum_{j=0}^{d}(\tau_{2j+1} -\tau_{2j}) \le d \epsilon \delta^{2(p-\eta)}/4\right) \le \mathbb{P}\left(\sum_{j=0}^{d}\chi_j\le d \epsilon \delta^{2(p-\eta)}/4\right)\\
\le \mathbb{P}\left(\sum_{j=1}^{d}\one_{[\chi_j > \epsilon \delta^{2(p-\eta)}]} \le d/4\right)
= \mathbb{P}\left(M^{\chi}_d + d\mathbb{P}\left(\chi_0 > \epsilon \delta^{2(p-\eta)}\right) \le d/4\right)\\
\le  \mathbb{P}\left(M^{\chi}_d  \le -d/4\right) \le e^{-d/32}.
\end{multline}

Note that if $Y^r_{\delta}(\tilde{t}) \le 3B\delta^{p-\eta}/2$ for some $\tilde{t}<\tau_0$, then, by definition \eqref{def:taueven}, the time of the arrival immediately following $\tilde{t}$ corresponds to $\tau_0$. By \eqref{delbd}, $Y^r_{\delta}(\tau_0) \le 3B\delta^{p-\eta}/2+\frac{c^r\delta}{r}< 2B\delta^{p-\eta}$, and as $Y^r_{\delta}(\cdot)$ is nonincreasing in the time interval $[\tilde{t}, \tau_0)$, $\sup_{t \in [\tilde{t}, \tau_0]}Y^r_{\delta}(t) < 2B\delta^{p-\eta}$. Consequently, if $Y^r_{\delta}(\cdot)$ attains any value $v > 2B\delta^{p-\eta}$ before $\tau_0$, $Y^r_{\delta}(0) > 3B\delta^{p-\eta}/2$ and the time at which $v$ is attained must be before $Y^r_{\delta}(\cdot)$ down crosses $3B\delta^{p-\eta}/2$. Thus, from the computation \eqref{SRWcomp}, recalling that $Y^r_{\delta}(0) = \frac{1}{r}\sum_{l=1}^{\qq^r}\breve{v}_l^r\one_{[\breve{v}_l^r \le \delta c^r]}$ and using the fact that the process $Y^r_{\delta}(\cdot)$ started from $Y^r_{\delta}(0) = \frac{9B\delta^{p-\eta}}{2}$ stochastically dominates (in a pathwise fashion) the process $Y^r_{\delta}(\cdot)$ started from any value less than or equal to $\frac{9B\delta^{p-\eta}}{2}$,
\begin{align}\label{eq:initterm}
	\mathbb{P}&\left(\sup_{t \in [0,\tau_0]}Y^r_{\delta}(t) > 3D'B\delta^{p-\eta}\log(1/\delta)\right)\\
	& \le \mathbb{P}\left(Y_{\delta}^r(\cdot) \text{ crosses } 3D'B\delta^{p-\eta}\log(1/\delta) \text{ before } \frac{3B\delta^{p-\eta}}{2}\right)\nonumber\\
&\le \mathbb{P}\left(Y_{\delta}^r(\cdot) \text{ crosses } 3D'B\delta^{p-\eta}\log(1/\delta) \text{ before } \frac{3B\delta^{p-\eta}}{2} \ \Big| \ Y^r_{\delta}(0) = \frac{9B\delta^{p-\eta}}{2}\right)\nonumber\\
&\qquad \quad \qquad \qquad \quad \qquad  + \mathbb{P}\left(\frac{1}{r}\sum_{l=1}^{\qq^r}\breve{v}_l^r\one_{[\breve{v}_l^r \le \delta c^r]} > \frac{9B\delta^{p-\eta}}{2}\right)\nonumber\\
&\le \delta^{D'} + \mathbb{P}\left(\frac{1}{r}\sum_{l=1}^{\qq^r}\breve{v}_l^r\one_{[\breve{v}_l^r \le \delta c^r]} > \frac{9B\delta^{p-\eta}}{2}\right).\nonumber
\end{align}\label{eq:chiless1}
Let $0<\delta(\eta)\le\delta_*(\eta)$ be such that $T< \epsilon\delta(\eta)^{-2(p-\eta)}/4$. Choose $r(\eta)\ge\hat r(\eta)$ such that $2M(\eta)(c^r)^{-1}< \delta(\eta)$ for all $r \ge r(\eta)$. For $r\ge r(\eta)$ and $\delta \in [2M(\eta)(c^r)^{-1}, \delta(\eta)]$, by \eqref{un2},
\begin{align}
\mathbb{P}&\left(\sup_{t \in [0,T]}(Q_{\delta}^r(t) - Q_{\delta/2}^r(t)) > 6 D'B \delta^{p-1-\eta}\log (1/\delta) + \frac{c^r}{r}\right)\\
&\le \mathbb{P}\left(\sup_{t \in [0,\epsilon\delta^{-2(p-\eta)}/4]}(Q_{\delta}^r(t) - Q_{\delta/2}^r(t)) > 6 D'B \delta^{p-1-\eta}\log (1/\delta) + \frac{c^r}{r}\right)\nonumber\\
&\le \mathbb{P}\left(\sup_{t \in [0,\epsilon\delta^{-2(p-\eta)}/4]}Y^r_{\delta}(t) > 3D'B \delta^{p -\eta}\log (1/\delta)\right)\nonumber\\
&\le \mathbb{P}\left(\sup_{t \in [0,\tau_0]}Y^r_{\delta}(t) > 3D'B \delta^{p -\eta}\log (1/\delta)\right)\nonumber\\
&+ \mathbb{P}\left(\sup_{t \in (0, \epsilon\delta^{-2(p-\eta)}/4]}Y^r_{\delta}(t) > 3D'B \delta^{p -\eta}\log (1/\delta), \ \sup_{t \in [0, \tau_0]}Y^r_{\delta}(t) \le 3D'B \delta^{p -\eta}\log (1/\delta)\right).\nonumber
\end{align}
Observe that if $\sup_{t \in [0, \tau_0]}Y^r_{\delta}(t) \le 3D'B \delta^{p -\eta}\log (1/\delta)$, then $$\sup_{t \in [0, \tau_{2\mathcal{N}-1})}Y^r_{\delta}(t) \le 3D'B\delta^{p-\eta}\log(1/\delta),$$ where $\mathcal{N}$ is given in \eqref{def:scrN}.  Then, if in addition $\mathcal{N} > \lfloor \delta^{-D'/2} \rfloor + 1$ and $$\sup_{t \in (0, \epsilon\delta^{-2(p-\eta)}/4]}Y^r_{\delta}(t) > 3D'B \delta^{p -\eta}\log (1/\delta),$$
then $\tau_{2\mathcal{N}-1} \le \epsilon\delta^{-2(p-\eta)}/4$ and hence, in this case, 
\begin{equation}\label{eq:tauineq}
\sum_{j=0}^{\lfloor \delta^{-D'/2} \rfloor + 1 }(\tau_{2j+1} -\tau_{2j}) \le \tau_{2\mathcal{N}-1} \le \epsilon\delta^{-2(p-\eta)}/4.
\end{equation}
This together with \eqref{eq:chiless1} gives that for $r\ge r(\eta)$ and $\delta \in [2M(\eta)(c^r)^{-1}, \delta(\eta)]$,
\begin{align}
\mathbb{P}&\left(\sup_{t \in [0,T]}(Q_{\delta}^r(t) - Q_{\delta/2}^r(t)) > 6 D'B \delta^{p-1-\eta}\log (1/\delta) + \frac{c^r}{r}\right)
\label{eq:chiless10}\\
&\le \mathbb{P}\left(\sup_{t \in [0,\tau_0]}Y^r_{\delta}(t) > 3D'B \delta^{p -\eta}\log (1/\delta)\right)\nonumber\\
&\qquad + \mathbb{P}\left(\sum_{j=0}^{\lfloor \delta^{-D'/2} \rfloor + 1 }(\tau_{2j+1} -\tau_{2j}) \le \epsilon\delta^{-2(p-\eta)}/4, \ \mathcal{N} > \lfloor \delta^{-D'/2} \rfloor + 1 \right)\nonumber\\
&\qquad + \mathbb{P}\left(\mathcal{N} \le \lfloor \delta^{-D'/2} \rfloor + 1 \right).\nonumber\
\end{align}
By \eqref{eq:chiless10}, \eqref{eq:initterm}, the fact that $\delta^{D'/2}\epsilon\delta^{-2(p-\eta)}  \le \epsilon \delta^{2(p-\eta)}$ since $D' \ge 8p$, \eqref{Nless}, 
and \eqref{chiless}, we obtain for $r\ge r(\eta)$ and $\delta \in [2M(\eta)(c^r)^{-1}, \delta(\eta)]$,  
\begin{align*}
\mathbb{P}&\left(\sup_{t \in [0,T]}(Q_{\delta}^r(t) - Q_{\delta/2}^r(t)) > 6 D'B \delta^{p-1-\eta}\log (1/\delta) + \frac{c^r}{r}\right)\\
&\le \delta^{D'} + \mathbb{P}\left(\frac{1}{r}\sum_{l=1}^{\qq^r}\breve{v}_l^r\one_{[\breve{v}_l^r \le \delta c^r]} > \frac{9B\delta^{p-\eta}}{2}\right)\\
& \qquad + \mathbb{P}\left(\sum_{j=0}^{\lfloor \delta^{-D'/2} \rfloor + 1}(\tau_{2j+1} -\tau_{2j})\le (\lfloor \delta^{-D'/2} \rfloor + 1 )\epsilon \delta^{2(p-\eta)}/4\right)+2\delta^{D'/2}\\
&\le \mathbb{P}\left(\frac{1}{r}\sum_{l=1}^{\qq^r}\breve{v}_l^r\one_{[\breve{v}_l^r \le \delta c^r]} > \frac{9B\delta^{p-\eta}}{2}\right) + e^{-\delta^{-D'/2}/32} + 3\delta^{D'/2}\\
& \le \mathbb{P}\left(\frac{1}{r}\sum_{l=1}^{\qq^r}\breve{v}_l^r\one_{[\breve{v}_l^r \le \delta c^r]} > \frac{9B\delta^{p-\eta}}{2}\right) + 35\delta^{D'/2},
\end{align*}
where, in the last inequality, we used the fact that $xe^{-x/32}\le 32$ for all $x\ge 1$. This proves the lemma with $D_1 := 6 D'B$, $D_2 := D'/2$ and $D_3 := 9B/2$.
\end{proof}

\begin{lemma}\label{smallql}
Fix $T>0$. Recall the constant $D_2>0$ from Lemma \ref{unifsmallint}. For any $\eta \in (\eta^*, p-1)$, there are $\tilde \theta_{\eta} \in \Theta$, and positive constants $r'(\eta)$, $D'(\eta), \widetilde{D}(\eta)$, $\delta(\eta) \in (0,1), M'(\eta)>1$ such that for all $r \ge r'(\eta)$ and $\delta \in [2M'(\eta)(c^r)^{-1}, \delta(\eta)]$,
$$
\mathbb{P}\left(\sup_{t \in [0,T]}Z^r_{\delta}(t) > D'(\eta) \delta^{p-1-\eta}(1 + \log (\delta^{-1})) + \tilde\theta_{\eta}(r)\right)
\le \widetilde{D}(\eta)\left(\delta^{D_2} +\delta^{\eta - \eta^*}\right) + \tilde \theta_{\eta}(r).$$
\end{lemma}

\begin{proof}
By \eqref{eq:comp3} in Proposition \ref{comp}, for any $r \in \clr$ and any $\delta, z \ge 0$,
\begin{equation}\label{smallq1}
\mathbb{P}\left(\sup_{t \in [0,T]}Z^r_{\delta}(t) > z\right) \le \mathbb{P}\left(\sup_{t \in [0,T]}Q^r_{\delta}(t) > z - \frac{c^r}{r}\right).
\end{equation}
Take $D_1$, $D_2$, $D_3$ as in Lemma \ref{unifsmallint}. Choose and fix $\eta \in (\eta^*, p-1)$ and obtain $M(\eta)>1$ and  $r(\eta) \ge 2$, $\delta(\eta) \in (0,1)$ as in Lemma \ref{unifsmallint}. Define $M'(\eta) := M(\eta) \vee \astar$ where $\astar$ appears in Assumption \eqref{eq:initcondq1}. Denote by $\theta_{\eta}$ and $r_0(\eta)$ the map $\theta$ and constant $r_0$ obtained in Lemma \ref{fifo} with $2M(\eta)$ in place of $a$. Define $D'(\eta) := D_1\sum_{k=0}^{\infty}2^{-k(p-1-\eta)} (1 + k \log 2)$. For $\delta \in [2M'(\eta)(c^r)^{-1}, \delta(\eta)]$, let $K(\eta, \delta, r)$ be a nonnegative integer such that $2^{-K(\eta, \delta, r) -1}\delta < 2M'(\eta)(c^r)^{-1} \le 2^{-K(\eta, \delta, r)}\delta$. This, along with \eqref{crup}, implies that for $r\ge r(\eta)$ and $\delta \in [2M'(\eta)(c^r)^{-1}, \delta(\eta)]$,
\begin{equation}\label{Kbd}
K(\eta, \delta, r) \le \log_2\left(\frac{\delta c^r}{2M'(\eta)}\right) \le \log_2\left(\frac{ c^r}{M'(\eta)}\right) \le C'(\eta) \log r,
\end{equation}
where $C'(\eta) =2\log_2(e)/(p-\eta/2)$ depends only on $\eta$ (and $p$).
Observe that for any $r\in\clr$ and $\delta>0$,
\begin{multline}\label{smallq2}
\mathbb{P}\left(\sup_{t \in [0,T]}Q^r_{\delta}(t) > D'(\eta) \delta^{p-1-\eta}[1+\log (1/\delta)] + C'(\eta)\frac{c^r\log r}{r} + \theta_{\eta}(r)\right)\\
\le \mathbb{P}\left(\sup_{t \in [0,T]}\left(Q^r_{\delta}(t) - Q^r_{2M'(\eta)(c^r)^{-1}}(t)\right) > D'(\eta) \delta^{p-1-\eta}[1+\log (1/\delta)]+ C'(\eta)\frac{c^r\log r}{r}\right)\\
+ \mathbb{P}\left(\sup_{t \in [0,T]}Q^r_{2M'(\eta)(c^r)^{-1}}(t) >  \theta_{\eta}(r)\right).
\end{multline}
By Lemma \ref{unifsmallint}, for every $r \ge r(\eta)$ and $\delta \in [2M'(\eta)(c^r)^{-1}, \delta(\eta)]$,
\begin{multline}\label{smallq3}
\mathbb{P}\left(\sup_{t \in [0,T]}\left(Q^r_{2^{-k}\delta}(t) - Q^r_{2^{-k-1}\delta}(t)\right) > D_1 (2^{-k}\delta)^{p-1-\eta}\log (2^k/\delta) + \frac{c^r}{r}\right)\\
 \le 35(2^{-k}\delta)^{D_2} +\mathbb{P}\left(\frac{1}{r}\sum_{l=1}^{\qq^r}\breve{v}_l^r\one_{[\breve{v}_l^r \le 2^{-k}\delta c^r]} > D_3(2^{-k}\delta)^{p-\eta}\right) \ \text{ for all } 0 \le k \le K(\eta, \delta, r).
\end{multline}
By Assumption \eqref{eq:initcondq1} and \eqref{Kbd} (and since $M'(\eta)\ge \astar$), there exist $C'', r''>0$ such that for all $r \ge r''$, $\delta \in [2M'(\eta)(c^r)^{-1}, \delta(\eta)]$, and $0 \le k \le K(\eta,\delta,r)$,
\begin{equation}\label{initialsm}
\mathbb{E}\left(\frac{1}{r}\sum_{l=1}^{\qq^r}\breve{v}_l^r\one_{[\breve{v}_l^r \le 2^{-k}\delta c^r]}\right) \le C''(2^{-k}\delta)^{p-\eta^*}.
\end{equation}
Let $r'(\eta) := \max\{r(\eta), r_0(\eta), r''\}$. For $r\ge r'(\eta)$ and $\delta \in [2M'(\eta)(c^r)^{-1}, \delta(\eta)]$,
since $2^{-K(\eta, \delta, r) -1}\delta < 2M'(\eta)(c^r)^{-1}$, by Lemma \ref{qlxy}, for any $t \ge 0$,
\begin{align*}
\left(Q^r_{\delta}(t) - Q^r_{2M'(\eta)(c^r)^{-1}}(t)\right) &= \left(Q^r_{\delta}(t) - Q^r_{2^{-K(\eta, \delta, r) -1}\delta}(t)\right)\\
&\qquad + \left(Q^r_{2^{-K(\eta, \delta, r) -1}\delta}(t) - Q^r_{2M'(\eta)(c^r)^{-1}}(t)\right)\\
&\le \left(Q^r_{\delta}(t) - Q^r_{2^{-K(\eta, \delta, r) -1}\delta}(t)\right).
\end{align*}
Using this observation, along with \eqref{Kbd}, \eqref{smallq3}, \eqref{initialsm}, Markov's inequality and the union bound, for any $r \ge r'(\eta)$ and $\delta \in [2M'(\eta)(c^r)^{-1}, \delta(\eta)]$,
\begin{align}\label{smallq4}
&\mathbb{P}\left(\sup_{t \in [0,T]}\left(Q^r_{\delta}(t) - Q^r_{2M'(\eta)(c^r)^{-1}}(t)\right) > D'(\eta) \delta^{p-1-\eta}(1 + \log (1/\delta)) + C'(\eta)\frac{c^r\log r}{r} + \frac{c^r}{r}\right)\\
&\le \mathbb{P}\left(\sup_{t \in [0,T]}\left(Q^r_{\delta}(t) - Q^r_{2^{-K(\eta, \delta, r) -1}\delta}(t)\right) > D'(\eta) \delta^{p-1-\eta}(1 + \log (1/\delta)) + C'(\eta)\frac{c^r\log r}{r} + \frac{c^r}{r}\right)\nonumber\\
&\le \sum_{k=0}^{K(\eta, \delta, r)}\mathbb{P}\left(\sup_{t \in [0,T]}\left(Q^r_{2^{-k}\delta}(t) - Q^r_{2^{-k-1}\delta}(t)\right) > D_1 (2^{-k}\delta)^{p-1-\eta}\log (2^{k}/\delta) + c^r/r\right)\nonumber\\
&\le \sum_{k=0}^{K(\eta, \delta, r)}35(2^{-k}\delta)^{D_2} + \sum_{k=0}^{K(\eta, \delta, r)}\mathbb{P}\left(\frac{1}{r}\sum_{l=1}^{\qq^r}\breve{v}_l^r\one_{[\breve{v}_l^r \le 2^{-k}\delta c^r]} > D_3(2^{-k}\delta)^{p-\eta}\right)\nonumber\\
&\le \sum_{k=0}^{K(\eta, \delta, r)}35(2^{-k}\delta)^{D_2} + \sum_{k=0}^{K(\eta, \delta, r)}(D_3(2^{-k}\delta)^{p-\eta})^{-1}C''(2^{-k}\delta)^{p-\eta^*}\nonumber\\
&\le 35\delta^{D_2}\sum_{k=0}^{\infty}2^{-D_2k} + C''(D_3)^{-1}\delta^{\eta - \eta^*}\sum_{k=0}^{\infty}2^{-(\eta - \eta^*)k} \le \widetilde{D}(\eta)\left(\delta^{D_2} +\delta^{\eta - \eta^*}\right),\nonumber
\end{align}
where $\widetilde{D}(\eta) := 35\sum_{k=0}^{\infty}2^{-D_2k} + C''(D_3)^{-1}\sum_{k=0}^{\infty}2^{-(\eta - \eta^*)k} \in (0, \infty)$. Finally, by Lemma \ref{fifo}, for any $r \ge r'(\eta)$,
\begin{equation}\label{smallq5}
 \mathbb{P}\left(\sup_{t \in [0,T]}Q^r_{2M(\eta)(c^r)^{-1}}(t) >  \theta_{\eta}(r)\right) \le \theta_{\eta}(r).
\end{equation}
Taking $\tilde \theta_{\eta}(r) = C'(\eta)\frac{c^r\log r}{r} + \theta_{\eta}(r) + \frac{2c^r}{r}$, the lemma now follows from \eqref{smallq1}, \eqref{smallq2}, \eqref{smallq4} and \eqref{smallq5}.
\end{proof}

\begin{remark}\label{uiinit}
By small modifications of some of the estimates in Lemmas \ref{fifo}, \ref{unifsmallint} and \ref{smallql}, it can in fact be shown that for {\color{black}a sequence of systems such that each system has no jobs in system at time zero,} for any $T>0$ and any $\eta \in (0, p-1)$, there exist positive constants $C, C', C'', r_0$ such that for any $r \ge r_0$, $a \in [(c^r)^{-1},1]$ and $z \ge 0$,
$$
\mathbb{P}\left(\sup_{t \in [0,T]}W_a^r(t) > C a^{p-\eta}z\right) \le \mathbb{P}\left(\sup_{t \in [0,T]}Z_a^r(t) > C a^{p-\eta-1}z\right) \le C'e^{-C''z},
$$
where we have used the elementary bound $W_a^r(t)\le aZ_a^r(t)$ for $t\ge 0$ to obtain the first inequality.
By integrating over $z$, this immediately implies that{\color{black}, in this case,} Assumption \eqref{eq:initcondq1} holds with $W_a^r(0)$ replaced by $W_a^r(t)$ for any fixed $t>0$.
\end{remark}

The next two lemmas concern the limiting random field $\{W_a(\cdot), a\in[0,\infty]\}$.  In preparation for using these two results both in the proof of Theorem \ref{workfn} and in the proof of Theorem \ref{collapse} (which concerns asymptotic state space collapse as $p\to\infty$), the dependence on $p$ of various objects is made explicit in the statements of these lemmas.  In this regard, we remind the reader that $p>1$ is presently fixed and therefore, the asymptotic conditions of Section \ref{sec:sc} need not hold for the results in these lemmas to be true.

Recall $\sigmap = \sqrt{\lambda \operatorname{Var}(v^{(p)}) + \lambda\sigma_A^2}$, where $v^{(p)}$ denotes the job processing time distribution with highlighted dependence on $p$. Also recall $\eta^* = \eta^*(p)$ in Assumption \eqref{eq:initcondq1}.

\begin{lemma}\label{BMsup}
Let $T>0$.  Set $\mop:=\max\{2,\lambda,4\kappa^2/\lambda^2,e^{\lambda/(\sigmap)^2},T\}$, $\aop=\mop^{-1/2p}$, and
$\hop:=8 p(\sigmap)^2/\lambda$.
Then $\aop\in(0,1)$ and for all $a\in(0,\aop)$, $\eta \in (\eta^*(p), p-1)$ and $H\ge \hop$, we have
\begin{multline*}
\mathbb{P}\left(\sup_{t \in [0,T]} W_a^{(p)}(t) > a^{p-\eta} + H a^p \log(1/a)\right)\\
\le \cop a^{\eta - \eta^*(p)} + e^{-\lambda/(2\left(\sigmap\right)^2a^{\eta})} + C(\lambda, \sigmap)a^{2p},
\end{multline*}
where $\cop :=2\sup_{a >0}a^{-(p-\eta^*(p))}\mathbb{E}\left(\xi^{(p)}(a)\right) < \infty$ due to \eqref{eq:Cxifinite} and $C(\lambda, \sigmap) := 2e^{\lambda/(\sigmap)^2} + \frac{16\left(\sigmap\right)^2}{\lambda}$.
\end{lemma}

\begin{proof} Since $\mop>1$, we have $\aop\in(0,1)$.  
Fix $a\in(0,\aop)$, $\eta \in (\eta^*(p), p-1)$ and $H\ge \hop$.
To ease the notation in this proof, we suppress the dependence on $p$
and write $\mo=\mop$, $\ao=\aop$, $\ho=\hop$, $\sigma=\sigmap$, $\eta^*=\eta^*(p)$ and $\co=\cop$. 
Observe that since $H\log(1/a)>H\log(1/\ao)=H\log(\mo)/2p\ge \lambda H/(2p\sigma^2)\ge 4>1$, we have $Ha^p\log(1/a)>a^p$.
Define the stopping times: $\tau^*_0 := \inf\{t \ge 0:  W_a(t) = 0\}$, and for $k \in\bbZp$,
\begin{align*}
\tau^*_{2k+1} &:= \inf\{t \ge \tau^*_{2k}: W_a(t) = a^p\},\\
\tau^*_{2k+2} &:= \inf\{t \ge \tau^*_{2k+1}: W_a(t)=0 \text{ or } W_a(t) = Ha^p\log(1/a)\}.
\end{align*}
Define $\mathcal{N}^* := \inf\{k\in\NN: W_a(\tau^*_{2k}) = Ha^p\log(1/a)\}$. Since $\kappa<\lambda/(2a^p)$, we have $\kappa - \frac{\lambda}{a^p} < - \frac{\lambda}{2a^p}$. Thus, 
by \eqref{eq:monoskor}, the process $\Gamma\left[ \overline{X}_a\right](\cdot)$ with $\overline{X}_a(t) := \xi(a) + \sigma B(t) - \frac{\lambda t}{2a^p}, \ t \ge 0$,
dominates the process $W_a(\cdot)$ pointwise.  Thus, using the fact that $t \mapsto e^{\lambda \overline{X}_a(t)/(\sigma^2a^p)}$ is a martingale (with respect to the filtration $\{{\mathcal G}_t\}_{t\ge 0}$ given by
${\mathcal G}_t=\sigma\left(\overline{X}_a(0), (B(s), 0\le s\le t)\right)$ for $t\ge 0$), by the optional stopping theorem and the strong Markov property,
\begin{align}\label{startmax}
\mathbb{P}&\left(\sup_{[0, \tau^*_0]} W_a(t) > a^{p-\eta}\right)\\
& \le \mathbb{P}\left(\xi(a) > a^{p-\eta}/2\right) + \mathbb{P}\left(\left\{\sup_{[0, \tau^*_0]} W_a(t) > a^{p-\eta}\right\} \cap \left\{\ \xi(a) \le a^{p-\eta}/2\right\}\right)\nonumber\\
&\le 2a^{\eta - p}\mathbb{E}\left(\xi(a)\right) +
\mathbb{P}\left(\overline{X}_a(t+\cdot) \mbox{ crosses } a^{p-\eta} \mbox{ before } 0 \mid
\overline{X}_a(t)= a^{p-\eta}/2\right)\nonumber\\
&\le 2a^{\eta - p}\mathbb{E}\left(\xi(a)\right) + \frac{e^{\lambda/(2\sigma^2a^{\eta})} - 1}{e^{\lambda/(\sigma^2a^{\eta})} - 1} \le \co a^{\eta - \eta^*} + e^{-\lambda/(2\sigma^2a^{\eta})}.\nonumber
\end{align}
As previously noted, $H\log(1/a)>1$.
This together with an argument using the optional stopping theorem in manner similarly to the above gives
\begin{multline*}
\mathbb{P}\left(W_a(\tau^*_2) = Ha^p\log(1/a)\right) \le 
\mathbb{P}\left(\overline{X}_a(t+\cdot) \mbox{ crosses } Ha^p\log(1/a)\mbox{ before } 0\mid
\overline{X}_a(t)= a^p\right)\\
= \frac{e^{\lambda/\sigma^2} - 1}{e^{\lambda H\log(1/a)/\sigma^2} - 1} < e^{\lambda(1-H\log(1/a))/\sigma^2}=e^{\lambda/\sigma^2}a^{\lambda H/\sigma^2}.
\end{multline*}
{\color{black}Using a union bound and the strong Markov property,} this implies
\begin{multline}\label{BMsup1}
\mathbb{P}\left(\mathcal{N}^* < \lfloor a^{-3H\lambda/(4\sigma^2)} \rfloor + 2\right) \le \left(a^{-3H\lambda/(4\sigma^2)} +1\right)e^{\lambda/\sigma^2}a^{\lambda H/\sigma^2}\\
\le\left(1 +a^{3H\lambda/(4\sigma^2)}\right)e^{\lambda/\sigma^2}a^{\lambda H/4\sigma^2}\le
2 e^{\lambda/\sigma^2}a^{H\lambda/(4\sigma^2)} \le 2e^{\lambda/\sigma^2}a^{2p}.
\end{multline}
Again, by our choice of $a$, $H$, and $\ao$, $a^{-\left(\frac{H\lambda}{2\sigma^2} - 2p\right)}\ge a^{-2p}>\ao^{-2p} \ge T$ and hence
\begin{align}\label{BMsup2}
\mathbb{P}&\left(\sup_{t \in [0,T]} W_a(t) > a^{p-\eta} + H a^p \log(1/a)\right)\\
& \le \mathbb{P}\left(\sup_{[0, \tau^*_0 \wedge T]} W_a(t) > a^{p-\eta}\right) + \mathbb{P}\left(\sup_{t \in [\tau^*_0 \wedge T,a^{-\left(\frac{H\lambda}{2\sigma^2} - 2p\right)}]} W_a(t) > H a^p \log(1/a)\right)\nonumber\\
&\le \mathbb{P}\left(\sup_{[0, \tau^*_0 \wedge T]} W_a(t) > a^{p-\eta}\right) + \mathbb{P}\left(\sum_{k=1}^{\mathcal{N}^*-1}\left(\tau^*_{2k} - \tau^*_{2k-1}\right) < a^{-\left(\frac{H\lambda}{2\sigma^2} - 2p\right)}\right)\nonumber\\
&\le \mathbb{P}\left(\sup_{[0, \tau^*_0 \wedge T]} W_a(t) > a^{p-\eta}\right) + \mathbb{P}\left(\mathcal{N}^* < \lfloor a^{-3H\lambda/(4\sigma^2)} \rfloor + 2\right)\nonumber\\
&\qquad +  \mathbb{P}\left(\sum_{k=1}^{\lfloor a^{-3H\lambda/(4\sigma^2)}\rfloor +1}\left(\tau^*_{2k} - \tau^*_{2k-1}\right) < a^{-\left(\frac{H\lambda}{2\sigma^2} - 2p\right)}, \mathcal{N}^* \ge \lfloor a^{-3H\lambda/(4\sigma^2)} \rfloor + 2\right).\nonumber
\end{align}
Denote by $\sigma^x$ the hitting time of level $x \le 0$ by the process $\{\sigma B(t) - \la t/2a^p, \ t\ge 0\}$, and let $\{\sigma^x_k\}_{k\in\NN}$ be independent and identically distributed copies of $\sigma^x$. For each $x<0$, by the explicit form of the moment generating function of $\sigma^x$ ({\color{black}see Exercise 5.10 in Chapter 3.5.C of \cite{karatzas}}),
$$
\mathbb{E}\left(\sigma^x\right) = \frac{2a^p|x|}{\lambda}\qquad\text{and}\qquad
\operatorname{Var}\left(\sigma^x\right) = \frac{8a^{3p}\sigma^2|x|}{\lambda^3}.
$$
Thus, again using the strong Markov property, $a^{-H\lambda/(4\sigma^2)}\ge a^{-2p}>\ao^{-2p} \ge \lambda$, and Chebyshev's inequality,
\begin{align}\label{BMsup3}
\mathbb{P}&\left(\sum_{k=1}^{\lfloor a^{-3H\lambda/(4\sigma^2)}\rfloor +1}\left(\tau^*_{2k} - \tau^*_{2k-1}\right) < a^{-\left(\frac{H\lambda}{2\sigma^2} - 2p\right)}, \mathcal{N}^* \ge \lfloor a^{-3H\lambda/(4\sigma^2)} \rfloor + 2\right) \\
&\le \mathbb{P}\left(\sum_{k=1}^{\lfloor a^{-3H\lambda/(4\sigma^2)}\rfloor +1}\sigma^{-a^p}_k < a^{-\left(\frac{H\lambda}{2\sigma^2} - 2p\right)}\right)\nonumber\\
&\le \mathbb{P}\left(\sum_{k=1}^{\lfloor a^{-3H\lambda/(4\sigma^2)}\rfloor +1}\left(\sigma^{-a^p}_k - \frac{2a^{2p}}{\lambda}\right) < a^{-\left(\frac{H\lambda}{2\sigma^2} - 2p\right)} - \frac{2a^{2p}a^{-3H\lambda/(4\sigma^2)}}{\lambda}\right)\nonumber\\
&\le \mathbb{P}\left(\sum_{k=1}^{\lfloor a^{-3H\lambda/(4\sigma^2)}\rfloor +1}\left(\sigma^{-a^p}_k - \frac{2a^{2p}}{\lambda}\right) < - \frac{a^{2p}a^{-3H\lambda/(4\sigma^2)}}{\lambda}\right)\nonumber\\
&\le \frac{8\sigma^2a^{4p}\left(\lfloor a^{-3H\lambda/(4\sigma^2)}\rfloor +1\right)}{\lambda a^{4p-3H\lambda/(2\sigma^2)}} \le \frac{16\sigma^2a^{3H\lambda/(4\sigma^2)}}{\lambda}.\nonumber
\end{align}
Finally, using \eqref{startmax}, \eqref{BMsup1} and \eqref{BMsup3} in \eqref{BMsup2}, we obtain the lemma.
\end{proof}

\begin{lemma}\label{flowsup}. Let $T>0$ and let $\aop$, $\hop$, $\cop$ and $C(\lambda,\sigmap)$ be as in Lemma \ref{BMsup}.
Then for all $\delta \in(0, \aop)$ and $\eta \in (\eta^*(p), p-1)$,
\begin{multline*}
\mathbb{P}\left(\sup_{t \in [0,T]}\Big( \int_0^{\delta}x^{-2}W_x^{(p)}(t)dx + \delta^{-1}W_{\delta}^{(p)}(t) \Big) > H(p, \eta) \delta^{p - \eta-1}(1+\log(1/\delta))\right)\\
\le \tilde C(\eta, \eta^*(p), \lambda, \sigmap)\delta^{\eta - \eta^*(p)} + 3C(\lambda, \sigmap)\delta^{2p},
\end{multline*}
where
\begin{align*}
H(p, \eta) &:= \hop\left[1 + \frac{2^{p-1}}{(2^{p-1}-1)} + \frac{(\log 2) 2^{p-1}}{(2^{p-1}-1)^2}\right] + 1 + \sum_{k=1}^{\infty}2^{-(k-1)(p-\eta -1)},\\
\tilde C(\eta, \eta^*(p), \lambda, \sigmap) &:= \left(\cop + \frac{2(\sigmap)^{2}}{\lambda}\left(\sup_{x\in\RRp}xe^{-x}\right)\right)\left(1 + \sum_{k=1}^{\infty}2^{-(k-1)(\eta - \eta^*(p))}\right).
\end{align*}
In particular, $\sup_{t \in [0,T]}\int_0^{\infty}x^{-2}W_x^{(p)}(t)dx < \infty$ almost surely.
\end{lemma}

\begin{proof} As in the proof of Lemma \ref{BMsup}, we suppress the dependence on $p$ in this proof to ease the notation in what follows.  Fix $\delta\in(0,\ao)$ and $\eta \in (\eta^*, p-1)$ and set $H=\ho$.
As for any $x_1< x_2$, $X_{x_2}(t) - X_{x_1}(t)$ is nonnegative and nondecreasing in $t$, using the monotonicity property in \eqref{eq:monoskor}, we obtain for $t \ge 0$,
\begin{multline}\label{eq:intxwst}
 \int_0^{\delta}x^{-2}W_x(t)dx = \sum_{k=1}^{\infty}\int_{\delta 2^{-k}}^{\delta 2^{-(k-1)}}x^{-2}W_x(t)dx\\ \le \sum_{k=1}^{\infty}W_{\delta 2^{-(k-1)}}(t)\int_{\delta 2^{-k}}^{\delta 2^{-(k-1)}}x^{-2}dx = \sum_{k=1}^{\infty}\frac{W_{\delta 2^{-(k-1)}}(t)}{\delta 2^{-(k-1)}}.
\end{multline}
By Lemma \ref{BMsup}, for any $k \in\NN$,
\begin{multline}\label{kge1}
\mathbb{P}\left(\sup_{t \in [0,T]} \frac{W_{\delta 2^{-(k-1)}}(t)}{\delta 2^{-(k-1)}} > (\delta 2^{-(k-1)})^{p-\eta-1} + \ho (\delta 2^{-(k-1)})^{p-1} \log \frac{1}{\delta 2^{-(k-1)}}\right)\\
\le \co(\delta 2^{-(k-1)})^{\eta - \eta^*} +e^{-\lambda/(2\sigma^2(\delta 2^{-(k-1)})^{\eta})} +  C(\lambda, \sigma) (\delta 2^{-(k-1)})^{2p}.
\end{multline}
Also,
\begin{multline*}
\ho\sum_{k=1}^{\infty}(\delta 2^{-(k-1)})^{p-1} \log \frac{1}{\delta 2^{-(k-1)}}\\
 = \ho\left[\frac{ 2^{p-1}}{2^{p-1}-1}\delta^{p-1}\log(1/\delta) + \frac{(\log 2) 2^{p-1}}{(2^{p-1}-1)^2}\delta^{p-1}\right]
\le H_1(p)\delta^{p-1}(1+\log(1/\delta)),
\end{multline*}
where
$$
H_1(p) := \ho\left[\frac{2^{p-1}}{2^{p-1}-1} + \frac{(\log 2) 2^{p-1}}{(2^{p-1}-1)^2}\right].
$$
Moreover,
$$
\sum_{k=1}^{\infty}(2^{-(k-1)}\delta)^{p-\eta -1} = H_2(p,\eta) \delta^{p-\eta - 1},
$$
where $H_2(p,\eta) := \sum_{k=1}^{\infty}2^{-(k-1)(p-\eta -1)}$. Using these observations, \eqref{eq:intxwst}, a union bound, and \eqref{kge1},
\begin{multline}\label{flowsup1}
\mathbb{P}\left(\sup_{t \in [0,T]}\int_0^{\delta}x^{-2}W_x(t)dx  > H_1(p)\delta^{p-1}(1+\log(1/\delta)) + H_2(p,\eta) \delta^{p-\eta - 1}\right)\\
\le \mathbb{P}\left(\sup_{t \in [0,T]}\sum_{k=1}^{\infty}\frac{W_{\delta 2^{-(k-1)}}(t)}{\delta 2^{-(k-1)}}  > H_1(p) \delta^{p-1}(1+\log(1/\delta)) + H_2(p,\eta) \delta^{p-\eta - 1}\right)\\
\le \sum_{k=1}^{\infty}\mathbb{P}\left(\sup_{t \in [0,T]} \frac{W_{\delta 2^{-(k-1)}}(t)}{\delta 2^{-(k-1)}} > (\delta 2^{-(k-1)})^{p-\eta-1} + \ho [\delta 2^{-(k-1)}]^{p-1} \log \frac{1}{\delta 2^{-(k-1)}}\right)\\
\le \hat C(\eta, \eta^*, \lambda, \sigma)\delta^{\eta - \eta^*} + C(\lambda, \sigma)\sum_{k=1}^{\infty}(\delta 2^{-(k-1)})^{2p} \le \hat C(\eta, \eta^*, \lambda, \sigma)\delta^{\eta - \eta^*} + 2C(\lambda, \sigma)\delta^{2p},
\end{multline}
where $\hat C(\eta, \eta^*, \lambda, \sigma) := \left(\co + \frac{2\sigma^{2}}{\lambda}\left(\sup_{x\in\RRp}xe^{-x}\right)\right)\left(\sum_{k=1}^{\infty}2^{-(k-1)(\eta - \eta^*)}\right)$.
By taking $k=1$ in \eqref{kge1}, we obtain
\begin{multline}\label{missing}
\mathbb{P}\left(\sup_{t \in [0,T]} \delta^{-1}W_{\delta}(t) > \delta^{p-\eta- 1}+\ho \delta^{p-1}\log\left(\frac{1}{\delta}\right)\right)\\ \le \co\delta^{\eta - \eta^*} + e^{-\lambda/(2\sigma^2\delta^{\eta})} +  C(\lambda, \sigma) \delta^{2p}
\le \left(\co + \frac{2\sigma^{2}}{\lambda}\left(\sup_{x\in\RRp}xe^{-x}\right)\right)\delta^{\eta - \eta^*} + C(\lambda, \sigma)\delta^{2p}.
\end{multline}
The first assertion of the lemma follows from \eqref{flowsup1} and \eqref{missing} upon noting that $\tilde C(\eta, \eta^*, \lambda, \sigma) = \hat C(\eta, \eta^*, \lambda, \sigma) + \left(\co + \frac{2\sigma^{2}}{\lambda}\left(\sup_{x\in\RRp}xe^{-x}\right)\right)$ and\\ $H(p, \eta) = H_1(p) + H_2(p,\eta) + 1 + \frac{8 \sigma^2 p}{\lambda}$. 

Now, we check the last assertion. If $\mathbb{P}\left(\sup_{t \in [0,T]}\int_0^{1}x^{-2}W_x(t)dx = \infty \right) > 0$, by the finiteness of $\sup_{t \in [0,T]}\int_{\delta}^1x^{-2}W_x(t)dx$ for all $\delta\in(0,1]$, there exists $\epsilon>0$ such that
$$
\mathbb{P}\left(\sup_{t \in [0,T]}\int_0^{\delta}x^{-2}W_x(t)dx = \infty \right) \ge \epsilon
$$
for all $\delta>0$, which contradicts the first assertion of the lemma. Thus,
\begin{equation}\label{fini1}
\sup_{t \in [0,T]}\int_0^{1}x^{-2}W_x(t)dx < \infty \text{ almost surely}.
\end{equation}
Moreover, by the monotonicity property noted previously
\begin{equation}\label{fini2}
\sup_{t \in [0,T]}\int_1^{\infty}x^{-2}W_x(t)dx \le \sup_{t \in [0,T]}W_{\infty}(t) < \infty \text{ almost surely}.
\end{equation}
The last assertion of the lemma follows from \eqref{fini1} and \eqref{fini2}.
\end{proof}

\begin{proof}[Proof of Theorem \ref{workfn}] Fix a $C^1$ function $f: [0, \infty) \rightarrow \mathbb{R}$ such that $\lim_{x \rightarrow \infty} \frac{f(x)}{x}$ exists and $\int_1^{\infty} \frac{|f'(x)|}{x^{\alpha^* +1}} < \infty$. Set $g(x)=f(x)/x$ for $x>0$ and define $g(\infty)=\lim_{x \rightarrow \infty}g(x)$.
By Lemma \ref{Minfty}, for each $\delta>0$, as $r \rightarrow \infty$,
\begin{equation}\label{delz1}
\int_{\delta}^{\infty}f(x) \tZ^r(\cdot)(dx) \xrightarrow{d} - \int_{\delta}^{\infty}g'(x)W_x(\cdot)dx + g(\infty)W_{\infty}(\cdot) - g(\delta)W_{\delta}(\cdot).
\end{equation}
Moreover, for all $r\in\clr$,
$\int_{0}^{\infty}f(x) \tZ^r(t)(dx)$
is finite for all $t \in [0,T]$ almost surely. Fix $\eta \in (\eta^*, p-1)$. Define $C_f := \sup_{z \in [0,1]}|f(z)|$ and let $D_2$, $D'(\eta)$ and $\tilde D(\eta)$ as in Lemma \ref{smallql}.  For each $\delta>0$, let
$b(\delta) :=\max\{2C_f  D'(\eta) \delta^{p-1-\eta}\left(1+ \log (1/\delta)\right), 2\widetilde{D}(\eta)\left(\delta^{D_2} +\delta^{\eta - \eta^*}\right)\}$. Then, by Lemma \ref{smallql}, for any $0<\delta \le \delta(\eta)$,
\begin{equation}\label{delz2}
\limsup_{r \rightarrow \infty}\mathbb{P}\left(\sup_{[0,T]}\left|\int_{0}^{\delta}f(x) \tZ^r(t)(dx)\right| > b(\delta)\right) < b(\delta).
\end{equation}  
As $f$ is $C^1$ on $[0, \infty)$, $g(x) \le C_fx^{-1}$ for all $x \in (0,1]$, and $g'(x) = \frac{-f(x)}{x^2} + \frac{f'(x)}{x}, x>0,$ satisfies $|g'(x)| \le C'_fx^{-2}$ for all $x \in (0,1]$ for some constant $C'_f>0$. Thus, by Lemma \ref{flowsup}, 
$
- \int_{0}^{\infty}g'(x)W_x(t)dx + g(\infty)W_{\infty}(t)
$
is well defined and finite for all $t \in [0,T]$ almost surely, $g(\delta)W_{\delta}(\cdot)\to 0$ in probability {\color{black}uniformly over compact time intervals} as $\delta \to 0$, and
\begin{equation}\label{delz3}
- \int_{\delta}^{\infty}g'(x)W_x(\cdot)dx + g(\infty)W_{\infty}(\cdot) - g(\delta)W_{\delta}(\cdot) \xrightarrow{d} - \int_{0}^{\infty}g'(x)W_x(\cdot)dx + g(\infty)W_{\infty}(\cdot)
\end{equation}
as $\delta \rightarrow 0$, {\color{black}in $\mathcal{D}([0, \infty):\RR)$}. By Lemma \ref{flowsup} and the monotonicity of $\int_0^{\delta}x^{-2}W_x(t)dx$ in $\delta$,
$$
 \sup_{t \in [0,T]}\int_0^{\delta}x^{-2}W_x(t)dx \to 0 \mbox{ as } \delta \to 0, \ \text{ almost surely.}
$$
This implies that, almost surely, $\int_{\delta}^{\infty}g'(x)W_x(\cdot)dx$ converges to $\int_{0}^{\infty}g'(x)W_x(\cdot)dx$ as $\delta\to 0$ uniformly in $t \in [0,T]$. Moreover, for any $\delta>0$, by Lemma \ref{Minfty}, $\int_{\delta}^{\infty}g'(x)W_x(\cdot)dx$ lies in $\mathcal{C}([0,T] : \mathbb{R})$. Thus, due to uniform convergence, $\int_{0}^{\infty}g'(x)W_x(\cdot)dx$ lies in $\mathcal{C}([0,T] : \mathbb{R})$ as well.
The theorem follows from this observation, \eqref{delz1}, \eqref{delz2}, \eqref{delz3} and Lemma \ref{seqconv}.
\end{proof}

\begin{remark}\label{rem:aeqze}
	Along the lines of the proof of Theorem \ref{workfn} one can analyze the convergence of $\int_{\delta}^{b_1}f(x) \tZ^r(\cdot)(dx)$ as $\delta \to 0$, where $b_1 \in (0,\infty)$, and conclude that $a_1$ in Theorem \ref{convplus} can be taken to be $0$.
\end{remark}

\subsection{Proofs of Theorems \ref{meascon} and \ref{tailtZ}}\label{sec:proofs}

\begin{proof}[Proof of Theorem \ref{meascon}]  We will use Theorem 2.1 in \cite{roelly1986}. This theorem says the following.  Let $\{f_n\}_{n\ge 1}$ be a countable collection of real-valued continuous functions with compact support on $\RRp$ which is dense in $\clc_0(\RRp)$ [the space of continuous functions on $\RRp$ vanishing at $\infty$ equipped with the uniform metric].
Let $f_0=1$. Suppose that
\begin{equation}\label{measdismain}
(\{Z_f^r(\cdot) =\lan f, \tZ^r(\cdot)\ran, \; r \in\clr\}
\text{ is tight in } \cld([0,T]:\RR) \text{ for every } f \in \{f_n\}_{n\in \NN_0}.
\end{equation}
Then
$\{\tZ^r(\cdot), \; r\in\clr\}$
is tight in $\cld([0,T]:\clm_F)$.

By Theorem \ref{workfn},
\begin{equation}\label{measdis1}
\int_0^{\infty} f_0(x)\tZ^r(\cdot)(dx) \xrightarrow{d} \int_0^{\infty} f_0(x)\tZ(\cdot)(dx) \qquad \mbox{ as } r\to\infty.
\end{equation}
Let
\begin{align*}
\mathcal{C} &:= \left\lbrace h = \sum_{j=1}^J c_j \one_{(a_j, b_j]} : J \in \mathbb{N}, \ 0 \le a_1 < b_1 \le a_2 < b_2 \dots \le a_J < b_J < \infty,\right.\\
&\hspace{10cm} \left. c_j \in \mathbb{R} \text{ for all } 1 \le j \le J\right\rbrace.
\end{align*}
By Theorem \ref{convplus} and Remark \ref{rem:takez}, for any $h \in \mathcal{C}$,
\begin{equation}\label{measdis2}
\int_0^{\infty} h(x)\tZ^r(\cdot)(dx) \xrightarrow{d} \int_0^{\infty} h(x)\tZ(\cdot)(dx) \quad\mbox{ as } \quad r\to\infty.
\end{equation}
Now, fix $T>0$ and take any compactly supported real-valued continuous function $f$ and let $\{h_k\}_{k \in\NN}$ be a sequence in $\mathcal{C}$ such that $\| h_k - f\|_{\infty} \le k^{-1}$ for $k \in\NN$. Thus, for any $k \in\NN$,
$$
\sup_{t\in[0,T]}\left |\int_0^{\infty} h_k(x) \tZ^r(t)(dx) - \int_0^{\infty} f(x) \tZ^r(t)(dx)\right| \le k^{-1}\sup_{t\in[0,T]}\int_0^{\infty} {\color{black}\tZ^r(t)}(dx).
$$
By Theorem \ref{workfn}, and the continuous mapping theorem,
$$
\sup_{t\in[0,T]}\int_0^{\infty} {\color{black}\tZ^r(t)}(dx) \xrightarrow{d} \sup_{t\in[0,T]}Q(t) \qquad\mbox{ as } r\to\infty,
$$
where we recall {\color{black}that} $Q(\cdot) = \int_0^{\infty}x^{-2}W_x(\cdot)dx \in \mathcal{C}([0, \infty):\mathbb{R}_+)$ a.s.
Therefore, by the Portmanteau Theorem,
\begin{multline}\label{measdis3}
\lim_{k \rightarrow \infty}\limsup_{r\to \infty} \mathbb{P}\left(\sup_{t\in[0,T]}\left |\int_0^{\infty} h_k(x) \tZ^r(t)(dx) - \int_0^{\infty} f(x) \tZ^r(t)(dx)\right| > k^{-1/2}\right)\\
\le \lim_{k \rightarrow \infty}\limsup_{r\to \infty} \mathbb{P}\left(\sup_{t\in[0,T]}\int_0^{\infty} \tZ^r(t)(dx) \ge k^{1/2}\right)
\le \lim_{k \rightarrow \infty}\mathbb{P}\left(\sup_{t\in[0,T]}Q(t) \ge k^{1/2}\right) = 0.
\end{multline}
Finally, we have that almost surely,
\begin{equation}\label{measdis4}
\lim_{k \rightarrow \infty}\sup_{t\in[0,T]}\left |\int_0^{\infty} h_k(x) \tZ(t)(dx) - \int_0^{\infty} f(x) \tZ(t)(dx)\right| \le \lim_{k \rightarrow \infty} k^{-1}\sup_{t\in[0,T]}Q(t) = 0.
\end{equation}
By \eqref{measdis2}, \eqref{measdis3}, \eqref{measdis4} and Lemma \ref{seqconv}, we conclude that for any compactly supported real-valued continuous function $f$,
\begin{equation}\label{measdis5}
\int_0^{\infty} f(x) \tZ^r(\cdot)(dx) \xrightarrow{d} \int_0^{\infty} f(x) \tZ(\cdot)(dx), \qquad\mbox{ as } r\to\infty.
\end{equation}
From \eqref{measdis1} and \eqref{measdis5}, \eqref{measdismain} is verified and hence, by Theorem 2.1 in \cite{roelly1986}, $\{\tZ^r(\cdot), \; r\in\clr\}$
is tight in $\cld([0,T]:\clm_F)$.

Suppose along a subsequence $\tZ^r(\cdot)\Rightarrow \tZ^*(\cdot)$ as $r\to\infty$.
By {\color{black}the} continuous mapping theorem, for any $k\in\NN$ and compactly supported real-valued continuous functions $G_1, \ldots, G_k$,
$$\left(\lan G_1, \tZ^r(\cdot)\ran, \ldots, \lan G_k, \tZ^r(\cdot)\ran\right) \Rightarrow 
\left(\lan G_1,  \tZ^*(\cdot)\ran, \ldots, \lan G_k,  \tZ^*(\cdot)\ran\right) \qquad\mbox{ as } r\to\infty.$$
But also, from \eqref{measdis5}, the Cramér–Wold theorem and using the linearity of the integral,
$$\left(\lan G_1, \tZ^r(\cdot)\ran, \ldots, \lan G_k, \tZ^r(\cdot)\ran\right) \Rightarrow 
\left(\lan G_1,  \tZ(\cdot)\ran, \ldots, \lan G_k,  \tZ(\cdot)\ran\right) \qquad\mbox{ as } r\to\infty.$$
Thus, $\left(\lan G_1,  \tZ^*(\cdot)\ran, \ldots, \lan G_k,  \tZ^*(\cdot)\ran\right)$ and $\left(\lan G_1,  \tZ(\cdot)\ran, \ldots, \lan G_k,  \tZ(\cdot)\ran\right)$
are equal in distribution.
This shows that $\tZ^*$ has the same law as $\tZ$ (Theorem 3.1 of \cite{kallenberg1974}) and so $\tZ^r$ converges to $\tZ$ in $\cld([0,T]:\clm_F)$
as $r\to\infty$.
\end{proof}

\begin{proof}[Proof of Theorem \ref{tailtZ}] For $x\in(0,\infty]$ and $t\ge 0$, let $\tilde X_x(t)=\xi(\infty)-\xi(x)+X_x(t)$.  By Theorem \ref{workless},
$$
0\le W_\infty(t)-W_x(t)=\Gamma[X_{\infty}](t)-\Gamma[\tilde X_x](t)+\Gamma[\tilde X_x](t)-\Gamma[X_x](t),\qquad\hbox{for all }t\ge 0.
$$
By the Lipschitz property \eqref{eq:lipSM} of $\Gamma$ and Assumption \eqref{initdec},
$$
\lim_{x\to\infty}\sup_{t\ge 0}\lambda^{-1}x^p\left|\Gamma[\tilde X_x](t)-\Gamma[X_x](t)\right|\le \lim_{x\to\infty}2\lambda^{-1}x^p\left|\xi(\infty)-\xi(x)\right|=0.
$$
The previous two displays together with \eqref{eq:deriv} imply that
$$
\lim_{x \rightarrow \infty} \lambda^{-1}x^p\left(W_x(t) - W_{\infty}(t)\right) =-W_{\infty}'(t),\qquad\hbox{for all }t\ge 0.
$$
Fix $t \ge 0$ and let $\epsilon>0$. There exists $x_0>0$ such that for all $x \ge x_0$,
\begin{equation}\label{limZ1}
\left| \lambda^{-1}x^p\left(W_x(t) - W_{\infty}(t)\right) + W_{\infty}'(t)\right| < \epsilon.
\end{equation}
This implies that, for $a \ge x_0$,
\begin{multline}\label{limZ2}
\left| \int_a^{\infty} x^{-2}W_x(t) dx - \frac{W_{\infty}(t)}{a} + \frac{\lambda}{(p+1)a^{p+1}}W_{\infty}'(t)\right|\\
 \le \int_a^{\infty}\lambda x^{-p-2}\left| \lambda^{-1}x^p\left(W_x(t) - W_{\infty}(t)\right) + W_{\infty}'(t)\right|dx \le  \frac{\lambda \epsilon}{(p+1)a^{p+1}}.
\end{multline}
By Theorem \ref{meascon}, for all $a\in(0,\infty)$,
$$
\tZ(t)[a,\infty) = \int_a^{\infty}\frac{1}{x^2}W_x(t)dx  - \frac{W_a(t)}{a}.
$$
Thus, from \eqref{limZ1} (with $x=a$) and \eqref{limZ2}, for any $a \ge x_0$,
\begin{align}\label{limZ3}
&\left| \tZ(t)[a,\infty) - \frac{p \lambda}{(p+1)a^{p+1}}W_{\infty}'(t) \right|
=\left| \int_a^{\infty}\frac{1}{x^2}W_x(t)dx  - \frac{W_a(t)}{a}- \frac{p \lambda}{(p+1)a^{p+1}}W_{\infty}'(t) \right|\\
&=\left|  \frac{W_{\infty}(t)}{a}- \frac{W_{a}(t)}{a}- \frac{ \lambda}{a^{p+1}}W_{\infty}'(t)+\int_a^{\infty}\frac{1}{x^2}W_x(t)dx  - \frac{W_{\infty}(t)}{a}+ \frac{\lambda}{(p+1)a^{p+1}}W_{\infty}'(t)\right|\nonumber\\
& \le \lambda a^{-p-1}\left| \lambda^{-1}a^p\left(W_a(t) - W_{\infty}(t)\right) + W_{\infty}'(t)\right|\nonumber\\
 &\qquad +
\left| \int_a^{\infty} x^{-2}W_x(t) dx - \frac{W_{\infty}(t)}{a} + \frac{\lambda}{(p+1)a^{p+1}}W_{\infty}'(t)\right|\nonumber \\
&\le \frac{\lambda \epsilon}{a^{p+1}} + \frac{\lambda \epsilon}{(p+1)a^{p+1}}.\nonumber
\end{align}
As $\epsilon>0$ is arbitrary, the first two limits claimed in the theorem follow from \eqref{limZ1} and \eqref{limZ3}. To prove the last limit, note that by the first two limit results of the theorem, for any $t$ such that $W_{\infty}'(t) \neq 0$,
\begin{equation}\label{limZ4}
\frac{p \langle \chi \one_{[a, \infty)}, \tZ(t) \rangle}{(p+1)a\tZ(t)[a, \infty)} \rightarrow 1 \ \text{ as } a \rightarrow \infty.
\end{equation}
Moreover, for each $a>0$,
$$
\frac{p\mathbb{E}\left(v \ \vert \ v > a\right)}{(p+1)a} = \frac{p}{(p+1)a}\left(\frac{a\overline{F}(a) + \int_a^{\infty}\overline{F}(x)dx}{\overline{F}(a)}\right)\\
 = \frac{p}{(p+1)}\left(1 + \frac{\int_a^{\infty}\overline{F}(x)dx}{a\overline{F}(a)}\right).
 $$
This together with \eqref{intregv} gives
\begin{equation}\label{limZ5}
\lim_{a\to\infty}\frac{p\mathbb{E}\left(v \ \vert \ v > a\right)}{(p+1)a} = \frac{p}{(p+1)}\left(1 + \frac{1}{p}\right) = 1.
\end{equation}
The last limit claimed in the theorem follows from \eqref{limZ4} and \eqref{limZ5}.
\end{proof}

\subsection{Proof of Theorem \ref{collapse}}\label{sec:pfsc}

In this section, we prove Theorem \ref{collapse}, which concerns an asymptotic relationship between of the limiting
processes $Q^{(p)}$ and $W^{(p)}$ as $p\to\infty$.  As in Section \ref{sec:sc}, we consider $p\ge 2$ and index all limiting
processes (resp.\ parameters and constants) that depend on $p$ with the superscript $(p)$ (resp.\ an argument of $p$).
In addition, we assume that the asymptotic conditions stated in Section \ref{sec:sc} hold.

\begin{proof}[Proof of Theorem \ref{collapse}]
Recall that, for all $p\ge 2$,
$$
Q^{(p)}(t) = \int_0^{\infty}\frac{1}{x^2}W^{(p)}_x(t)dx, \quad t \ge 0,
$$
where
$$
W^{(p)}_a(t) := \Gamma[X^{(p)}_a](t), \quad {\color{black}t \ge 0, \ a >0},
$$
with $\Gamma$ denoting the Skorohod map and
$$
X^{(p)}_a(t) := \xi^{(p)}(a) + \sigmap B(t) + \left(\kappa - \frac{\lambda}{a^p}\right)t, \quad t \ge 0, \ a >0.
$$
Let $T, \gamma>0$.  
Take any $\vartheta>0$. Note that, for any $p\ge 2$ and $\epsilon \in (0,1)$,
\begin{equation}\label{pnew1}
\sup_{t \in [0,T]}\int_{1-\epsilon}^1x^{-2}W^{(p)}_x(t)dx \le \frac{\epsilon}{1-\epsilon}\sup_{t \in [0,T]}W^{(p)}_{\infty}(t).
\end{equation}
Using the Lipschitz property \eqref{eq:lipSM} of the Skorohod map, that $\sigmap=\sqrt{\lambda \operatorname{Var}(v^{(p)}) + \lambda\sigma_A^2}$ and Assumption \eqref{varunibdd}, for all $p\ge 2$,
\begin{multline}\label{pnew2}
\mathbb{E}\left[\sup_{t \in [0,T]}W^{(p)}_{\infty}(t)\right] \le 2\mathbb{E}\left[\sup_{t \in [0,T]}\left(\xi^{(p)}(\infty) + \sigmap |B(t)| + \kappa t\right)\right]\\
 \le 2\sup_{p \ge 2}\mathbb{E}\left[\xi^{(p)}(\infty)\right] + 2 \sqrt{\sup_{p \ge 2}\left(\lambda \operatorname{Var}(v^{(p)}) + \lambda\sigma_A^2\right)}\mathbb{E}\left[\sup_{t \in [0,T]}|B(t)|\right] + 2\kappa T := \mathcal{B} < \infty,
\end{multline}
where the bound $\mathcal{B}$ does not depend on $p$. Hence, by \eqref{pnew1}, \eqref{pnew2} and Markov's inequality, we can choose $\epsilon \in (0,1)$ such that
\begin{equation}\label{middlesmall}
\mathbb{P}\left(\sup_{t \in [0,T]}\int_{1-\epsilon}^1x^{-2}W^{(p)}_x(t)dx > \gamma/3\right) \le \vartheta \ \text{ for all } p\ge 2.
\end{equation}
By \eqref{assp}, we obtain $p_0'\ge 2$ such that
\begin{equation}\label{eq:petap}
\frac{p-1-\eta^*(p)}{\log p} > \frac{4}{\log((1-\epsilon)^{-1})} \ \text{ for all } p \ge p_0'.
\end{equation}
For each $p\ge 2$, let $\mop, \aop$ be defined as in Lemma \ref{BMsup}.
Since $\mop>1$ for all $p\ge 2$, $\aop\in(0,1)$ for all $p\ge 2$.  Due to \eqref{varunibdd}, $0<\inf_{p\ge 2}\sigmap\le \sup_{p\ge 2}\sigmap<\infty$.  Thus,
$\lim_{p\to\infty}\aop=\lim_{p\to\infty}\mop^{-1/2p}=1$.
Take $p_0 \ge p_0'$ such that for all $p \ge p_0$, $\aop > 1-\epsilon$.  For $p\ge 2$, write 
$$
H'(p) := H(p, (p-1 + \eta^*(p))/2) \ \ \hbox{and} \ \ C'(p,\lambda,\sigmap) :=\tilde C\left((p-1 + \eta^*(p))/2, \eta^*(p), \lambda, \sigmap\right),
$$
where the functions $H$ and $\tilde C$ were defined in Lemma \ref{flowsup}. Then, by Lemma \ref{flowsup}, taking $\delta = 1-\epsilon$ and $\eta = (p-1 + \eta^*(p))/2$, we obtain for any $p \ge p_0$,
\begin{multline}\label{intineold}
\mathbb{P}\left(\sup_{t \in [0,T]}\int_0^{1-\epsilon}x^{-2}W_x^{(p)}(t)dx > H'(p) (1-\epsilon)^{(p -1-\eta^*(p))/2}(1+\log((1-\epsilon)^{-1}))\right)\\
\le C'(p,\lambda, \sigmap)(1-\epsilon)^{(p -1-\eta^*(p))/2} + 3C(\lambda, \sigmap)(1-\epsilon)^{2p}.
\end{multline}
Using the explicit forms of $C'(p,\lambda, \sigmap)$ (defined in Lemma \ref{flowsup}) and $C(\lambda, \sigmap)$ (defined in Lemma \ref{BMsup}), Assumption \eqref{varunibdd}, and \eqref{eq:petap}, and recalling $\inf_{p\ge 2}\sigmap >0$, note that
\begin{align*}
C'(\lambda) &:= \sup_{p \ge 2}C'(p,\lambda, \sigmap)\\
& = \sup_{p \ge 2}\left(\cop + \frac{2(\sigmap)^2}{\lambda}\left(\sup_{x\in\RRp}xe^{-x}\right)\right)\left(1 + \sum_{k=1}^{\infty}2^{-(k-1)(p-1 - \eta^*(p))/2)}\right) < \infty,
\end{align*}
and
\begin{align*}
C(\lambda) := \sup_{p \ge 2} C(\lambda, \sigmap) = \sup_{p \ge 2}\left(2e^{\lambda/(\sigmap)^2} + \frac{16(\sigmap)^2}{\lambda}\right) < \infty.
\end{align*}
Using these observations in \eqref{intineold}, we obtain for any $p \ge p_0$,
\begin{multline}\label{intine}
\mathbb{P}\left(\sup_{t \in [0,T]}\int_0^{1-\epsilon}x^{-2}W_x^{(p)}(t)dx > H'(p) (1-\epsilon)^{(p -1-\eta^*(p))/2}(1+\log((1-\epsilon)^{-1}))\right)\\
\le C'(\lambda)(1-\epsilon)^{(p -1-\eta^*(p))/2} + 3C(\lambda)(1-\epsilon)^{2p}.
\end{multline}
Using \eqref{eq:petap}, we have that $\log p + \frac{p-1-\eta^*(p)}{2}\log(1-\epsilon) \rightarrow -\infty$ as $p \rightarrow \infty$. Exponentiating, we obtain
$$
p(1-\epsilon)^{(p -1-\eta^*(p))/2}(1+\log((1-\epsilon)^{-1})) \rightarrow 0 \text{ as } p \rightarrow \infty.
$$
From this and the explicit form of $H(p,\eta)$ given in Lemma \ref{flowsup}, we conclude that $H'(p) (1-\epsilon)^{(p -1-\eta^*(p))/2}(1+\log((1-\epsilon)^{-1})) \rightarrow 0$ as $p \rightarrow \infty$. Moreover, the right hand side of \eqref{intine} also goes to zero as $p \rightarrow \infty$. Thus,
\begin{equation}\label{collapse1}
\sup_{t \in [0,T]}\int_0^{1-\epsilon}x^{-2}W^{(p)}_x(t)dx \xrightarrow{P} 0 \ \text{ as } p \rightarrow \infty.
\end{equation}
Moreover, by the Lipschitz property \eqref{eq:lipSM} and Assumption \eqref{assp2}, for each $x \in (1, \infty)$, 
\begin{align*}
\mathbb{E}\left(\sup_{t \in [0,T]}|W^{(p)}_x(t) - W^{(p)}_{\infty}(t)|\right) &\le 2\mathbb{E}\left(\sup_{t \in [0,T]}|X^{(p)}_x(t) - X^{(p)}_{\infty}(t)|\right)\\
&\le 2\mathbb{E}\left(\xi^{(p)}(\infty) - \xi^{(p)}(x)\right) + \frac{2\lambda T}{x^p} \rightarrow 0 \ \text{ as } p \rightarrow \infty,
\end{align*}
where {\color{black}we} recall $X^{(p)}_{\infty}(t) = \xi^{(p)}(\infty) + \sigma^{(p)} B(t) + \kappa t, \ t \ge 0$. By the monotonicity property noted in \eqref{eq:monoskor} and using \eqref{pnew2}, for all $p\ge 2$,
$$
\mathbb{E}\left(\sup_{t \in [0,T]}|W^{(p)}_x(t) - W^{(p)}_{\infty}(t)|\right) \le \mathbb{E}\left(\sup_{t \in [0,T]}W^{(p)}_{\infty}(t)\right) \le \mathcal{B}.
$$
Thus, by the dominated convergence theorem, 
$$
\int_{1}^{\infty}x^{-2}\mathbb{E}\left(\sup_{t \in [0,T]}|W^{(p)}_x(t) - W^{(p)}_{\infty}(t)|\right)dx \rightarrow 0 \ \text{ as } p \rightarrow \infty,
$$
which implies
\begin{equation}\label{collapse2}
\sup_{t \in [0,T]}\left| \int_{1}^{\infty}x^{-2}W^{(p)}_x(t)dx -  W^{(p)}_{\infty}(t)\right| \xrightarrow{P} 0 \ \text{ as } p \rightarrow \infty.
\end{equation}
From \eqref{middlesmall}, \eqref{collapse1} and \eqref{collapse2},
\begin{align*}
&\limsup_{p \rightarrow \infty}\mathbb{P}\left(\sup_{t \in [0,T]}\left|\int_{0}^{\infty}x^{-2}W^{(p)}_x(t)dx -  W^{(p)}_{\infty}(t)\right| > \gamma\right)\\
&\le \limsup_{p \rightarrow \infty}\mathbb{P}\left(\sup_{t \in [0,T]}\int_{0}^{1-\epsilon}x^{-2}W^{(p)}_x(t)dx > \gamma/3\right)\\
&\qquad + \limsup_{p \rightarrow \infty} \mathbb{P}\left(\sup_{t \in [0,T]}\int_{1-\epsilon}^1x^{-2}W^{(p)}_x(t)dx > \gamma/3\right)\\
&\qquad + \limsup_{p \rightarrow \infty} \mathbb{P}\left(\sup_{t \in [0,T]}\left| \int_{1}^{\infty}x^{-2}W^{(p)}_x(t)dx -  W^{(p)}_{\infty}(t)\right| > \gamma/3\right) \le \vartheta.
\end{align*}
As $T$, $\gamma$, $\vartheta>0$ are arbitrary, the theorem is proved.
\end{proof}

\appendix

\section{Verifying Assumptions $\eqref{eq:assuinitcond}-\eqref{smalljobas}$ for some initial conditions}\label{icverify}

\subsection{Checking Assumptions \eqref{eq:assuinitcond}--\eqref{smalljobas} at fixed time $t >0$ for a sequence of systems with $\qq^r=0$ for all $r\in\clr$}
Here we sketch how to verify that if each system in the sequence starts
with zero jobs  then at any time $t>0$, Assumptions \eqref{eq:assuinitcond}--\eqref{smalljobas}
are satisfied with $\left( W_{\cdot}^r(0), W_{\infty}^r(0)\right)$ replaced by $\left( W_{\cdot}^r(t), W_{\infty}^r(t)\right)$ for each $r\in\clr$ and $\{\breve{v}_l^r\}_{1\le l \le \qq^r}$ replaced with 
$\{v_i(r^2t) : 1 \le i \le E^r(r^2t), v_i(r^2t) >0\} \cup \{\breve{v}_l^r(r^2t), 1 \le l \le \qq^r, \breve{v}_l^r(r^2t) >0\}$. Fix $t>0$ and note that since $\qq^r=0$ for all $r\in\clr$,
Assumptions \eqref{eq:assuinitcond}--\eqref{smalljobas} hold at time zero.  Thus, Theorem \ref{workless}, along with tightness arguments similar to those in the proof of Theorem \ref{convplus} and the estimates in \eqref{eq:vabcr} and \eqref{delz2}, can be used to show that for any fixed $t >0$, \eqref{eq:assuinitcond} holds with $\left( W_{\cdot}^r(0), W_{\infty}^r(0)\right)$ replaced by $\left( W_{\cdot}^r(t), W_{\infty}^r(t)\right)$ for each $r\in\clr$ and $(w^*(\cdot), w^*(\infty))$ replaced by $(W_{\cdot}(t), W_{\infty}(t))$, where $W$ is defined in Theorem \ref{workless} with $\xi(a)=0$ for all $a\in[0,\infty]$. The uniform integrability assumption \eqref{eq:assuinitcondb} can be shown to hold for $\{W_{\infty}^r(t), r\in\clr\}$ by first noting that for each $r\in\clr$, $W_{\infty}^r(t) = \Gamma[X^r_{\infty}](t)$, where $\Gamma$ is the Skorohod map defined in \eqref{eq:skormap} and $X^r_{\infty}(\cdot)$ is defined in \eqref{infwork} (taking $X^r_{\infty}(0) = 0$). By \eqref{eq:assuht}, the finiteness of $\operatorname{Var}(v)$ and by applications of Doob's $L^2$-maximal inequality and Azuma-Hoeffding inequality, we can obtain for any $t>0$ that $\mathbb{E}\left[\left(\sup_{0 \le s \le t}X^r_{\infty}(s)\right)^{\beta}\right] < \infty$ for any $\beta \in (1,2)$. From this observation and the Lipschitz property of the Skorohod map stated in \eqref{eq:lipSM}, we can deduce $\{W_{\infty}^r(t) : r \in \mathcal{R}\}$ is $L^\beta$-bounded for any $\beta \in (1,2)$ and thus \eqref{eq:assuinitcondb} holds. Assumption \eqref{eq:initcondq1} follows along the same lines as the proof of Lemmas \ref{fifo}, \ref{unifsmallint} and \ref{smallql} (see Remark \ref{uiinit}). Assumption \eqref{eq:initcondq2} follows by recalling that $(w^*(\cdot), w^*(\infty)) = (W_{\cdot}(t), W_{\infty}(t))$ and using the explicit form of $W$ defined in Theorem \ref{workless} and the Lipschitz property \eqref{eq:lipSM} of the Skorohod map. Finally, \eqref{smalljobas} follows from Proposition \ref{comp} and Lemma \ref{fifo}.

\subsection{Checking Assumptions \eqref{eq:assuinitcond}--\eqref{smalljobas} for initial conditions (I) given in Subsection \ref{iceg}}

We first show that \eqref{eq:assuinitcond} holds. For $0 \le x \le \infty$, define $\hat{W}^r(x) := \frac{c^r\qq^r}{r}\mathbb{E}\left(\frac{\breve{v}_1^r}{c^r} \one_{[\breve{v}_1^r \le xc^r]}\right)$. For any $A \in (0, \infty)$,
\begin{align}\label{innat1}
\sup_{x \in [0,A]}&\left|\mathbb{E}\left(\frac{\breve{v}_1^r}{c^r}\one_{[\breve{v}_1^r \le xc^r]}\right) - \mathbb{E}\left(\breve{v}^* \one_{[\breve{v}^* \le x]}\right)\right|\nonumber\\
& \le \sup_{x \in [0,A]}\left| \int_0^x\mathbb{P}\left(z c^r < \breve{v}_1^r \le xc^r\right)dz - \int_0^x\mathbb{P}\left(z < \breve{v}^* \le x\right)dz\right|\nonumber\\
&\le \sup_{x \in [0,A]} \left(x \left|\mathbb{P}\left(\breve{v}_1^r \le xc^r\right) - \mathbb{P}\left(\breve{v}^* \le x\right)\right| + \int_0^x \left|\mathbb{P}\left(\breve{v}_1^r \le z c^r\right) - \mathbb{P}\left(\breve{v}^* \le z\right)\right|dz \right)\nonumber\\
&\le 2A \sup_{x \in [0,A]}\left|\mathbb{P}\left(\breve{v}_1^r \le xc^r\right) - \mathbb{P}\left(\breve{v}^* \le x\right)\right| \rightarrow 0 \quad \text{ as } \quad r \rightarrow \infty
\end{align}
by P\'olya's Theorem \cite[Exercise 3.2.9, Page 107]{durrett}, as $\breve{v}^*$ has a continuous distribution. As the map $x \mapsto \mathbb{E}\left(\breve{v}^* \one_{[\breve{v}^* \le x]}\right)$ is continuous by (iii), it follows from \eqref{innat1} and (iii) that for any $\epsilon>0$, there exists $A \in (0, \infty)$, $\delta>0$ and $r_0 \in \mathcal{R}$ such that for all $r \ge r_0$,
\begin{equation}\label{innat2}
\sup_{0 \le x \le y \le A, \ y-x \le \delta}\mathbb{E}\left(\frac{\breve{v}_1^r}{c^r}\one_{[xc^r < \breve{v}_1^r \le yc^r]}\right) < \epsilon \quad \text{ and } \quad \mathbb{E}\left(\frac{\breve{v}_1^r}{c^r}\one_{[\breve{v}_1^r > Ac^r]}\right) < \epsilon.
\end{equation}
{\color{black}
Also, by (i), for each $r\in\clr$ and $0\le x\le \infty$,
\begin{eqnarray*}
\mathbb{E}\left(W_x^r(0) - \hat{W}^r(x)\right)^2
&=&
\frac{\left(c^r\right)^2\mathbb{E}(\qq^r)}{r^2}\mathbb{E}\left(\frac{\breve{v}_1^r}{c^r} \one_{[\breve{v}_i^r \le xc^r]}-\mathbb{E}\left(\frac{\breve{v}_1^r}{c^r} \one_{[\breve{v}_1^r \le xc^r]}\right)\right)^2\\
&\le& 
\frac{\left(c^r\right)^2\mathbb{E}(\qq^r)}{r^2}\mathbb{E}\left(\frac{\breve{v}_1^r}{c^r} \one_{[\breve{v}_1^r \le xc^r]}\right)^2
\le \frac{\left(c^r\right)^2\mathbb{E}(\qq^r)}{r^2}\mathbb{E}\left(\frac{\breve{v}_1^r}{c^r}\right)^2.
\end{eqnarray*}
This together with} the conditions (ii) and (iii) {\color{black}and \eqref{limitcroverr}} imply that 
\begin{equation}\label{innat3}
\sup_{0 \le x \le \infty}\mathbb{E}\left(W_x^r(0) - \hat{W}^r(x)\right)^2 \rightarrow 0 \quad \text{ as } \quad r \rightarrow \infty.
\end{equation}
Fix any $\epsilon>0$. Note that by (ii), $\sup_{r \in \mathcal{R}}\mathbb{E}\left(\frac{c^r\qq^r}{r}\right) < \infty$. By \eqref{innat2} and \eqref{innat3}, one can obtain a partition $0=x_0 < x_1 <\cdots <x_k<x_{k+1}=\infty$ of $[0,\infty]$ and $r_1 \in \mathcal{R}$ such that the following hold for all $r \ge r_1$:
\begin{equation}\label{part1}
\mathbb{E}\left|W_{x_j}^r(0) - \hat{W}^r(x_j)\right| < \frac{\epsilon}{2(k+2)} \quad \text{ for all } \quad 0 \le j \le k+1,
\end{equation}
and
\begin{equation}\label{part2}
\mathbb{E}\left(\frac{\breve{v}_1^r}{c^r}\one_{[x_jc^r < \breve{v}_1^r \le x_{j+1}c^r]}\right) < \frac{\epsilon}{2 \sup_{r \in \mathcal{R}}\mathbb{E}\left(\frac{c^r\qq^r}{r}\right)} \quad \text{ for all } \quad 0 \le j \le k.
\end{equation}
By the monotonicity of the maps $x \mapsto W_x^r(0)$ and $x \mapsto \hat{W}^r(x)$, one obtains the bound
\begin{align*}
\sup_{x \in [0, \infty]}\left|W_x^r(0) - \hat{W}^r(x)\right| &\le \sup_{0 \le j \le k+1}\left|W_{x_j}^r(0) - \hat{W}^r(x_j)\right| + \sup_{0 \le j \le k}\left|\hat{W}^r(x_{j+1}) - \hat{W}^r(x_j)\right|\\
&\le \sum_{j=0}^{k+1}\left|W_{x_j}^r(0) - \hat{W}^r(x_j)\right| + \sup_{0 \le j \le k}\left|\hat{W}^r(x_{j+1}) - \hat{W}^r(x_j)\right|.
\end{align*}
Hence, using \eqref{part1} and \eqref{part2} and (i), we obtain for $r \ge r_1$
\begin{align*}
\mathbb{E}&\left(\sup_{x \in [0, \infty]}\left|W_x^r(0) - \hat{W}^r(x)\right|\right)\\
&\le \sum_{j=0}^{k+1}\mathbb{E}\left|W_{x_j}^r(0) - \hat{W}^r(x_j)\right|
 + \mathbb{E}\left(\frac{c^r\qq^r}{r}\right)\sup_{0 \le j \le k}\mathbb{E}\left(\frac{\breve{v}_1^r}{c^r}\one_{[x_jc^r < \breve{v}_1^r \le x_{j+1}c^r]}\right)
< \epsilon.
\end{align*}
As $\epsilon>0$ is arbitrary, we conclude
\begin{equation}\label{mainin1}
\lim_{r \rightarrow \infty}\mathbb{E}\left(\sup_{x \in [0, \infty]}\left|W_x^r(0) - \hat{W}^r(x)\right|\right) = 0.
\end{equation}
Note that for any $\epsilon>0$, by the second assertion of \eqref{innat2} and that fact that $\mathbb{E}\left(\breve{v}_1^r/c^r\right) \rightarrow \mathbb{E}\left(\breve{v}^*\right)< \infty$ as $r \rightarrow \infty$, which follows from (iii), we can obtain $A>0$, $r_2 \in \mathcal{R}$ such that for all $r \ge r_2$
\begin{align*}
\sup_{x \in [A, \infty]}&\left|\mathbb{E}\left(\frac{\breve{v}_1^r}{c^r}\one_{[\breve{v}_1^r \le xc^r]}\right) - \mathbb{E}\left(\breve{v}^* \one_{[\breve{v}^* \le x]}\right)\right|\\
& \le \left|\mathbb{E}\left(\breve{v}_1^r/c^r\right) - \mathbb{E}\left(\breve{v}^*\right)\right| + \sup_{x \in [A, \infty]}\left|\mathbb{E}\left(\frac{\breve{v}_1^r}{c^r}\one_{[\breve{v}_1^r > xc^r]}\right) - \mathbb{E}\left(\breve{v}^* \one_{[\breve{v}^* > x]}\right)\right|< \epsilon.
\end{align*}
Combining this with \eqref{innat1}, we obtain
\begin{equation}\label{mainin2}
\lim_{r \rightarrow \infty}\sup_{x \in [0,\infty]}\left|\mathbb{E}\left(\frac{\breve{v}_1^r}{c^r}\one_{[\breve{v}_1^r \le xc^r]}\right) - \mathbb{E}\left(\breve{v}^* \one_{[\breve{v}^* \le x]}\right)\right| = 0.
\end{equation}
Defining $\tilde{W}^r(x) := \frac{c^r\qq^r}{r}\mathbb{E}\left(\breve{v}^* \one_{[\breve{v}^* \le x]}\right)$ for $x \in [0,\infty]$, we conclude from \eqref{mainin1}, \eqref{mainin2} and the fact $\sup_{r \in \clr}\mathbb{E}\left(c^r\qq^r/r\right) < \infty$ that
\begin{equation}\label{mainin}
\lim_{r \rightarrow \infty}\mathbb{E}\left(\sup_{x \in [0, \infty]}\left|W_x^r(0) - \tilde{W}^r(x)\right|\right) = 0.
\end{equation}
Finally as $c^r\qq^r/r \xrightarrow{L^1} \qq^*$ as $r \rightarrow \infty$ by (ii), \eqref{mainin} implies that Assumption \eqref{eq:assuinitcond} holds with the given choice of $w^*(\cdot)$.
In fact we have shown that
\begin{equation}\label{eq:liwrw}
	\lim_{r \rightarrow \infty}\mathbb{E}\left(\sup_{x \in [0, \infty]}\left|W_x^r(0) - w^*(x)\right|\right) = 0.
	\end{equation}
Assumption \eqref{eq:assuinitcondb} follows from the observation that $W_{\infty}^r(0) \xrightarrow{L^1} w^*(\infty)$ which holds by \eqref{eq:liwrw}. Assumption \eqref{eq:initcondq1} is a direct consequence of (i), (ii) and (iv). Assumption \eqref{eq:initcondq2} follows from (i), (ii) and the observation that $\mathbb{E}\left(\breve{v}^*\right) < \infty$ which follows from (iii) and Fatou's Lemma. Assumption \eqref{smalljobas} follows from (ii) and (iii).\\

\textbf{Acknowledgement: }The authors acknowledge the Open Problem Session held during the Seminar on Stochastic Processes, 2019, where the addressed problem was presented as open by AP. The authors also thank the anonymous referees for very helpful advice. 
Research of SB was supported in part by a Junior Faculty Development Award.
AB acknowledges support from the National Science Foundation (DMS-1814894 and DMS-1853968). He is also grateful for the support from Nelder Fellowship from Imperial College, London, where part of this research was completed.
Research of AP was supported in part by National Science Foundation grants DMS-1712974 and DMS-2054505 and the Charles Lee Powell Foundation.
\bibliographystyle{imsart-number}
\bibliography{NSFbib,bibl,scaling,ap}

\begin{thebibliography}{36}

\bibitem{ABKR18}
\begin{barticle}[author]
\bauthor{\bsnm{Atar},~\bfnm{R.}\binits{R.}},
  \bauthor{\bsnm{Biswas},~\bfnm{A.}\binits{A.}},
  \bauthor{\bsnm{Kaspi},~\bfnm{H.}\binits{H.}} \AND
  \bauthor{\bsnm{Ramanan},~\bfnm{K.}\binits{K.}}
(\byear{2018}).
\btitle{A {S}korohod map on measure-valued paths with applications to priority
  queues}.
\bjournal{Annals of Applied Probability}
\bvolume{28}
\bpages{418--481}.
\end{barticle}
\endbibitem

\bibitem{BHB01}
\begin{binproceedings}[author]
\bauthor{\bsnm{Bansal},~\bfnm{N.}\binits{N.}} \AND
  \bauthor{\bsnm{Harchol-Balter},~\bfnm{M.}\binits{M.}}
(\byear{2001}).
\btitle{Analysis of {S}{R}{P}{T} scheduling: {I}nvestigating unfairness}.
In \bbooktitle{ACM SIGMETRICS 2001 Conference on Measurement and Modeling of
  Computer Systems}
\bpages{279--290}.
\end{binproceedings}
\endbibitem

\bibitem{BCM98}
\begin{binproceedings}[author]
\bauthor{\bsnm{Bender},~\bfnm{M.}\binits{M.}},
  \bauthor{\bsnm{Chakrabarti},~\bfnm{S.}\binits{S.}} \AND
  \bauthor{\bsnm{Muthukrishnan},~\bfnm{S.}\binits{S.}}
(\byear{1998}).
\btitle{Flow and stretch metrics for scheduling continous job streams}.
In \bbooktitle{Proceedings of the 9th Annual ACM-SIAM Symposium on Discrete
  Algorithms}.
\end{binproceedings}
\endbibitem

\bibitem{billingsley2013}
\begin{bbook}[author]
\bauthor{\bsnm{Billingsley},~\bfnm{Patrick}\binits{P.}}
(\byear{2013}).
\btitle{Convergence of {P}robability {M}easures}.
\bpublisher{John Wiley \& Sons}.
\end{bbook}
\endbibitem

\bibitem{BD01}
\begin{barticle}[author]
\bauthor{\bsnm{Bramson},~\bfnm{M.}\binits{M.}} \AND
  \bauthor{\bsnm{Dai},~\bfnm{J.}\binits{J.}}
(\byear{2001}).
\btitle{Heavy traffic limits for some queueing networks}.
\bjournal{Annals of Applied Probability}
\bvolume{11}
\bpages{49--90}.
\end{barticle}
\endbibitem

\bibitem{CD20}
\begin{barticle}[author]
\bauthor{\bsnm{Chen},~\bfnm{Y.}\binits{Y.}} \AND
  \bauthor{\bsnm{Dong},~\bfnm{J.}\binits{J.}}
(\byear{2020}).
\btitle{Scheduling with Service-Time Information: The Power of Two Priority
  Classes}.
\bjournal{preprint}.
\end{barticle}
\endbibitem

\bibitem{CSN09}
\begin{barticle}[author]
\bauthor{\bsnm{Clauset},~\bfnm{A.}\binits{A.}},
  \bauthor{\bsnm{Shalizi},~\bfnm{C.~R.}\binits{C.~R.}} \AND
  \bauthor{\bsnm{Newman},~\bfnm{M.~E.~J.}\binits{M.~E.~J.}}
(\byear{2009}).
\btitle{Power-law distributions in empirical data}.
\bjournal{SIAM Review}
\bvolume{51}
\bpages{661–703}.
\end{barticle}
\endbibitem

\bibitem{dieker2014}
\begin{barticle}[author]
\bauthor{\bsnm{Dieker},~\bfnm{Antonius~B}\binits{A.~B.}} \AND
  \bauthor{\bsnm{Gao},~\bfnm{Xuefeng}\binits{X.}}
(\byear{2014}).
\btitle{Sensitivity analysis for diffusion processes constrained to an
  orthant}.
\bjournal{The Annals of Applied Probability}
\bvolume{24}
\bpages{1918--1945}.
\end{barticle}
\endbibitem

\bibitem{DGP1109}
\begin{barticle}[author]
\bauthor{\bsnm{Down},~\bfnm{D.}\binits{D.}},
  \bauthor{\bsnm{Gromoll},~\bfnm{H.~C.}\binits{H.~C.}} \AND
  \bauthor{\bsnm{Puha},~\bfnm{A.}\binits{A.}}
(\byear{2009}).
\btitle{Fluid limits for shortest remaining processing time queues}.
\bjournal{Mathematics of Operations Research}
\bvolume{34}
\bpages{880 -- 911}.
\end{barticle}
\endbibitem

\bibitem{DGP909}
\begin{binproceedings}[author]
\bauthor{\bsnm{Down},~\bfnm{D.}\binits{D.}},
  \bauthor{\bsnm{Gromoll},~\bfnm{H.~C.}\binits{H.~C.}} \AND
  \bauthor{\bsnm{Puha},~\bfnm{A.}\binits{A.}}
(\byear{2009}).
\btitle{State-dependent response times via fluid limits for shortest remaining
  processing time queues}.
In \bbooktitle{San Diego ACM-Sigmetrics Performance Evaluation Review}
\bvolume{27}
\bpages{75--76}.
\end{binproceedings}
\endbibitem

\bibitem{durrett}
\begin{bbook}[author]
\bauthor{\bsnm{Durrett},~\bfnm{Rick}\binits{R.}}
(\byear{2019}).
\btitle{Probability: {T}heory and {E}xamples}
\bvolume{49}.
\bpublisher{Cambridge university press}.
\end{bbook}
\endbibitem

\bibitem{ethier2009markov}
\begin{bbook}[author]
\bauthor{\bsnm{Ethier},~\bfnm{Stewart~N}\binits{S.~N.}} \AND
  \bauthor{\bsnm{Kurtz},~\bfnm{Thomas~G}\binits{T.~G.}}
(\byear{2009}).
\btitle{Markov {P}rocesses: {C}haracterization and {C}onvergence}
\bvolume{282}.
\bpublisher{John Wiley \& Sons}.
\end{bbook}
\endbibitem

\bibitem{GKP11}
\begin{barticle}[author]
\bauthor{\bsnm{Gromoll},~\bfnm{H.~C.}\binits{H.~C.}},
  \bauthor{\bsnm{Kruk},~\bfnm{L.}\binits{L.}} \AND
  \bauthor{\bsnm{Puha},~\bfnm{A.}\binits{A.}}
(\byear{2011}).
\btitle{The diffusion limit of an {S}{R}{P}{T} queue}.
\bjournal{Stochastic Systems}
\bvolume{1}
\bpages{1--16}.
\end{barticle}
\endbibitem

\bibitem{IW70}
\begin{barticle}[author]
\bauthor{\bsnm{Iglehart},~\bfnm{D.}\binits{D.}} \AND
  \bauthor{\bsnm{Whitt},~\bfnm{W.}\binits{W.}}
(\byear{1970}).
\btitle{Multiple channel queues in heavy traffic {I}}.
\bjournal{Advances in Applied Probability}
\bvolume{2}
\bpages{150--177}.
\end{barticle}
\endbibitem

\bibitem{kallenberg1974}
\begin{bbook}[author]
\bauthor{\bsnm{Kallenberg},~\bfnm{Olav}\binits{O.}}
(\byear{1974}).
\btitle{Lectures on {R}andom {M}easures}.
\bpublisher{Consolidated University of North Carolina, Institute of
  Statistics.}
\end{bbook}
\endbibitem

\bibitem{karatzas}
\begin{bbook}[author]
\bauthor{\bsnm{Karatzas},~\bfnm{Ioannis}\binits{I.}} \AND
  \bauthor{\bsnm{Shreve},~\bfnm{Steven~E}\binits{S.~E.}}
(\byear{1998}).
\btitle{Brownian {M}otion and {S}tochastic {C}alculus}.
\bpublisher{Springer-Verlag, New York}.
\end{bbook}
\endbibitem

\bibitem{K07}
\begin{barticle}[author]
\bauthor{\bsnm{Kruk},~\bfnm{L.}\binits{L.}}
(\byear{2007}).
\btitle{Diffusion approximation for a {G}/{G}/1 {E}{D}{F} queue with unbounded
  lead times}.
\bjournal{Ann. UMCS Math. A}
\bvolume{61}
\bpages{51--90}.
\end{barticle}
\endbibitem

\bibitem{K19}
\begin{binproceedings}[author]
\bauthor{\bsnm{Kruk},~\bfnm{L.}\binits{L.}}
(\byear{2019}).
\btitle{Diffusion Limits for {S}{R}{P}{T} and {L}{R}{P}{T} Queues via {E}{D}{F}
  Approximations}.
In \bbooktitle{Queueing Theory and Network Applications, 14th International
  Conference, QTNA 2019, Ghent, Belgium}.
\end{binproceedings}
\endbibitem

\bibitem{K16}
\begin{barticle}[author]
\bauthor{\bsnm{Kruk},~\bfnm{L.}\binits{L.}} \AND
  \bauthor{\bsnm{Sokołowska},~\bfnm{E.}\binits{E.}}
(\byear{2016}).
\btitle{Fluid limits for multiple-input shortest remaining processing time
  queues}.
\bjournal{Mathematics of Operations Research}
\bvolume{41}
\bpages{1055--1092}.
\end{barticle}
\endbibitem

\bibitem{LWZ11}
\begin{barticle}[author]
\bauthor{\bsnm{Lin},~\bfnm{M.}\binits{M.}},
  \bauthor{\bsnm{Wierman},~\bfnm{A.}\binits{A.}} \AND
  \bauthor{\bsnm{Zwart},~\bfnm{B.}\binits{B.}}
(\byear{2011}).
\btitle{The heavy-traffic growth rate of shortest remaining processing time}.
\bjournal{Performance Evaluation}
\bvolume{68}
\bpages{955--966}.
\end{barticle}
\endbibitem

\bibitem{L12}
\begin{barticle}[author]
\bauthor{\bsnm{Loboz},~\bfnm{C.}\binits{C.}}
(\byear{2012}).
\btitle{Cloud Resource Usage—Heavy Tailed Distributions Invalidating
  Traditional Capacity Planning Models}.
\bjournal{Journal of Grid Computing}
\bvolume{10}
\bpages{85–108}.
\end{barticle}
\endbibitem

\bibitem{mandelbaum2010}
\begin{barticle}[author]
\bauthor{\bsnm{Mandelbaum},~\bfnm{Avi}\binits{A.}} \AND
  \bauthor{\bsnm{Ramanan},~\bfnm{Kavita}\binits{K.}}
(\byear{2010}).
\btitle{Directional derivatives of oblique reflection maps}.
\bjournal{Mathematics of Operations Research}
\bvolume{35}
\bpages{527--558}.
\end{barticle}
\endbibitem

\bibitem{mikosch1999}
\begin{bbook}[author]
\bauthor{\bsnm{Mikosch},~\bfnm{Thomas}\binits{T.}}
(\byear{1999}).
\btitle{Regular variation, subexponentiality and their applications in
  probability theory}.
\bpublisher{Eindhoven University of Technology}.
\end{bbook}
\endbibitem

\bibitem{P93}
\begin{barticle}[author]
\bauthor{\bsnm{Perera},~\bfnm{R.}\binits{R.}}
(\byear{1993}).
\btitle{The variance of delay time in queueing system {M}/{G}/1 with optimal
  strategy {S}{R}{P}{T}}.
\bjournal{Archiv f\"{u}r Elektronik und \"{U}bertragungstechnik}
\bvolume{47}
\bpages{110--114}.
\end{barticle}
\endbibitem

\bibitem{P15}
\begin{barticle}[author]
\bauthor{\bsnm{Puha},~\bfnm{A.}\binits{A.}}
(\byear{2015}).
\btitle{Diffusion limits for shortest remaining processing time queues under
  nonstandard spatial scaling}.
\bjournal{Annals of Applied Probability}
\bvolume{25}
\bpages{3381--3404}.
\end{barticle}
\endbibitem

\bibitem{roelly1986}
\begin{barticle}[author]
\bauthor{\bsnm{Roelly-Coppoletta},~\bfnm{Sylvie}\binits{S.}}
(\byear{1986}).
\btitle{A criterion of convergence of measure-valued processes: {A}pplication
  to measure branching processes}.
\bjournal{Stochastics: An International Journal of Probability and Stochastic
  Processes}
\bvolume{17}
\bpages{43--65}.
\end{barticle}
\endbibitem

\bibitem{S90}
\begin{barticle}[author]
\bauthor{\bsnm{Schassberger},~\bfnm{R.}\binits{R.}}
(\byear{1990}).
\btitle{The steady-state appearance of the {M}/{G}/1 queue under the discipline
  of shortest remaining processing time}.
\bjournal{Advances in Applied Probability}
\bvolume{22}
\bpages{456--479}.
\end{barticle}
\endbibitem

\bibitem{S68}
\begin{barticle}[author]
\bauthor{\bsnm{Schrage},~\bfnm{L.}\binits{L.}}
(\byear{1968}).
\btitle{A proof of the optimality of the shortest remaining processing time
  discipline}.
\bjournal{Operations Research}
\bvolume{16}
\bpages{687--690}.
\end{barticle}
\endbibitem

\bibitem{S93}
\begin{barticle}[author]
\bauthor{\bsnm{Schreiber},~\bfnm{F.}\binits{F.}}
(\byear{1993}).
\btitle{Properties and applications of the optimal queueing strategy
  {S}{R}{P}{T}: {A} survey}.
\bjournal{Archiv f\"{u}r Elektronik und \"{U}bertragungstechnik}
\bvolume{47}
\bpages{372--378}.
\end{barticle}
\endbibitem

\bibitem{SG98}
\begin{bbook}[author]
\bauthor{\bsnm{Silberschatz},~\bfnm{A.}\binits{A.}} \AND
  \bauthor{\bsnm{Galvin},~\bfnm{P.}\binits{P.}}
(\byear{1998}).
\btitle{Operating System Concepts, 5th Edition}.
\bpublisher{John Wiley \& Sons}.
\end{bbook}
\endbibitem

\bibitem{S78}
\begin{barticle}[author]
\bauthor{\bsnm{Smith},~\bfnm{D.~R.}\binits{D.~R.}}
(\byear{1978}).
\btitle{A new proof of the optimality of the shortest remaining processing time
  discipline}.
\bjournal{Operations Research}
\bvolume{26}
\bpages{197--199}.
\end{barticle}
\endbibitem

\bibitem{S92}
\begin{bbook}[author]
\bauthor{\bsnm{Stallings},~\bfnm{W.}\binits{W.}}
(\byear{1995}).
\btitle{Operating Systems,2nd Edition}.
\bpublisher{Prentice Hall}.
\end{bbook}
\endbibitem

\bibitem{T92}
\begin{bbook}[author]
\bauthor{\bsnm{Tanenbaum},~\bfnm{A.~S.}\binits{A.~S.}}
(\byear{1992}).
\btitle{Modern Operating Systems}.
\bpublisher{Prentice Hall}.
\end{bbook}
\endbibitem

\bibitem{W71}
\begin{barticle}[author]
\bauthor{\bsnm{Whitt},~\bfnm{W.}\binits{W.}}
(\byear{1971}).
\btitle{Weak convergence theorems for priority queues: Preemptive-resume
  discipline}.
\bjournal{Journal of Applied Probability}
\bvolume{8}
\bpages{74--94}.
\end{barticle}
\endbibitem

\bibitem{W01}
\begin{bbook}[author]
\bauthor{\bsnm{Whitt},~\bfnm{W.}\binits{W.}}
(\byear{2001}).
\btitle{Stochastic Process Limits: An Introduction to Stochastic Process Limits
  and their Applications}.
\bpublisher{Springer Series in Operations Research}.
\end{bbook}
\endbibitem

\bibitem{WHB03}
\begin{binproceedings}[author]
\bauthor{\bsnm{Wierman},~\bfnm{A.}\binits{A.}} \AND
  \bauthor{\bsnm{Harchol-Balter},~\bfnm{M.}\binits{M.}}
(\byear{2003}).
\btitle{Classifying scheduling policies with respect to unfairness in an
  {M}/{G}{I}/1}.
In \bbooktitle{ACM SIGMETRICS 2003 Conference on Measurement and Modeling of
  Computer Systems}
\bpages{238--249}.
\end{binproceedings}
\endbibitem

\end{thebibliography}
\end{document}